\pdfoutput=1
\documentclass[12pt,a4paper]{article}
\usepackage[utf8]{inputenc}
\usepackage[english]{babel}
\usepackage{amsmath}
\usepackage{amsthm}
\usepackage{enumitem}
\usepackage{xcolor}
\usepackage{tikz}
\usetikzlibrary{matrix}
\usetikzlibrary{cd}
\usepackage{amsfonts}
\usepackage{amssymb}
\usepackage{makeidx}
\usepackage{graphicx}
\usepackage{hyperref}
\usepackage{nameref}
\usepackage{cancel}
\usepackage{bookmark}
\usepackage{microtype}
\usepackage{caption, booktabs}
\usepackage{authblk}
\usepackage[left=2cm,right=2cm,top=2cm,bottom=2cm]{geometry}
\author{Josua Schott\thanks{The author was supported by SNF-Grant $200021\_ 207335$.\\ E-Mail: josua.schott@unibe.ch}  }
\affil{Departement Mathematik, Universität Bern, Sidlerstrasse 5, CH--3012 Bern, Switzerland}
\title{Holomorphic Factorization of Mappings into the Symplectic Group}
\date{June 23, 2022}

\newcommand{\Cm}{\mathbb{C}^{\frac{n(n+1)}{2}}}

\newcommand{\C}[1]{\mathbb{C}^{#1}}

\newcommand{\SpC}{\mathrm{Sp}_{2n}(\mathbb{C})}

\newcommand{\Eu}[1]{\begin{pmatrix}
I_n&#1\\0&I_n
\end{pmatrix}}
\newcommand{\El}[1]{\begin{pmatrix}
I_n&0\\#1&I_n
\end{pmatrix}}

\newcommand{\CwO}{\mathbb{C}^{2n}\setminus \{0\}}
\newcommand{\kfiber}{\mathcal{F}^K_{a,b}}

\newcommand{\fiber}{\mathcal{F}^K_y}

\newcommand{\complete}{\mathcal{VC}_{hol}(M)}
\newcommand{\deriv}[1]{\frac{\partial}{\partial {#1}}}
\newcommand{\Deriv}[2]{\deriv{{#1}_{#2}}}
\newcommand{\spanning}{\nameref{theorem:spanTangentSpace}}

\newcommand{\Oset}[2]{\mathcal{N}^{#1}(#2)}

\newcommand{\Wk}[1]{\mathcal{W}_{#1}}
\newcommand{\threefiber}{\mathcal{F}^3_y}

\newcommand{\Qk}[1]{\mathcal{Q}_{#1}}
\newcommand{\rel}{\sim_x}
\newcommand{\PK}{\tilde{P}^K}
\newcommand{\PKM}{\tilde{P}^{K-1}}
\newcommand{\PKMM}{\tilde{P}^{K-2}}
\newcommand{\PKMMM}{\tilde{P}^{K-3}}

\makeatletter
\newcommand{\subjclass}[2][1991]{%
  \let\@oldtitle\@title%
  \gdef\@title{\@oldtitle\footnotetext{#1 \emph{Mathematics subject classification.} #2}}%
}
\newcommand{\keywords}[1]{%
  \let\@@oldtitle\@title%
  \gdef\@title{\@@oldtitle\footnotetext{\emph{Key words and phrases.} #1.}}%
}
\makeatother

\newtheorem{theorem}{Theorem}[section]
\newtheorem{lemma}[theorem]{Lemma}
\newtheorem{cor}[theorem]{Corollary}
\newtheorem{prop}[theorem]{Proposition}
\newtheorem{defi}[theorem]{Definition}
\newtheorem{Rem}[theorem]{Remark}
\newtheorem{Exa}[theorem]{Example}

\begin{document}
%
%
\maketitle
\begin{abstract}
It is shown that any symplectic $2n\times 2n$-matrix, whose entries are complex holomorphic functions on a reduced Stein space, can be decomposed into a finite product of elementary symplectic matrices if and only if it is null-homotopic. Moreover, if this is the case, the number of factors can be bounded by a constant depending only on $n$ and the dimension of the space. 
\end{abstract}
\tableofcontents
%
%
\section{Introduction}
It is well known that over any field, in particular over the field of complex numbers, a matrix in the Special Linear Group $\mathrm{SL}_n(\mathbb{C})$ is a product of elementary matrices. The proof is usually an  application of a Gauss or a Gauss-Jordan process. The same question for $\mathrm{SL}_n(R)$ over a commutative ring $R$ is much more difficult and has been studied a lot. For the ring $R=\mathbb{C}[z]$ of polynomials in one variable, it is true since $R$ is an Euclidean ring. For $n=2$ and the ring $R=\mathbb{C}[z,w]$, Cohn \cite{Cohn:1966tz} found the following counterexample: the matrix
$$ \begin{pmatrix}
1 - zw & z^2\\ -w^2& 1+zw
\end{pmatrix}\in \mathrm{SL}_2(R)$$
does not decompose as a finite product of unipotent matrices. However, it turns out that the solvability of the factorization problem with respect to a polynomial ring $\mathbb{C}[z_1,...,z_m]$ in several variables depends on the matrix size $n$. In fact, Suslin \cite{Suslin:1977us} gave a positive answer in the case $\mathrm{SL}_n(\mathbb{C}[z_1,...,z_m])$, for $n\geq 3$ and $m\geq 1$ in the 1970s.
For the ring $R$ of complex-valued continuous functions on a normal topological space,  Vaserstein \cite{Vaserstein:1988td} showed, that a matrix in $\mathrm{SL}_n(R)$ decomposes into a product of elementary matrices if and only it is null-homotopic. In Gromov's seminal paper on the Oka principle, the starting point of modern Oka theory, he asks this question for the ring of complex-valued holomorphic functions (he calls it the \textit{Vaserstein problem}, see \cite[p. 886]{Gromov:1989aa}). In 2012, Ivarsson and Kutzschebauch \cite{Ivarsson:2012aa} were able to give a positive answer to this problem in full generality. 

The same question for the symplectic group $\mathrm{Sp}_{2n}(R)$ hasn't been studied to the same degree. Again, there is a positive answer for the ring $\mathbb{C}[z]$, since this is an Euclidean ring. For $n\geq 2$, Kopeiko \cite{Kopeiko:1978ub} proved it for the polynomial ring $R=k[z_1,...,z_m]$; and Grunewald, Mennicke and Vaserstein \cite{Grunewald:1991ww} for the ring $R=\mathbb{Z}[z_1,...,z_m]$. In \cite{Bjorn-Ivarson:2020aa}, Ivarsson, Kutzschebauch and Løw prove it for every commutative ring $R$ with identity and bass stable rank $sr(R)=1$. Furthermore, they show that, for the ring $R=C(X)$ of complex-valued continuous functions on a normal topological space, a matrix in $\mathrm{Sp}_{2n}(C(X))$ can be decomposed into a product of elementary matrices if and only it is null-homotopic; they call it the \textit{Continuous Vaserstein problem for symplectic matrices}. The same authors \cite{Ivarson:2020aa} also  tackle what they call the \textit{Holomorphic Vaserstein problem for symplectic matrices}, that is, they ask for a similar result in the case of $\mathrm{Sp}_{2n}(\mathcal{O}(X))$, where $\mathcal{O}(X)$ denotes the ring of holomorphic functions on a reduced Stein space $X$. In fact, they are able to give a positive answer for $\mathrm{Sp}_4(\mathcal{O}(X))$. However, their proof does only apply for $n=2$.

In this article, we'll see another strategy to handle this factorization problem and even more, we give a solution of the Holomorphic Vaserstein problem for symplectic matrices in full generality. In order to formulate our result let us introduce some terms. 
We call a $(2n\times 2n)$-matrix $M$ \textit{symplectic} if it satisfies $M^T\Omega M=\Omega$ with respect to the skew-symmetric matrix
$$ \Omega=\begin{pmatrix}
0&I_n\\ -I_n & 0
\end{pmatrix},$$
where $I_n$ denotes the $n\times n$-identity matrix and $0$ the $n\times n$-zero matrix. For symmetric matrices $A\in \C{n\times n}$, i.e. $A^T=A$, matrices of the form
$$ \Eu{A} \quad \text{and} \quad \El{A}$$ are symplectic and 
we call them \textit{elementary symplectic matrices}. For simplicity reasons, let's identify the space of symmetric matrices $\mathrm{Sym}_n(\mathbb{C})=\{A\in\C{n\times n}: A^T=A\}$ with $\Cm$. 

\begin{theorem}[Holomorphic Vaserstein problem for symplectic matrices]\label{MainTheorem}
There exists a natural number $K=K(n,d)$ such that given any finite dimensional reduced Stein space $X$ of dimension $d$ and any null-homotopic holomorphic mapping $f:X\rightarrow \textnormal{Sp}_{2n}(\mathbb{C})$ there exist a holomorphic mapping $$G=(G_1,...,G_K):X\rightarrow (\Cm)^K$$
such that
$$f(x)=\El{G_1 (x) }\Eu{G_2 (x) }\cdots \El{G_{K-1}(x)}\Eu{G_K(x)}, \quad x\in X.$$
\end{theorem}
It is immediate that any product of elementary symplectic matrices is connected by a path to the constant identity matrix $I_{2n}$, by multiplying the off-diagonal entries by $t\in [0,1]$, i.e. 
$$ \Eu{tA} \quad \text{or} \quad \El{tA}.$$
Therefore the requirement of null-homotopy of the map $f$ is neccessary.

\subsection{Strategy of proof}
The proof relies heavily on an application of the Oka principle. Put very loosely, this application allows us to deform the continuous solution provided by the \textit{Continuous Vaserstein problem for symplectic matrices} (see Theorem \ref{t:mainthmtoprestate}) into a holomorphic solution. The most obvious strategy to check the sufficient conditions of the principle requires a deep understanding of complete holomorphic vector fields. And although their classification is already being studied in the case $n=2$ in \cite{Ivarson:2020aa}, the case of general $n$ is still very hard. 

To explain the idea of the proof in a little more detail, let's introduce the \textit{elementary symplectic matrix mapping} ${M_k:\Cm \to \SpC}$ by
$$ M_k(Z) = \begin{cases} \Eu{Z} & \text{if }k=2l\\
\El{Z} & \text{if } k=2l+1
\end{cases}$$
and then define the \textit{product mapping} $\Psi_K:(\Cm)^K\to \SpC$ by
$$ \Psi_K(Z_1,...,Z_K) = M_1(Z_1)M_2(Z_2)\cdots M_K(Z_K).$$

We now consider a holomorphic mapping $f:X\to \SpC$, where $X$ denotes a finite-dimensional Stein space of dimension $d$. According to the Continuous Vaserstein problem for symplectic matrices, there exists a natural number $K=K(n,d)$ and a continuous lift $F:X\to (\Cm)^K$ which leads to the following commuting diagram:
\begin{center}
\begin{tikzcd}
 & (\Cm)^K \arrow[d,"\Psi_K"] \\ X \arrow[ru,"F"] \arrow[r,"f" '] & \SpC.
\end{tikzcd}
\end{center}

The analysis of the fibers of $\Psi_K$ is too difficult, so let's first consider $\Phi_K:=\pi_{2n}\circ \Psi_K$, where $\pi_{2n}$ denotes the projection of a $(2n\times 2n)$-matrix to its last row. We obtain the commutative diagram
\begin{center}
\begin{tikzcd}
 & (\Cm)^K \arrow[d,"\Phi_K=\pi_{2n}\circ \Psi_K"] \\ X \arrow[ru,"F"] \arrow[r,"\pi_{2n}\circ f" '] & \CwO.
\end{tikzcd}
\end{center}
The mapping $\Phi_K$ is surjective for $K\geq 3$ (see Theorem \ref{theorem:PhiSurjective}) and it is submersive outside some set of singularities $S_K\subset (\Cm)^K$ (see Theorem \ref{theorem:singularitySet}). We will find an open submanifold $E_K \subset (\Cm)^K\setminus S_K$ such that
\begin{enumerate}[label=(\roman*)]
\item $\Phi_K|_{E_K}:E_K\to \CwO$ is a stratified elliptic submersion (see section 3),
\item the image of $F:X\to (\Cm)^K$ is contained in $E_K$, i.e. $\mathrm{Im}(F)\subset E_K$.
\end{enumerate}
In other words, we will find $E_K$ such that
\begin{center}
\begin{tikzcd}
 & E_K \arrow[d,"\Phi_K"] \\ Y \arrow[ru,"F"] \arrow[r,"\pi_{2n}\circ f" '] & \CwO
\end{tikzcd}
\end{center}
commutes (see Theorem \ref{theorem:holomorphicSection}) and $\Phi_K$ satisfies the conditions to apply the Oka principle. Thus it is possible to homotopically deform $F$ into a holomorphic mapping, such that the above diagram commutes \textit{holomorphically}.

Now we still need to recover the information lost through the projection $\pi_{2n}$, which can be done by induction on the size of the matrices $2n$ and a $K$-theoretic argument (see section \ref{subsection:mainproof}). 

The organisation of this paper is as follows. In section 2 we focus on the factorization problem. First we introduce the most important terms regarding the Oka principle (see subsection \ref{section:Okaprinciple}). In subsection \ref{subsection:holomorphicVaserstein} we apply the Oka principle in more detail on the one hand and we see a trick on the other hand how we can control the abstract continuous lift $F:X\to (\Cm)^K$ in an appropriate way (cf. Theorem \ref{theorem:holomorphicSection}). And finally, in subsection \ref{subsection:mainproof}, we inductively prove the factorization result.

Section 3 is dedicated to the existence of the manifold $E_K$ and the construction of \textit{stratified sprays}. In subsections 3.1 and 3.2, basic properties of the mapping $\Phi_K$ are analyzed. In subsection 3.3 we introduce a family of fiber-preserving holomorphic vector fields. Unfortunately, not all vector fields in this family are complete; the (partial) classification of the complete vector fields is the subject of subsection 3.4. Theorem \ref{lemma:classifyCompleteFields} provides a useful sufficient condition for deciding whether a given vector field is complete. In subsection 3.6, a strategy is developed to enlarge the collection of complete holomorphic vector fields to such an extend that stratified sprays can be constructed. And finally, in subsection 3.7, all the details are checked. 

\subsection{Acknowledgements}
First of all, I would like to thank my supervisor Frank Kutzschebauch for the opportunity to work on such an interesting problem and for the many exciting and supportive discussions. 
I am also very thankful to Dr. George Ionita for proofreading and his helpful feedback. Last but not least, I would like to thank my family, especially my wife Anna, who always supported me and believed in me. 

%
%
\section{Holomorphic Vaserstein problem for symplectic matrices}
Let's start with the so-called \textit{Continuous Vaserstein problem} for symplectic matrices. This has been proved by Ivarsson, Kutzschebauch and Løw (\cite[Theorem 1.3]{Bjorn-Ivarson:2020aa}) and is one of the key ingredients. 
\begin{theorem}[Continuous Vaserstein problem for symplectic matrices]\label{t:mainthmtoprestate}
There exists a natural number $K(n,d)$ such that given any  finite dimensional normal topological space $X$ of (covering) dimension $d$ and any null-homotopic continuous mapping $M\colon X\to \operatorname{Sp}_{2n}(\mathbb{C})$ there exist $K$ continuous mappings \[G_1,\dots, G_{K}\colon X\to \mathbb{C}^{n(n+1)/2}\] such that \[M(x)=M_{1}(G_1(x))\dots M_{K}(G_{K}(x)).\]
\end{theorem}

\subsection{Oka principle}\label{section:Okaprinciple}
The Oka principle is a powerful tool. Roughly speaking, it allows us to homotopically deform a continuous mapping into a holomorphic one in certain situations. We want to make this more precise in this section and therefore introduce some definitions and terminologies.

Let's start with the notion of a \textit{stratified elliptic submersion} $h:Z\to X$ from a complex space $Z$ onto a complex space $X$, following  \cite{Gromov:1989aa} and \cite{Forstneric:2010aa}.

\begin{defi}
Let $X$ and $Z$ be complex spaces. A holomorphic map $h\colon Z \to X$ is said to be a submersion if for each point $z_0\in Z$ it is locally equivalent via a fiber-preserving biholomorphic map to a projection $p\colon U\times V \to U$, where $U\subset X$ is an open set containing $h(z_0)$ and $V$ is an open set in some $\mathbb{C}^d$.  
\end{defi}

Let $h\colon Z \to X$ be a holomorphic submersion of a complex manifold $Z$ onto a complex manifold $X$. For any $x\in X$ the fiber over $x$ of this submersion will be denoted by $Z_x$. At each point $z\in Z$ the tangent space $T_zZ$ contains {\it the vertical tangent space} $VT_zZ=\ker Dh$. For holomorphic vector bundles $p\colon E \to Z$ we denote the zero element in the fiber $E_z$ by $0_z$.

\begin{defi} \label{definspray}
Let $h\colon Z \to X$ be a holomorphic submersion of a complex manifold $Z$ onto a complex manifold $X$. A spray on $Z$ associated with $h$ is a triple $(E,p,s)$, where $p\colon E\to Z$ is a holomorphic vector bundle and $s\colon E\to Z$ is a holomorphic map such that for each $z\in Z$ we have 
\begin{itemize}
\item[(i)]{$s(E_z)\subset Z_{h(z)}$,}
\item[(ii)]{$s(0_z)=z$, and}
\item[(iii)]{the derivative $Ds(0_z)\colon T_{0_z}E\to T_zZ$ maps the subspace $E_z\subset T_{0_z}E$ surjectively onto the vertical tangent space $VT_zZ$.}
\end{itemize} 
\end{defi}

\begin{Rem}
We will also say that the submersion admits a spray. A spray associated with a holomorphic submersion is sometimes called a (fiber) dominating spray.
\end{Rem}

One way of constructing dominating sprays, as pointed out by \textsc{Gromov}, is to find finitely many $\mathbb{C}$-complete vector fields that are tangent to the fibers and span the tangent space of the fibers at all points in $Z$. One can then use the flows $\varphi_j^t$ of these vector fields $V_j$ to define $s\colon Z\times \mathbb{C}^N\to Z$ via $s(z,t_1,\dots, t_N)=\varphi_1^{t_1}\circ \dots \circ \varphi_N^{t_N}(z)$ which gives a spray. 

\begin{defi} 
\label{definstratifiedspray}
We say that a submersion $h\colon Z\to X$ is \textit{stratified elliptic} if there is a descending chain of closed complex subspaces $X=X_m\supset \cdots \supset X_0$ such that each stratum $Y_k = X_k\setminus X_{k-1}$ is regular and the restricted submersion $h\colon Z|_{Y_k}\to Y_k$ admits a spray over a small neighborhood of any point $x\in Y_k$.  
\end{defi}

\begin{Rem}
We say that the submersion admits stratified sprays and that the stratification $X=X_m\supset \cdots \supset X_0$ is associated with the stratified spray.
\end{Rem}

Let's consider the following diagram
\begin{center}
\begin{tikzcd}
P_0\times X \arrow[d,"\mathrm{incl}" '] \arrow[r,"F"]  & E \arrow[d,"\pi"]\\ P\times X \arrow[ru, "F"] \arrow[r, "f"] & B
\end{tikzcd}
\end{center}
Here $\pi:E\to B$ is a holomorphic submersion of a complex space $E$ onto a complex space $B$, $X$ is a Stein space, $P_0\subset P$ are compact Hausdorff spaces (the parameter spaces), and $f:P\times X\to B$ is an $X$-holomorphic map, meaning that $f(p,\cdot):X\to B$ is holomorphic on $X$ for every fixed $p\in P$. A map $F:P\times X\to E$ such that $\pi\circ F = f$ is said to be a \textit{lifting} of $f$; such $F$ is $X$-holomorphic on $P_0$ if $F(p,\cdot)$ is holomorphic for every $p\in P_0$. 

\begin{defi}
A holomorphic map $\pi:E\to B$ between reduced complex spaces enjoys the Parametric Oka Property (POP) if for any collection $(X,X',K,P,P_0,f,F_0)$ where $X$ is a reduced Stein space, $X'$ is a closed complex subvariety of $X$, $P_0\subset P$ are compact Hausdorff spaces, $f:P\times X\to B$ is an $X$-holomorphic map, and $F_0:P\times X\to E$ is a continuous map such that $\pi\circ F=f$, the map $F_0(p,\cdot)$ is holomorphic on $X$ for all $p\in P_0$ and is holomorphic on $K\cup X'$ for all $p\in P$, there exists a homotopy $F_t:P\times X\to E$ such that the following hold for all $t\in [0,1]$:
\begin{enumerate}[label=(\roman*)]
\item $\pi\circ F_t = f$,
\item $F_t=F_0$ on $(P_0\times X)\cup (P\times X')$,
\item $F_t$ is $X$-holomorphic on $K$ and uniformly close to $F_0$ on $P\times K$, and
\item the map $F_1:P\times X\to E$ is $X$-holomorphic.
\end{enumerate}
\end{defi}

An important cornerstone in the proof of holomorphic factorization is the following very general version of the Oka principle.
\begin{theorem}[Oka principle]\label{theorem:OkaPrincipal}
Every stratified elliptic submersion enjoys POP.
\end{theorem}

\begin{Rem}
Forstneric proved this result for Euclidean compact parameter spaces $P$ (see \cite[Corollary 6.14.4 (i)]{Forstneric:2011wk}) and Kusakabe generalized it to the case where the parameter space $P$ is a compact Hausdorff space (see \cite{Kusakabe:2021tv}).
\end{Rem}

\subsection{Holomorphic Vaserstein problem for symplectic matrices}\label{subsection:holomorphicVaserstein}
Recall the mapping $\Psi_K:(\Cm)^K\to \SpC$ given by $ \Psi_K(Z_1,...,Z_K)=M_1(Z_1)\cdots M_K(Z_K)$, where $M_k:\Cm\to \SpC$ is defined by
$$ M_k(Z) = \begin{cases} \Eu{Z} & \text{if }k=2l\\
\El{Z} & \text{if } k=2l+1.
\end{cases}$$
Further, recall the map $\Phi_K=\pi_{2n}\circ \Psi_K:(\Cm)^K\to \CwO$, where $\pi_{2n}:\C{2n\times 2n}\to \C{2n}$ denotes the projection of a $2n\times 2n$-matrix to its last row. 

Define the set
$$\Wk{K}:=\left\lbrace (Z_1,...,Z_K)\in (\Cm)^K: Z_{2i-1}e_n\neq 0 \text{ for some } 1\leq i\leq \lceil\tfrac{K-1}{2}\rceil \right\rbrace$$
and let $S_K\subset (\Cm)^K$ denote the set of points, where the mapping $\Phi_K$ is not submersive. 

The following is another cornerstone in the proof of the main theorem and we will prove it in the next section. 
\begin{theorem}\label{theorem:PhiHasPop}
For $K\geq 3$, there exists an open submanifold $E_K\subset (\Cm)^K$ satisfying
\begin{enumerate}[label=(\roman*)]
\item $\Wk{K}\subset E_K\subset (\Cm)^K\setminus S_K$
\item the restriction $\Phi_K|_{E_K}:E_K\to \CwO$ is a stratified elliptic submersion.
\end{enumerate}
\end{theorem}

The \nameref{theorem:OkaPrincipal} tells us, that $\Phi_K|_{E_K}$ has the Parametric Oka Property, i.e. we get the following
\begin{cor}[Application of the Oka principle]\label{corollary:ApplicationOkaPrinciple}
Let $P$ be a compact Hausdorff space, $X$ a finite dimensional reduced Stein space and $f:P\times X\to \SpC$ a null-homotopic, continuous $X$-holomorphic mapping. Assume there is a natural number $K$ and a continuous map $F:P\times X\to E_K$ such that
\begin{center}
\begin{tikzcd}
& E_K \arrow[d, "\Phi_K"] \\ P\times X \arrow[ru, "F"] \arrow[r, "\pi_{2n}\circ f"] & \CwO
\end{tikzcd}
\end{center}
is commutative. Then there exists a continuous homotopy $F_t:P\times X\to E_K$ with $F_0=F$, $\pi_{2n}\circ f = \Phi_K\circ F_t$ and such that $F_1:P\times X\to E_K$ is $X$-holomorphic.
\end{cor}

In a next step, we prove the existence of a natural number $K$ and a continuous lifting $F:P\times X\to E_K$ such that the diagram in the corollary commutes.
\begin{theorem}\label{theorem:holomorphicSection}
There exists a natural number $L(n,d)$ such that given any compact Hausdorff space $P$, any finite dimensional reduced Stein space $X$, such that $P\times X$ has covering dimension $d$, and any null-homotopic, continuous $X$-holomorphic mapping $f:P\times X\to \SpC$, there exists a continuous lifting $F:P\times X\to E_L$ of $\pi_{2n}\circ f$. In particular, there exists a continuous homotopy $F_t:P\times X\to E_L$ of liftings of $\pi_{2n}\circ f$, such that $F_0=F$ and $F_1$ is $X$-holomorphic.
\end{theorem}
\begin{proof}
The \nameref{t:mainthmtoprestate} provides us with a natural number $K=K(n,d)$ such that given any normal topological space $Y$ of dimension $d$ and any null-homotopic continuous mapping $f:Y\to \SpC$ there exists a continuous lifting $G=(G_1,...,G_K):Y\to (\Cm)^{K}$ with $f=\Psi_{K}\circ G$. For $L=K+2$, we define the mapping $F:Y\to (\Cm)^L$ by
$$ F = (I_n,0,G_1-I_n,G_2,...,G_K)$$
where $I_n$ denotes the $n\times n$-identity matrix and $0$ the $n\times n$-zero matrix. Observe that
$$ \El{G_1} = \El{I_n}\Eu{0}\El{G_1-I_n}$$
and therefore $F$ is a lifting of $f$, more precisely, $f=\Psi_L\circ F$. The image of $F$ is contained in $\Wk{L}$, since $I_ne_n\neq 0$. By Theorem \ref{theorem:PhiHasPop}, we have $\Wk{L}\subset E_L$ which means that $F$ is a mapping $F:Y\to E_L$ and we obtain a commuting diagram
\begin{center}
\begin{tikzcd}
& E_L \arrow[d, "\Psi_L"] \arrow[rd, "\Phi_L"] & \\ Y \arrow[ru, "F"] \arrow[r, "f"] & \SpC \arrow[r,"\pi_{2n}"] & \CwO.
\end{tikzcd}
\end{center}
Choose $Y=P\times X$, where $P$ is a compact Hausdorff space and $X$ a reduced Stein space such that $Y$ has dimension $d$. An application of Corollary \ref{corollary:ApplicationOkaPrinciple} completes the proof.
\end{proof}

Before we come to the proof of the main theorem, we need one more result also known as \textit{Whitehead's lemma}. For some commutative ring $R$ with identity, the symplectic group $\mathrm{Sp}_{2n}(R)$ is a subgroup of $\mathrm{SL}_{2n}(R)$. We shall write matrices with block notation
$$\begin{pmatrix}
A&B\\C&D
\end{pmatrix}$$
where $A,B,C$ and $D$ are $n\times n$ matrices with entries in $R$ satisfying the symplectic conditions
\begin{align}
A^TC = C^TA\\
B^TD = D^TB\\
A^TD-C^TB = I_n
\end{align}
where $I_n$ is the $n\times n$ identity matrix.

An \textit{elementary symplectic matrix} is either of the form
$$ \begin{pmatrix}
I_n & B \\ 0 & I_n
\end{pmatrix}$$
where $B$ is symmetric ($B^T=B$) or of the form 
$$ \El{C}$$
where $C$ is symmetric. Products of matrices of the first type are additive in $B$ and of the second type in $C$. Special cases are the matrices $E_{ij}(a)$ when $B$ is the matrix with $a$ in position $ij$ and $ji$ and otherwise zero. For $F_{ij}(a)$ the roles of $B$ and $C$ are changed. Clearly any elementary matrix of the first type is a product of matrices $E_{ij}(b_{ij})$ for $i\leq j$ and similarly for the second type.

We also introduce the symplectic matrices $K_{ij}(a)$ defined by $B=C=0$ and $A=I_n$ except in position $ij$ where there is an $a$. Finally, $D=(A^T)^{-1}$. This equals $I$ except in position $ji$ where there is $-a$ if $i\neq j$ and $a^{-1}$ if $i=j$ (this requires $a\in R^*$).

\begin{lemma}[Whitehead's lemma]\label{lemma:whitehead}
There holds
$$ K_{ii}(a) = E_{ii}(a-1)F_{ii}(1)E_{ii}(a^{-1}-1)F_{ii}(-a),\quad a\in R^*,$$
and if $i\neq j$, then
$$ K_{ij}(a) = F_{jj}(-a)E_{ij}(1)F_{jj}(a)E_{ii}(a)E_{ij}(-1), \quad a\in R.$$
\end{lemma}

\subsection{Proof of the \nameref{MainTheorem}}\label{subsection:mainproof}
The proof is by induction on $n$. Note that the theorem is true for $n=1$, since $\mathrm{Sp}_2(\mathbb{C}) = \mathrm{SL}_2(\mathbb{C})$ (see \cite{Ivarsson:2012aa}), that is, the base case is fine. 

For the induction step, we first observe that, according to Theorem \ref{theorem:holomorphicSection}, there exists a natural number $L(n,d)$ such that given any compact Hausdorff space $P$, any finite dimensional Stein space $X$, such that $P\times X$ has covering dimension $d$, and any null-homotopic $X$-holomorphic mapping $f:P\times X\to \SpC$, there exists a continuous homotopy $F_t:P\times X\to E_L$ of liftings of $\pi_{2n}\circ f$ (i.e. $\pi_{2n}\circ f = \Phi_L\circ F_t$ for all $0\leq t\leq 1$), such that $F_1$ is $X$-holomorphic. In particular, this implies that $\Phi_L\circ F_t$ doesn't depend on $t$, hence we get
$$ \Psi_L(F_t(p,x))f(p,x)^{-1} = \begin{pmatrix} A_{1,t}(p,x) & a_{2,t}(p,x) & B_{1,t}(p,x) & b_{2,t}(p,x) \\
a_{3,t}(p,x) & a_{4,t}(p,x) & b_{3,t}(p,x) & b_{4,t}(p,x)\\
C_{1,t}(p,x) & c_{2,t}(p,x) & D_{1,t}(p,x) & d_{1,t}(p,x)\\
0 & 0 & 0 & 1
\end{pmatrix},$$
where $A_{1,t}(p,x), B_{1,t}(p,x), C_{1,t}(p,x)$ and $D_{1,t}(p,x)$ are $(n-1)\times (n-1)$ matrices, and the remaining mappings are vectors of appropriate dimension. 
Since $\Psi_L(F_t(p,x))f(p,x)^{-1}$ are symplectic matrices for all $0\leq t \leq 1$, there follows $a_{2,t}(p,x)\equiv 0, a_{4,t}(p,x)\equiv 1$ and $c_{2,t}(p,x)\equiv 0$, so that
$$ \Psi_L(F_t(p,x))f(p,x)^{-1} = \begin{pmatrix} A_{1,t}(p,x) & 0 & B_{1,t}(p,x) & b_{2,t}(p,x) \\
a_{3,t}(p,x) & 1 & b_{3,t}(p,x) & b_{4,t}(p,x)\\
C_{1,t}(p,x) & 0 & D_{1,t}(p,x) & d_{1,t}(p,x)\\
0 & 0 & 0 & 1
\end{pmatrix}$$
and in addition
$$ \tilde{f}_t(p,x) = \begin{pmatrix}
A_{1,t}(p,x) & B_{1,t}(p,x) \\ C_{1,t}(p,x) & D_{1,t}(p,x)
\end{pmatrix} \in \mathrm{Sp}_{2n-2}(\mathbb{C}).$$
Since $\Psi_L(F_0(p,x))=f(p,x)$, there holds $\tilde{f}_0 = I_{2n-2}$ and thus the holomorphic map\newline ${\tilde{f}:=\tilde{f}_1:P\times X\to \mathrm{Sp}_{2n-2}(\mathbb{C})}$ is null-homotopic. Let $\psi$ be the standard inclusion of $\mathrm{Sp}_{2n-2}$ into $\mathrm{Sp}_{2n}$ given by
$$ \psi\begin{pmatrix}
A & B \\ C & D
\end{pmatrix} = \begin{pmatrix}
A & 0 & B & 0 \\ 0 & 1 & 0 & 0 \\ C & 0 & D & 0 \\ 0 & 0 & 0 & 1
\end{pmatrix}.$$
 By the induction hypothesis,
$$\psi(\tilde{f}(p,x)^{-1}) = \begin{pmatrix}
\tilde{A}(p,x) & 0 & \tilde{B}(p,x) & 0 \\
0 & 1& 0& 0 \\
\tilde{C}(p,x) & 0 & \tilde{D}(p,x) & 0 \\
0 & 0 & 0 & 1
\end{pmatrix}$$
is a finite product of holomorphic elementary symplectic matrices. Then the matrix \newline ${M(p,x):=\Psi_L(F_1(p,x))f(p,x)^{-1} \psi(\tilde{f}(p,x)^{-1})}$ is given by
\begin{align*}
 M(p,x)= \begin{pmatrix}
I_{n-1} & 0 & 0 & b_{2,1}(p,x) \\
a_{3,1}(p,x)\tilde{A}(p,x)+ b_{3,1}(p,x)\tilde{C}(p,x) & 1 & a_{3,1}(p,x)\tilde{B}(p,x) + b_{3,1}(p,x)\tilde{D}(p,x) & b_{4,1}(p,x)\\
0 & 0 & I_{n-1} & d_{2,1}(p,x) \\
0 & 0 & 0 & 1
\end{pmatrix}.
\end{align*}
Since $M(p,x)$ is a symplectic matrix, there holds
$$ a_{3,1}(p,x)\tilde{A}(p,x)+ b_{3,1}(p,x)\tilde{C}(p,x) = -d_{2,1}(p,x)^T$$
and
$$ a_{3,1}(p,x)\tilde{B}(p,x) + b_{3,1}(p,x)\tilde{D}(p,x) = b_{2,1}(p,x)^T,$$
so that
$$ \Psi_L(F_1(p,x))f(p,x)^{-1} \psi(\tilde{f}(p,x)^{-1}) = \begin{pmatrix}
I_{n-1} & 0 & 0 & b_{2,1}(p,x) \\
-d_{2,1}(p,x)^T & 1 & b_{2,1}(p,x)^T & b_{4,1}(p,x)\\
0 & 0 & I_{n-1} & d_{2,1}(p,x) \\
0 & 0 & 0 & 1
\end{pmatrix}.$$ Then
$$  \Psi_L(F_1(p,x))f(p,x)^{-1} \psi(\tilde{f}(p,x)^{-1}) \Eu{-B(p,x)} = \begin{pmatrix}
I_{n-1} & 0 & 0 & 0 \\ -d_{2,1}(p,x)^T & 1 & 0 & 0 \\ 0 & 0 & I_{n-1} & d_{2,1}(p,x) \\ 0 & 0 & 0 & 1
\end{pmatrix},$$
where $$B(p,x) = \begin{pmatrix}
0 & -b_{2,1}(p,x) \\ -b_{2,1}(p,x)^T & -b_{4,1}(p,x)-d_{2,1}(p,x)^Tb_{2,1}(p,x)
\end{pmatrix}$$ is holomorphic. By making the substitution $y:=-d_{2,1}$, the matrix on the right hand side equals the product 
$$ K_{n1}(y_1(p,x))K_{n2}(y_2(p,x))\cdots K_{n,n-1}(y_{n-1}(p,x))$$
which is a finite product of elementary symplectic matrices due to \nameref{lemma:whitehead}. In summary, this proves that $f(p,x)$ is indeed a finite product of elementary symplectic matrices.

So far, the number of factors depends on the mapping $f:P\times X\to \SpC$. Assume there is no uniform bound $K(n,d)$, that is, for each natural number $i$, there is an Euclidean compact $P_i$ and a reduced Stein space $X_i$, such that $Y_i=P_i\times X_i$ has dimension $d$, and a null-homotopic $X$-holomorphic mapping $f_i:Y_i\to \SpC$ which does not factor into less than $i$ elementary matrix factors. Set $Y=\bigcup_{i} Y_i$ and let $F:Y\to \SpC$ be the null-homotopic mapping, which equals $f_i$ on $Y_i$. We just proved the existence of a constant $K$ which bounds the number of elementary factors in which $F$ decomposes. But then $K$ is an upper bound for each $f_i$ which contradicts the assumption. Hence there is a uniform bound $K(n,d)$ of factors and this complectes the proof.

Actually, we've shown a generalized version of the \nameref{MainTheorem}.
\begin{theorem}[Generalized version of main theorem]
There is a natural number $K=K(n,d)$ such that given any compact Hausdorff space $P$, any finite dimensional reduced Stein space $X$, such that $P\times X$ has covering dimension $d$, and any null-homotopic $X$-holomorphic mapping $f:P\times X\to \SpC$ there exist $K$ $X$-holomorphic mappings
$$ G_1,...,G_K:P\times X\to \Cm$$
with
$$ f(p,x) = M_1(G_1(p,x))M_2(G_2(p,x))\cdots M_K(G_K(p,x)).$$
\end{theorem}
%
%
\section{Stratified ellipticity of the mapping \texorpdfstring{$\Phi_K$}{}}
Recall that we identified $\mathrm{Sym}_n(\mathbb{C})=\{Z\in \C{n\times n}: Z^T=Z\}$ with $\Cm$ for simplicity. Furthermore, the \textit{elementary symplectic matrix mapping} ${M_K:\Cm \to \SpC}$ we defined by
$$ M_K(Z) = \begin{cases} \Eu{Z} & \text{if }K=2k+1\\
\El{Z} & \text{if } K=2k.
\end{cases}$$
Let's write $\vec{Z}_K:=(Z_1,...,Z_K) \in (\Cm)^K$. Then the mapping $\Phi_K:(\Cm)^K\to \CwO$ is defined by
$$ \Phi_K(\vec{Z}_K) := M_K(Z_K)\cdots M_1(Z_1)e_{2n}.$$
\begin{Rem}
The attentive reader may have noticed that we now define $\Phi_K$ in a transposed way compared to the previous sections. This is purely for aesthetic reasons.
\end{Rem}
The following recursive formula can be derived immediately from the definition.
\begin{cor}[Recursive formula of $\Phi_K$] For $K\geq 1$, the mapping $\Phi_K:(\Cm)^K\to \CwO$ satisfies
\begin{align}\label{RecursiveFormula}
\Phi_K(\vec{Z}_K)=M_K(Z_K)\Phi_{K-1}(\vec{Z}_{K-1}),
\end{align}
with the convention $\Phi_0:=e_{2n}$. 
\end{cor}
The main goal of this section is to prove
\begin{theorem} For $K\geq 3$, there is an open submanifold $E_K$ of $(\Cm)^K$ such that $$\Phi_K|_{E_K}:E_K\to \CwO$$ is a stratified elliptic submersion.
\end{theorem}

In subsection 3.1 we will classify the points for which $\Phi_K$ is not submersive. We will also show that $\Phi_K$ is surjective. 

The subsequent subsections are then devoted to the task of finding stratified sprays. In subsection 3.2, we stratify $\CwO$ suitably. In 3.3, we find formulas of holomorphic vector fields which are fiber-preserving for $\Phi_K$. Unfortunately, some of those fields aren't $\mathbb{C}$-complete. We therefore classify some complete vector fields in subsection 3.4. In subsection 3.5, we analyze the fibers of $\Phi_K$ from a topological point of view. In subsection 3.6 we lay the mathematical basis for the construction of the sprays. And finally we carry out all the necessary calculations in 3.7.
\subsection{Notations and basic properties}\label{subsection:notationsAndBasicProperties}
Let's start with some notations. Let $E_{ij}$ be the $n\times n$ matrix having a $1$ at entry $(i,j)$ and is zero elsewhere. Then $\tilde{E}_{ij}=\frac{1}{1+\delta_{ij}}(E_{ij}+E_{ji})$ is an \textit{elementary symmetric matrix}.

In the following we will make the identification of $\mathrm{Sym}_n(\mathbb{C})$ and $\Cm$ more precise. The set $\{\tilde{E}_{ij}:1\leq i \leq j \leq n\}$ forms a basis of $\mathrm{Sym}_n(\mathbb{C})$. The sets $ I:=\{i\in \mathbb{N}: 1\leq i \leq \frac{n(n+1)}{2}\}$ and $J:=\{(i,j)\in \mathbb{N}^2: 1\leq i\leq j\leq n\}$ have the same order. Hence there is a bijection $\alpha: I\to J$, which induces an isomorphism $S:\Cm\to \mathrm{Sym}_n(\mathbb{C})$ by defining $S(e_i)=\tilde{E}_{\alpha(i)}$ for all $1\leq i\leq \tfrac{n(n+1)}{2}$.
By an abuse of notation, $Z\in \Cm$ denotes both, the vector and the corresponding symmetric matrix $S(Z)$, depending on the corresponding context, of course. 

In this very first subsection, we will compute the set of points where $\Phi_K$ is \textit{not} submersive. Also we'll give a proof for the surjectivity of $\Phi_K$. But let's spell that out in more detail first.

We denote the set of points in $(\Cm)^K$ where $\Phi_K$ is not submersive by $S_K$.  We also define the open set
$$\Wk{K}:=\left\lbrace \vec{Z}_K\in (\Cm)^K: Z_{2i-1}e_n\neq 0 \text{ for some } 1\leq i\leq \lceil\tfrac{K-1}{2} \rceil \right\rbrace$$
and $\Wk{K}^c$ denotes its complement in $(\Cm)^K$. 
\begin{theorem}[Singularity set of $\Phi_K$]\label{theorem:singularitySet}
For $K\geq 2$, the set $S_K$ is given by
$$ S_K=\left\lbrace \vec{Z}_K\in \Wk{K}^c: \mathrm{rank}(W_K(\vec{Z}_K))<n\right\rbrace$$
where $W_K(\vec{Z}_K)$ is the augmented matrix $(Z_2|Z_4|\cdots | Z_{2k})$ for $k=\lfloor \tfrac{K-1}{2} \rfloor$.
\end{theorem}

\begin{theorem}\label{theorem:PhiSurjective}
For $K\geq 3$, the mapping $\Phi_K|_{\Wk{K}}:\Wk{K}\to \CwO$ is surjective.
\end{theorem}

A direct consequence of these two statements is
\begin{cor}
For $K\geq 3$ and for any open submanifold $E$ in $(\Cm)^K$ with $$\Wk{K}\subset E\subset (\Cm)^K\setminus S_K,$$ the mapping $\Phi_K|_{E}:E\to \CwO$ is a surjective submersion.
\end{cor}
\subsubsection{Singularity set of \texorpdfstring{$\Phi_K$}{Phi}}
In order to compute the singularity set $S_K$, we need to know the Jacobian of $\Phi_K$, denoted by $J\Phi_K$. For the computations, we need some auxiliary tools. For a fixed $1\leq i \leq n$, let's define the mapping $F_i:\C{n}\to \C{n\times n}$ by 
$$ F_i(v):=\left[ \tilde{E}_{i1}v \cdots \tilde{E}_{in}v\right] = \begin{pmatrix}
v_i & & & & \\ & \ddots &&& \\ v_1&\cdots& v_i &\cdots & v_n \\ &&&\ddots & \\ &&&&v_i
\end{pmatrix}.$$ Furthermore, we define $F:\C{n}\to \C{n\times \frac{n(n+1)}{2}}$ by
$$F(v) := \left[\tilde{E}_{\alpha(1)}v \cdots \tilde{E}_{\alpha(\frac{n(n+1)}{2})}v\right].$$ Observe that the matrix $F_i(v)$ is a submatrix of $F(v)$, for every $1\leq i\leq n$. This implies that $F(v)$ is surjective if and only if $v\neq 0$. And if $F(v)$ is not surjective, then we even have $F(v)=0$.

\begin{cor}\label{cor:JacobianElementaryMatrix}
For $u,v \in\C{n}$ such that $\begin{pmatrix}
u\\ v
\end{pmatrix}\in \CwO$, the Jacobian of the mapping ${\Cm\to \CwO},{ Z\mapsto M_K(Z)\begin{pmatrix}
u\\ v
\end{pmatrix}}$ is given by
$$ A_K(u,v):=\begin{cases}
\begin{pmatrix}
F(v)\\0
\end{pmatrix} &\text{if }K=2k+1\\
\begin{pmatrix}
0\\ F(u)
\end{pmatrix} & \text{if }K=2k.
\end{cases}$$
In particular, there holds $M_K(Z)A_K(u,v)=A_K(u,v)$ for all $Z\in \Cm$.
\end{cor}
\begin{proof}
Note that each symmetric matrix $Z\in \Cm$ can be written as a sum $Z=\sum_{1\leq i\leq j\leq n} z_{ij}\tilde{E}_{ij}$. Hence $\Deriv{z}{ij}Z = \tilde{E}_{ij}$. Further, note 
$$ \Deriv{z}{ij}M_{2k+1}(Z)\begin{pmatrix}
u\\v
\end{pmatrix} = 
\begin{pmatrix}
0& \tilde{E}_{ij} \\ 0 & 0 
\end{pmatrix} \begin{pmatrix}
u\\v
\end{pmatrix} = \begin{pmatrix}
\tilde{E}_{ij}v\\ 0
\end{pmatrix}$$
and
$$ \Deriv{z}{ij}M_{2k}(Z)\begin{pmatrix}
u\\v
\end{pmatrix} = 
\begin{pmatrix}0&0\\
 \tilde{E}_{ij}  & 0 
\end{pmatrix} \begin{pmatrix}
u\\v
\end{pmatrix} = \begin{pmatrix}0\\
\tilde{E}_{ij}u
\end{pmatrix},$$
respectively. From here, the claim follows by definition of the mapping $F$.
\end{proof}

We are now ready to compute the Jacobian of $\Phi_K$.  By the recursive formula (\ref{RecursiveFormula}), the product rule and the previous corollary, we obtain
\begin{cor}[Jacobian of $\Phi_K$]\label{cor:JacobianRecursive}
 The Jacobian $J\Phi_1$ is given by $A_1(e_{2n})$. For $K\geq 2$, the Jacobian $J\Phi_K$ of $\Phi_K$ is given by
$$ J\Phi_K(\vec{Z}_K) = \begin{pmatrix} M_K(Z_K)
J\Phi_{K-1}(\vec{Z}_{K-1}) &|& A_K(\Phi_{K-1}(\vec{Z}_{K-1}))
\end{pmatrix}.$$
\end{cor}
If the Jacobian $J\Phi_{K-1}$ is surjective, then so is $J\Phi_K$, since $M_K(Z_K)$ is a regular matrix. Or, equivalently, if $\vec{Z}_K=(\vec{Z}_{K-1},Z_K)\in S_K$ is a singular point for $\Phi_K$, then $\vec{Z}_{K-1}\in S_{K-1}$ is a singular point for $\Phi_{K-1}$. This observation suggests that we will compute $S_K$ recursively. 
\begin{lemma}\label{lemma:propertyOne}
Let $\vec{Z}_K\in (\Cm)^K$ and assume that there is $1\leq k\leq \big \lceil \frac{K-1}{2}\big\rceil$ such that ${Z_{2i-1}e_n=0}$ for all $1\leq i\leq k$. Then $\Phi_j(\vec{Z}_j)=e_{2n}$ for all $1\leq j \leq 2k$.
\end{lemma}
\begin{proof} We prove this by induction on $j$. For the base step, observe that
$$\Phi_1(\vec{Z}_1)=M_1(Z_1)e_{2n} = \begin{pmatrix}
Z_1e_n\\ e_n
\end{pmatrix}$$ by definition. Since we assume $Z_1e_n=0$, $\Phi_1(\vec{Z}_1)=e_{2n}$ follows immediately.

For the induction step, let $1<j\leq 2k$. There holds $\Phi_{j-1}(\vec{Z}_{j-1}) =e_{2n},$ by the induction hypothesis. Then we obtain $$\Phi_j(\vec{Z}_j) = M_j(Z_j)\Phi_{j-1}(\vec{Z}_{j-1}) = M_j(Z_j)e_{2n}$$ by the recursive formula (\ref{RecursiveFormula}). If $j$ is even, then we're done, by definition of $M_j$. Let's assume that $j=2l-1$ for some integer $l$. Observe that $l\leq k$ is satisfied, hence $Z_je_n=0$ by assumption. This implies $$\Phi_j(\vec{Z}_j)= \Eu{Z_j}\begin{pmatrix}
0\\ e_n 
\end{pmatrix} = \begin{pmatrix}
Z_je_n\\ e_n
\end{pmatrix} = e_{2n}.$$ This completes the proof.
\end{proof}
\begin{lemma}\label{lemma:phisASubmersion}
For $K\geq 2$, there holds $\Wk{K}\subset (\Cm)^K\setminus S_K$, that is, the restriction $\Phi_K|_{\Wk{K}}$ is a submersion.
\end{lemma}
\begin{proof}
Consider $\vec{Z}_K\in \mathcal{W}_K$. There is a smallest index $1\leq k\leq \big\lceil \frac{K-1}{2}\big\rceil$ such that $Z_{2k-1}e_n\neq 0$ and $Z_{2i-1}e_n=0$ for all $1\leq i<k$. Setting $L:=2k-1$, the Jacobian $J\Phi_L$ is of the form
$$ \begin{pmatrix}
* & F(P_s^{L-1}) \\ * & 0
\end{pmatrix}$$
and 
$$ J\Phi_{L+1} = \begin{pmatrix}
M_{L+1}(Z_{L+1})J\Phi_L & A_{L+1}(\Phi_L)
\end{pmatrix} = \begin{pmatrix}
* & F(P_s^{L-1}) & 0 \\ * & 0 & F(P_f^L)
\end{pmatrix}.$$
By definition of $F$, the Jacobian $J\Phi_{L+1}$ has full rank, if $P_s^{L-1}\neq 0$ and $P_f^L\neq 0$. By Lemma \ref{lemma:propertyOne}, there holds $\Phi_{L-1}=e_{2n}$. Hence
$$ P_s^{L-1} = \begin{pmatrix}
0 & I_n
\end{pmatrix} \Phi_{L-1} = e_n \neq 0.$$
Furthermore,
$$ P_f^L = \begin{pmatrix}
I_n & 0
\end{pmatrix} \Phi_L = \begin{pmatrix}
I_n & 0
\end{pmatrix} M_L(Z_L) \Phi_{L-1} = \begin{pmatrix}
I_n & Z_L
\end{pmatrix} e_{2n} = Z_Le_n \neq 0.$$
This showes that the Jacobian $J\Phi_{L+1}$ has full rank. By construction, there holds $L+1\leq K$. By the recursive formula of the Jacobian and regularity of $M_i(Z_i)$, $1\leq i\leq K$, we conclude that $J\Phi_K$ has full rank, too.
\end{proof}
Let's write $C_K:=\begin{pmatrix}
I_n & 0
\end{pmatrix} J\Phi_K$ and $D_K:=\begin{pmatrix}
0 & I_n
\end{pmatrix} J\Phi_K$. 
\begin{lemma}
For a point $\vec{Z}_{2k+2}=(\vec{Z}_{2k+1},Z_{2k+2})\in \mathcal{W}_{2k+2}^c$ the following statements are equivalent.
\begin{enumerate}[label=(\roman*)]
\item The Jacobian $J\Phi_{2k+2}(\vec{Z}_{2k+2})$ is surjective in $\vec{Z}_{2k+2}$.
\item The Jacobian $J\Phi_{2k+1}(\vec{Z}_{2k+1})$ is surjective in $\vec{Z}_{2k+1}$.
\item $\mathrm{rank}(D_{2k+1}) =\mathrm{rank}(D_{2k}) =n.$
\end{enumerate}
\end{lemma}
\begin{proof}
Note that $\mathcal{W}_K^c\subset \mathcal{W}_{K-1}^c\times \Cm$ by definition. Moreover, there holds
$$ A_{2k+1}(e_{2n}) = \begin{pmatrix}
F(e_n)\\ 0
\end{pmatrix} \quad \text{and} \quad A_{2k+2}(e_{2n})=0,$$
for a point $\vec{Z}_{2k+2}\in \mathcal{W}_{2k+2}^c$, by Lemma \ref{lemma:propertyOne} and definition of $F$. We conclude
$$ J\Phi_{2k+2}(\vec{Z}_{2k+2}) = M_{2k+2}(Z_{2k+2})\begin{pmatrix}
J\Phi_{2k+1}(\vec{Z}_{2k+1}) & 0
\end{pmatrix}$$
and this shows equivalence of $(i)$ and $(ii)$, since $M_{2k+2}(Z_{2k+2})$ is a regular matrix.

To show equivalence of $(ii)$ and $(iii)$ observe that
\begin{align*}
 J\Phi_{2k+1}(\vec{Z}_{2k+1}) = \begin{pmatrix}
C_{2k}+Z_{2k+1}D_{2k} & F(e_n) \\ D_{2k} & 0
\end{pmatrix}.
\end{align*}
Since $F(e_n)$ is surjective by definition of $F$, the claim follows immediately.
\end{proof}
In order to prove Theorem \ref{theorem:singularitySet}, it remains to show the next lemma.
\begin{lemma}\label{lemma:singularitySet}
For a point $\vec{Z}_{2k+2}\in \mathcal{W}_{2k+2}^c$ there holds
\begin{equation}
\mathrm{im}(D_{2k+2}) = \mathrm{im}(D_{2k}|Z_{2k+2}).\label{equation:augmentedMatrix} \tag{$\star$}
\end{equation}
In particular, the singularity set of $\Phi_K$, $K\geq 2$, is given by
$$ S_K:=\{\vec{Z}_K\in \mathcal{W}_K^c: \mathrm{rank}(W_K(\vec{Z}_K))<n\},$$
where $W_K$ is an augmented matrix $ W_K(\vec{Z}_K):=(Z_2|Z_4|\cdots |Z_{2k}),$ for $k=\lfloor \frac{K-1}{2} \rfloor$.
\end{lemma}
\begin{proof}
By the previous lemma, it is enough to show \eqref{equation:augmentedMatrix}. There holds
\begin{align*}
J\Phi_{2k+2}(\vec{Z}_{2k+2}) &= \begin{pmatrix}
C_{2k+1} & 0 \\ D_{2k+1} + Z_{2k+2}C_{2k+1} & 0
\end{pmatrix} = \begin{pmatrix}
C_{2k}+Z_{2k+1}D_{2k} & F(e_n) & 0 \\ D_{2k}+Z_{2k+2}(C_{2k}+Z_{2k+1}D_{2k}) & Z_{2k+2}F(e_n) & 0
\end{pmatrix}.
\end{align*}
Since $F(e_n)$ is surjective, we get $\mathrm{im}(Z_{2k+2})=\mathrm{im}(Z_{2k+2}F(e_n))$ and thus $\mathrm{im}(D_{2k+2}) = \mathrm{im}(D_{2k}|Z_{2k+2})$.
\end{proof}
\subsubsection{Surjectivity of \texorpdfstring{$\Phi_K$}{Phi} }
The proof of surjectivity is strongly based on this lemma.
\begin{lemma}\label{lemma:symmetricSurjective}
For $a\in \C{n}\setminus \{0\}$ fixed, the mapping $\varphi_a:\mathrm{Sym}_n(\mathbb{C})\to \C{n}, Z\mapsto Za$ is surjective.
\end{lemma}
\begin{proof}
Let $a\neq 0$ being fixed. Then the linear mapping $F(a): \Cm\to \C{n}$ is surjective, by definition of $F$. Therefore it is enough to show that the following diagram 
\begin{center}
\begin{tikzcd}
& \Cm \arrow[d,"F(a)"]\\
\mathrm{Sym}_n(\mathbb{C}) \arrow[ru, "S^{-1}"] \arrow[r,"\varphi_a" '] & \C{n}.
\end{tikzcd}
\end{center}
is commutative. This is the case if and only if $F(a)v= S(v)a$ for all $v\in \Cm$. Write $v=\sum_{i=1}^{\frac{n(n+1)}{2}} v_ie_i$. By definition of $F$ and $S$, we get
$$ F(a)v = \sum_{i=1}^{\frac{n(n+1)}{2}} v_i F(a)e_i = \sum_{i=1}^{\frac{n(n+1)}{2}} v_i \tilde{E}_{\alpha(i)}a = \left( \sum_{i=1}^{\frac{n(n+1)}{2}} v_i \tilde{E}_{\alpha(i)}\right)a = S(v)a.$$
This completes the proof.
\end{proof}
We are now ready for the
\begin{proof}[Proof of Theorem \ref{theorem:PhiSurjective}]
In a first step, let $K=3$ and consider $\begin{pmatrix}
a\\b
\end{pmatrix}\in \CwO$. Observe that $M_K(-Z)$ is the inverse of $M_K(Z)$ for every $K$ and every $Z$. Then $\Phi_3(\vec{Z}_3)=\begin{pmatrix}
a\\b
\end{pmatrix}$ if and only if
$$ M_2(Z_2)M_1(Z_1)e_{2n} = M_3(-Z_3)\begin{pmatrix}
a\\ b
\end{pmatrix}.$$
The left-hand-side is given by
$$ \El{Z_2}\begin{pmatrix}
Z_1e_n\\e_n
\end{pmatrix} = \begin{pmatrix}
Z_1e_n\\ (I_n+Z_2Z_1)e_n
\end{pmatrix}$$
and the right-hand-side by
$$ \Eu{-Z_3}\begin{pmatrix}
a\\ b
\end{pmatrix} = \begin{pmatrix}
a-Z_3b\\ b
\end{pmatrix}.$$
The symmetric matrix $Z_3$ can be chosen such that $a-Z_3b\neq 0$. To see this, observe that it is automatically satisfied if $b=0$, since $a$ and $b$ aren't simultaneously zero. On the other hand, if $b\neq 0$, then an application of Lemma \ref{lemma:symmetricSurjective} provides the existence of such $Z_3$.

The vector $e_n$ is obviously non-zero, hence another application of Lemma \ref{lemma:symmetricSurjective} yields the existence of a symmetric matrix $Z_1$ such that
$$ Z_1e_n = a-Z_3b.$$
We have $Z_1e_n\neq 0$ by construction, which enables a third application of Lemma \ref{lemma:symmetricSurjective} and proves the existence of a symmetric matrix $Z_2$ such that $(I_n+Z_2Z_1)e_n = b$. Thus we've found $\vec{Z}_3\in (\Cm)^3$ such that $$\Phi_3(\vec{Z}_3)=\begin{pmatrix}
a\\ b
\end{pmatrix}.$$
Moreover, since $Z_1e_n\neq 0$, we even have $\vec{Z}_3\in \Wk{3}$, which completes the proof for $K=3$.

For $K>3$ and $x\in \CwO$ we find $\vec{Z}_3\in \Wk{3}$ such that $\Phi_3(\vec{Z}_3)=x$. Now, we set $Z_K:=(\vec{Z}_3,0,...,0)\in (\Cm)^K$. Then we have $\Phi_K(\vec{Z}_K)=\Phi_3(\vec{Z}_3)=x$ and, moreover, $\vec{Z}_K\in \Wk{K}$ by definition. This completes the proof.
\end{proof}

\subsection{Stratification of \texorpdfstring{$\CwO$}{}}\label{subsection:stratification}
Consider a fixed point $y:=(a,b)\in \CwO$ and let $\mathcal{F}^K_y:=\kfiber:=\Phi_K^{-1}(y)$ denote the fiber of $\Phi_K$ over $y$. By the recursive formula (\ref{RecursiveFormula}) of $\Phi_K$ we can write the $K$-fiber $\mathcal{F}^K_y$ as a union of $(K-1)$-fibers
$$ \mathcal{F}^K_y = \bigcup_{Z\in \Cm} \mathcal{F}^{K-1}_{M_K(Z)y}$$

For an ``appropriate'' stratification of $\CwO$, we will use the projection $\pi_K:\C{n}\times \C{n}\to \C{n}$ given by
$$ \pi_K(u,v):=\begin{cases} u &\text{if } K=2k\\
v &\text{if } K=2k+1.
\end{cases}$$
The following stratification turns out to be a natural one.
Let $${Y^K_{ng} = \{y\in \CwO: \pi_K(y)=0\}}$$ denote the \textit{non-generic} stratum and its complement $Y^K_g = (\CwO)\setminus Y^K_{ng}$ the \textit{generic} stratum.

\begin{lemma}\label{lemma:nonGenericFibers}
For a point $y\in Y_{ng}^K$ in the non-generic stratum, the corresponding non-generic fiber $\mathcal{F}^K_y$ satisfies
$$ \mathcal{F}^K_y = \mathcal{F}^{K-1}_y\times \Cm,$$
where $\mathcal{F}^{K-1}_y$ is a generic $(K-1)$-fiber.
\end{lemma}
\begin{proof}
We carry out the proof only for $K=2k+1$, since it applies equally to $K=2k$ for reasons of symmetry. Let $y=(a,b)\in Y^K_{ng}$, that is, $\pi_K(y)=0$. Observe that $y=(\pi_{K-1}(y),\pi_K(y))$. On the one hand, this means $b=\pi_K(y)=0$ and on the other $a=\pi_{K-1}(y)\neq 0$, since $y\neq 0$ by definition. This implies that $\mathcal{F}^{K-1}_y$ is a generic $(K-1)$-fiber.

The non-generic $K$-fiber $\mathcal{F}^K_y$ is given by the defining equations $\Phi_K(\vec{Z}_K)=y$. By the recursive formula (\ref{RecursiveFormula}) of $\Phi_K$ this system of equations is equivalent to $\Phi_{K-1}(\vec{Z}_{K-1})=M_K(-Z_K)y$. But there holds
$$ M_K(-Z_K)y = \Eu{-Z_K}\begin{pmatrix}
a\\ 0
\end{pmatrix} = \begin{pmatrix}
a\\ 0
\end{pmatrix} = y,$$
which means that the defining equations are independent of the matrix $Z_K\in \Cm$. In fact, we obtain
$$\mathcal{F}^K_y = \mathcal{F}^{K-1}_y\times \Cm.$$
\end{proof}

Informally, the next statement tells us that for fibers in the generic stratum we can reduce the number of defining equations from $2n$ to $n$. 

First we introduce the following convention. Let $\pi:\C{k}\to \C{l}$ be the standard projection $(z_1,...,z_l,...,z_k)\mapsto (z_1,...,z_l)$. For a continuous mapping $f:\C{l}\to \C{m}$ its pullback $\pi^*f$ is a mapping $\pi^*f:\C{k}\to \C{m}$ and by an abuse of notation, we just write $f$ instead of $\pi^*f$.

Also define $\tilde{P}^K:=\pi_{K+1}\circ \Phi_K$. Then, $\tilde{P}^K_j$ and $\pi_K(y)_j$ denote the $j$-th component of $\tilde{P}^K$ and $\pi_K(y)$, respectively.

\begin{lemma}\label{lemma:genericFibers}
Set $X:=(\Cm)^{K-1}\times \C{\frac{n(n-1)}{2}}$ and define the variety
$$ \mathcal{G}_{\pi_K(y)} :=\{\vec{Z}\in X: \tilde{P}^{K-1}(\vec{Z})=\pi_K(y)\}.$$
Then there are $n$ meromorphic mappings $\psi_j^K:X\to X\times \C{n}$, $1\leq j\leq n$, such that each generic fiber $\mathcal{F}^K_y$ is biholomorphic to $\mathcal{G}_{\pi_K(y)}$ via $\psi_j^K$ for some $1\leq j\leq n$.
\end{lemma}

\begin{proof}
Let $y\in Y^K_g$ be a point in the generic stratum, i.e. $\pi_K(y)_j\neq 0$ for some $1\leq j\leq n$. From the definition of $\pi_K$ and $\tilde{P}^K$ on the one hand and the recursive formula (\ref{RecursiveFormula}) on the other hand, it follows that the fiber $\mathcal{F}^K_y$ is given by
$$ \mathcal{F}^K_y = \{\vec{Z}_K\in (\Cm)^K: \tilde{P}^K(\vec{Z}_K)=\pi_{K+1}(y), \tilde{P}^{K-1}(\vec{Z}_K)=\pi_K(y)\}$$
and further there holds
\begin{align}\label{RecursiveFormulaGeneralized}
\tilde{P}^K=\tilde{P}^{K-2}+Z_K\tilde{P}^{K-1}.
\end{align}
Hence $\PK_j\neq 0$ is satisfied on the fiber $\mathcal{F}^K_y$ and the latter equation can be rearranged to
$$ Z_Ke_j = \frac{1}{\PKM_j}\left(\PK-\PKMM - \sum_{k=1,k\neq j}^n \PKM_k Z_Ke_k\right).$$
We obtain
\begin{align}
f_{ij}:=z_{K,ij} &= \frac{1}{\PKM_j}\left( \PK_i - \PKMM_i - \sum_{k=1,k\neq j}^n \PKM_kz_{K,ik}\right) \quad 1\leq i\leq n, i\neq j,\\
f_{jj}:=z_{K,jj} &= \frac{1}{\PKM_j}\left( \PK_j - \PKMM_j - \frac{1}{\PKM_j} \sum_{k=1,k\neq j}^n b_k \left( \PK_k - \PKMM_k - \sum_{l=1,l\neq j}^n \PKM_l z_{K,kl}\right) \right).
\end{align}
Set $f_j:=(f_{1j},...,f_{nj}):X\to \C{n}$ and $\psi_j: X \to X\times \C{n}$, $\Psi_j(x)=(x,f_j(x))$. By construction, the variety $\mathcal{G}_{\pi_K(y)}$ is mapped biholomorphically onto $\mathcal{F}^K_y$ by $\psi_j$. 
\end{proof}

\subsubsection{On the singularities of the fibers}\label{subsubsection:singularities}
In this short section we will classify the fibers; we distinguish between smooth and singular fibers. In fact, most of the fibers $\mathcal{F}^K_y$ are completely contained in $\mathcal{W}_K$ and therefore smooth, by Lemma \ref{lemma:phisASubmersion}.
\begin{lemma}\label{lemma:singularFibers}
A fiber $\mathcal{F}^K_y$ contains singularities if and only if $\pi_1(y)=e_n$.
\end{lemma}
\begin{proof}
We start with the case $K=2k$. Suppose there is a singularity $\vec{Z}_K\in \mathcal{F}^K_y\cap S_K$. Then there holds $Z_{2i-1}e_n=0$ for all $1\leq i \leq k$, by Lemma \ref{lemma:singularitySet}. Lemma \ref{lemma:propertyOne} implies $\Phi_K(\vec{Z}_K)=e_{2n}$ and thus $\pi_1(y)=\pi_1(\Phi_K(\vec{Z}_K))=\pi_1(e_{2n}) = e_n$.

Now let $K=2k+1$ and suppose again there is a singularity $\vec{Z}_K\in \mathcal{F}^K_y\cap S_K$. Again there holds $Z_{2i-1}e_n=0$ for all $1\leq i \leq k$, by Lemma \ref{lemma:singularitySet}. From the even case, we know that $\Phi_{K-1}(\vec{Z}_{K-1})=e_{2n}$. The recursive formula (\ref{RecursiveFormula}) of $\Phi_K$ implies 
$$ y=\Phi_K(\vec{Z}_K) = M_K(Z_K)\Phi_{K-1}(\vec{Z}_{K-1}) = \Eu{Z_K}\begin{pmatrix}
0\\ e_n
\end{pmatrix} = \begin{pmatrix}
Z_Ke_n \\ e_n
\end{pmatrix},$$
and therefore $\pi_1(y)=e_n$. This completes the proof of necessary condition.

For the proof of the sufficient condition, we consider a fiber $\mathcal{F}^K_y$ with $\pi_1(y)=e_n$ and set $k:=\big\lceil \frac{K-1}{2}\big\rceil$. Let $\vec{Z}_K\in \mathcal{F}^K_y$ and recall that $\vec{Z}_K=(Z_1,Z_2,...,Z_K)\in (\Cm)^K$. The matrices $Z_2,Z_4,...,Z_{2k}\in \Cm$ can take any value, since $Z_1e_n=Z_3e_n=...=Z_{2k-1}e_n=0$ by assumption and therefore $\Phi_i(\vec{Z}_i)=e_{2n}$ for all $1\leq i\leq 2k$ by Lemma \ref{lemma:propertyOne}. Hence we set $Z_2=Z_4=...=Z_{2k}=0$ and then $\vec{Z}_K\in \mathcal{F}^K_y\cap S_K$ is a singularity by Lemma \ref{lemma:singularitySet}.
\end{proof}

\subsection{Holomorphic vector fields tangential to the fibers}
We will construct stratified sprays using $\mathbb{C}$-complete holomorphic vector fields. 
The main goal of this subsection is therefore to find enough vector fields that are holomorphic on $(\Cm)^K$, complete and, in particular, tangential to the fibers $\mathcal{F}^K_y$. 

\subsubsection{Fiber-preserving vector fields}
Let $X$ be some Stein manifold and $f=(f_1,...,f_n):X\to \C{n}$ a holomorphic mapping.
A holomorphic vector field ${V:X\to TX}, {x\mapsto (x,V_x)}$ is \textit{fiber-preserving} for $f$, if it is tangential to the fibers of $f$. This is the case if and only if $V$ is in the kernel of the tangent map $df$, that is, $df_x(V_x)=0$. This is equivalent to say that the Lie derivative $\mathcal{L}_{V_x}(f_i)= V_x(f_i)=(df_i)_x(V_x)=0$ vanishes for all $1\leq i\leq n$.
\begin{lemma}\label{lemma:fiberPreservingVectorFields}
For $N>n$, let $P\in \mathbb{C}[z_1,...,z_N]^n$ be a polynomial mapping $P:\C{N}\to \C{n}$ and let $x=(z_{\alpha_0},...,z_{\alpha_n})$, with $1\leq \alpha_0<...<\alpha_n\leq N$. Let's write $\deriv{z}P:=(\deriv{z}P_1,...,\deriv{z}P_n)^T$. Then
$$ D_x(P):= \det \begin{pmatrix}
\Deriv{z}{\alpha_0} & \cdots & \Deriv{z}{\alpha_n}\\
\Deriv{z}{\alpha_0}P & \cdots & \Deriv{z}{\alpha_n}P
\end{pmatrix}$$
is a holomorphic vector field on $\C{N}$ which is fiber-preserving for $P$.
\end{lemma}
\begin{proof}
The Lie derivative $\mathcal{L}_{D_{x}(P)}(P_i)= D_{x}(P)(P_i)=0$ vanishes for each $1\leq i\leq n$, since the first and the $(i+1)$-th row of $D_{x}(P)(P_i)$ are the same.
\end{proof}

We now introduce a few more notations. For a fixed natural number $K$, the mapping $\tilde{P}^K$ (defined before Lemma \ref{lemma:genericFibers}) is a polynomial mapping in $\mathbb{C}[z_1,...,z_{N_K}]^n$ with $N_K:=K\tfrac{n(n+1)}{2}$. Since we'll only be interested in $K>1$, the constraint $N_K>n$ is given. 

There are $\binom{N_K}{n+1}$ possibilities to choose $(n+1)$ of the $N_K$ variables. Let $\mathcal{T}_K$ denote the set of all such possible choices. Recall that we see $\C{N_K}$ as a product of $K$ copies of $\Cm$ and we write $\vec{Z}_K=(Z_1,...,Z_K)\in (\Cm)^K$. With this convention, the set $\mathcal{T}_K$ can be given by
$$ \mathcal{T}_{K} := \{x=(z_{i_0,j_0k_0},...,z_{i_n,j_nk_n}): 1\leq i_0\leq ... \leq i_n\leq K, 1\leq j_r\leq k_r\leq n, 0\leq r\leq n\}.$$
The vector fields $\partial_x^{K}:=D_x(\tilde{P}^{K}), x\in \mathcal{T}_{K}$, are fiber-preserving for $\tilde{P}^{K}$, by Lemma \ref{lemma:fiberPreservingVectorFields}; or, equivalently, they're tangential to the variety $\mathcal{G}_{\pi_{K+1}(y)}$ by construction (c.f. Lemma \ref{lemma:genericFibers}).

The following lemma collects some interesting examples of fiber-preserving vector fields for $\Phi_K$. In fact, they'll play a cruical role in the construction of a dominating spray (see Theorem \ref{theorem:spanTangentSpace})

\begin{lemma}\label{lemma:liftingVectorFields}
Let $1\leq j^*\leq n$, $x\in \mathcal{T}_{K-1}$ and $u := (\partial_x^{K-1}(\PKMM_1),...,\partial_x^{K-1}(\PKMM_n))^T$. Then the vector field 
\begin{align*}
\varphi^K_{x,j^*} &= (\PKM_{j^*})^2\partial_x^{K-1} - \PKM_{j^*} \sum_{\substack{i=1\\ i\neq j^*}}^n u_i \Deriv{z}{K,j^*i} + \left( \sum_{\substack{i=1\\i\neq j^*}}^n \PKM_iu_i-\PKM_{j^*}u_{j^*}\right) \Deriv{z}{K,j^*j^*}
\end{align*}
is holomorphic on $(\Cm)^K$ and fiber-preserving for $\Phi_K$. Moreover, $\varphi^K_{x,j^*}$ is complete if and only if $\partial_x^{K-1}$ is complete.

For $1\leq i\leq j\leq  n, i\neq j^*,j\neq j^*$, the vector field
\begin{align*}
\gamma_{ij,j^*}^K &= (\tilde{P}^{K-1}_{j^*})^2\Deriv{z}{K,ij} + \frac{1}{1+\delta_{ij}}\left( 2\PKM_i\PKM_j \Deriv{z}{K,j*j*} - \PKM_i\PKM_{j^*} \Deriv{z}{K,j^*j} - \PKM_j\PKM_{j^*} \Deriv{z}{K,j^*i}\right)
\end{align*}
is complete, holomorphic on $(\Cm)^K$ and fiber-preserving for $\Phi_K$.
\end{lemma}

\begin{proof}
Observe that $\tilde{P}^{K-1}\equiv 0$ on non-generic fibers $\mathcal{F}^K_y$. Hence the above fields are trivial and there is nothing to show. We therefore consider $y\in Y_g^K$ in the generic stratum, i.e. $\pi_K(y)\neq 0$. Without loss of generality assume $\pi_K(y)_1\neq 0$. According to Lemma \ref{lemma:genericFibers}, the mapping $\psi_1^K$ is defined by $\Psi_1(x)=(x,f_1(x))$ for some meromorphic map $f_1=(f_{11},...,f_{n1})$ and it maps $\mathcal{G}_{\pi_K(y)}$ biholomorphically onto the generic fiber $\mathcal{F}^K_y$.

Consider a vector field $V$ tangential to $\mathcal{G}_{\pi_K(y)}$. Then the push-forward $W:=(\Psi_1)_*(V)$ is given by
$$ W = V + \sum_{i=1}^n V(f_{i1})\Deriv{z}{K,i1}.$$
On the one hand $W$ is tangential to the fiber $\mathcal{F}^K_y$ and on the other hand it is complete if and only if $V$ is complete. 

Let's write $W(\PK):=(W(\PK_1),...,W(\PK_n))^T$ and $W(Z_K):=\sum_{1\leq i\leq j\leq n} W(z_{K,ij})\tilde{E}_{ij}$. By the recursive formula (\ref{RecursiveFormulaGeneralized}) and since $W(\PKM)=0$, we get
\begin{align*}
0 &= W(\PK) = W(\PKMM) + W(Z_K)\PKM = u + W(Z_K)\PKM.
\end{align*}

In the special case, where $V=\partial_x^{K-1}$ for some $x \in \mathcal{T}_{K-1}$, there holds
\begin{align*}
W(Z_K)\PKM &= \left(\sum_{i=1}^n V(f_{i1})\tilde{E}_{i1}\right) \PKM = \begin{pmatrix}
\sum_{i=1}^n V(f_{i1})\PKM_i\\ \PKM_1 V(f_{21}) \\ \vdots \\ \PKM_1 V(f_{n1})
\end{pmatrix} \\&= \underbrace{\begin{pmatrix}
\PKM_1 & \cdots & \PKM_n \\ & \ddots & \\ && \PKM_1
\end{pmatrix}}_{=:A}\underbrace{\begin{pmatrix}
V(f_{11})\\ \vdots \\ V(f_{n1})
\end{pmatrix}}_{=:b},
\end{align*}
where $A$ is a regular matrix, since $\PKM_1\neq 0$. Therefore we obtain $b = -A^{-1}u$ with
$$ A^{-1} = \frac{1}{(\PKM_1)^2}\begin{pmatrix}
\PKM_1 & -\PKM_2 & -\PKM_3 & \cdots & -\PKM_n \\ & \PKM_1 & && \\ && \ddots & & \\
&&& \ddots & \\ &&&& \PKM_1
\end{pmatrix}.$$
The vector field $\varphi_{x,1}^K:=(\PKM_1)^2W$ is holomorphic on $(\Cm)^K$ and fiber-preserving for $\Phi_K$. Moreover, it is complete if and only if $W$ is complete, since $\PKM$ is in the kernel of $W$. This proves the first part of the lemma.

In the special case, where $V=\Deriv{z}{K,ij}, 1<i\leq j\leq n$, there holds $u=W(\PKMM)=0$ and
$$ W(Z_K)\PKM = \tilde{E}_{ij}\PKM + \left(\sum_{i=1}^n V(f_{i1})\tilde{E}_{i1}\right) \PKM = \tilde{E}_{ij}\PKM + Ab. $$
Therefore
\begin{align*}
b &= -A^{-1}\tilde{E}_{ij}\PKM = -\frac{1}{1+\delta_{ij}}A^{-1}(\PKM_ie_j + \PKM_je_i)\\
&= \frac{1}{1+\delta_{ij}}\frac{1}{(\PKM_1)^2}\left(2\PKM_i\PKM_j e_1 - \PKM_1\PKM_i e_j-\PKM_1\PKM_j e_i\right).
\end{align*}
As before, we multiply $W$ by $(\PKM_1)^2$ to obtain a fiber-preserving vector field $\gamma^K_{ij,1}:=(\PKM_1)^2W$ which is holomorphic on $(\Cm)^K$. Note that $\Deriv{z}{K,ij}$ is complete on $\mathcal{G}_{\pi_K(y)}$ by definition, hence $\gamma_{ij,1}^K$ is complete. This proves the second part of the lemma.
\end{proof} 
\subsection{Complete holomorphic vector fields tangent to the fibers}
In the previous subsection, we've constructed vector fields tangent to the fibers. Unfortunately, some of those fields aren't complete (see Example \ref{example:incomplete} below). In addition, it is quite laborious to decide whether a given field is complete. The goal of this subsection is to build machinery that will make this decision easier. Furthermore, we will list the most important examples of complete fields.

For $K>L$, let $f_{K,L}:(\Cm)^K\to (\Cm)^L, (Z_1,...,Z_K)\mapsto (Z_1,...,Z_L)$ be the standard projection. Given a vector field $V$ on $(\Cm)^L$ tangential to the fibers $\mathcal{F}^L_y$, its pullback $f_{K,L}^*V$ is a vector field on $(\Cm)^K$ tangential to the fibers $\mathcal{F}^K_{y'}$, by the recursive formula (\ref{RecursiveFormula}). By an abuse of notation, we just write $V$ instead of $f_{K,L}^*V$. Furthermore, set
$$ \mathcal{T}_{K}^C:=\{x\in \mathcal{T}_{K}: \partial_x^{K} \text{ is a complete vector field}\}.$$

\begin{defi}\label{def:liftedVectorFields} For $K\geq 3$, define the collection
$$\mathcal{V}_K:=\bigcup_{L=3}^K \left(\bigcup_{j=1}^n \left\lbrace \varphi_{x,j}^L: x\in \mathcal{T}_{L-1}^C \right\rbrace \cup \left\lbrace \gamma^L_{rs,j}: 1\leq r,s \leq n \right\rbrace\right)$$
of \textbf{principal vector fields} for $\Phi_K$.
\end{defi}

\begin{Rem}
We can consider $\mathcal{V}_{K-1}$ as a subset of $\mathcal{V}_K$ using the convention introduces just before the definition.
\end{Rem}

We are now working on a machinery that should make it easier for us to decide whether a tuple $x$ corresponds to a $\mathbb{C}$-complete vector field $\partial_x^K$. Let $x=(x_1,...,x_m)\in \C{m}$ and $P:\C{m}\to\C{m-1}$ be a polynomial mapping. We define an equivalence relation on $\C{l}$ in the following way
$$ u,v\in \C{l}, u\rel v \quad  :\Leftrightarrow \quad  u_i-v_i\in \bigcap_{k=1}^m \ker\left(\frac{\partial}{\partial x_k}\right), \text{ for all } 1\leq i \leq l.$$
A vector $v\in \C{l}$ is called \textit{constant in $x$} if $v\rel 0$.

\begin{theorem}\label{lemma:classifyCompleteFields}
Let $x=(x_1,...,x_m)\in \C{m}$ and $P:\C{m}\to\C{m-1}, x\mapsto P(x)$ be a polynomial mapping. Assume there exists $v\in \C{m-1}$, $v \rel 0$, and $\lambda_{ij}\in \mathbb{C}$, $\lambda_{ij}\rel 0$, for all $1\leq i,j\leq m$, such that
\begin{equation}
\Deriv{x}{i} \Deriv{x}{j} P(x) = \lambda_{ij}v.
\end{equation}
Then $V=D_x(P)$ is $\mathbb{C}$-complete.
\end{theorem}
\begin{proof}
We consider the vector field $V=D_x(P)=\sum_{j=1}^m V_j\Deriv{x}{j}$, where $V_j$ is given by
$$ V_j = \det\left( \frac{\partial P}{\partial x_1},...,\frac{\partial P}{\partial x_{j-1}},\frac{\partial P}{\partial x_{j+1}},...,\frac{\partial P}{\partial x_m}\right).$$
For $1\leq j\leq m$, define $f_j(x)=\sum_{k=1}^m \lambda_{kj}x_k$. Then $\Deriv{x}{j}P(x) \rel f_j(x)v$ by construction, this means
$$ \Deriv{x}{j}P(x) = c_j + f_j(x)v,$$
for some $c_j\in \mathbb{C}, c_j\rel 0$. 
Obviously, $f_k(x)v$ and $f_l(x)v$ are linearly dependent, therefore
\begin{align*}
V_j&= \det \left( c_1+f_1v,...,c_{j-1}+f_{j-1}v,c_{j+1}+f_{j+1}v,...,c_m+f_mv\right)\\
&= \underbrace{\det(c_1,...,c_{j-1},c_{j+1},...,c_m)}_{=:\alpha_{j0}} + f_1\underbrace{\det(v,c_2,...,c_{j-1},c_{j+1},...,c_m)}_{=:\alpha_{j1}} +\\
&+\cdots + f_n\underbrace{\det(c_1,...,c_{j-1},c_{j+1},...,c_{m-1},v)}_{=:\alpha_{jm}}.
\end{align*}
With the convention $\alpha_{jj}:=0$, we obtain
\begin{align*}
V_j &= \alpha_{j0} + \sum_{k=1}^m \alpha_{jk}f_k = \alpha_{j0}+ \sum_{k=1}^m \alpha_{jk}\sum_{l=1}^m \lambda_{lk}x_l = \alpha_{j0} + \sum_{l=1}^m x_l \underbrace{\sum_{k=1}^m \lambda_{lk}\alpha_{jk}}_{=:a_{jl}}.
\end{align*}
Set $b:=(\alpha_{10},...,\alpha_{m0})^T$. Then we just proved that $V(x)=Ax+b$, where $A=(a_{jl})_{1\leq j,l\leq m}$ is a $m\times m$-matrix with $a_{jl}\rel 0$.

Let $\gamma$ be a flow curve, i.e. a holomorphic map $\gamma:\mathbb{C}\to \C{m}$ with $\frac{d}{dt}\gamma(t)=V(\gamma(t))$. This leads to the system
$$ \frac{d}{dt}\gamma(t) = A\gamma(t) + b,$$
which implies that $\gamma$ exitsts for all time $t\in \mathbb{C}$.
\end{proof}

The mapping $\tilde{P}^K:(\Cm)^K\to \C{n}$ does not a priori fit into the setting of the previous lemma, but this problem can be solved with a simple trick. By fixing all but $(n+1)$ of the $\frac{n(n+1)}{2}K$ variables, we may interpret $\tilde{P}^K$ as a polynomial mapping $\C{n+1}\to \C{n}$. More precisely, each $(n+1)$-tupel $x\in \mathcal{T}_K$ corresponds to a natural inclusion map $i_x:\C{n+1}\to (\Cm)^K$. Then $\tilde{P}^K\circ i_x$ is a polynomial mapping $\C{n+1}\to \C{n}$.
\begin{prop}\label{lemma:completeFields} \textbf{(List of complete vector fields)}

\begin{enumerate}[align=left,labelwidth=0.85cm,label=(\textbf{Type \arabic*}),ref=\textit{Type \arabic*}]
\item \label{typeOne} For $1\leq m\leq n$, $x=(z_{k-1,mm},z_{k,11},...,z_{k,nn}) \in \mathcal{T}_K^C,\quad K\geq k.$
\item \label{typeTwo} For $l\neq m$, $x=(z_{k-1,mm},z_{k,l1},...,z_{k,ln}) \in \mathcal{T}_K^C,\quad K\geq k.$
\item \label{typeThree} For $(n+1)$ distinct pairs of indices $(i_0,j_0),...,(i_n,j_n)$,
${x=(z_{k,i_0j_0},...,z_{k,i_nj_n}) \in \mathcal{T}_K^C},\  {K\geq k}.$
\item \label{typeFour} For $1\leq i^*\leq n$, $x=(z_{1,ni^*},z_{2,11},...,z_{2,nn})\in \mathcal{T}_K^C, K\geq 2.$
\item \label{typeFive} Let $k<l$, $1\leq j^*\leq n$ and let $(i_1,j_1),...,(i_n,j_n)$ be $n$ distinct pairs of indices. Then
$$ x=(z_{k,i_1j_1},...,z_{k,i_nj_n}, z_{l,j^*j^*}) \in \mathcal{T}_K^C, K\geq l.$$
\item \label{typeSix} For $1\leq r\leq n$ consider the partition $\{1,...,n\}=\{i_1,...,i_r\}\dot{\cup}\{j_1,...,j_{n-r}\}$. Let $i^*\in \{i_1,...,i_r\}$ and $j^*,j'\in \{j_1,...,j_{n-r}\}$. Then $$x=(z_{k-1,j^*j_1},...,z_{k-1,j^*j_{n-r}},z_{k,i^*i_1},...,z_{k,i^*i_r},z_{k,i^*j'})\in \mathcal{T}_K^C, K\geq k.$$
\item \label{typeSeven} For $0\leq r\leq n$ consider the partition $\{1,...,n\}=\{i_1,...,i_r\}\dot{\cup}\{j_1,...,j_{n-r}\}$. Let $i^*\in \{i_1,...,i_r\}$ and $j^*,j'\in \{j_1,...,j_{n-r}\}$. Then $$x=(z_{k,j^*j_1},...,z_{k,j^*j_{n-r}},z_{k+1,i^*i_1},...,z_{k+1,i^*i_r},z_{k+2,j'j'})\in \mathcal{T}_K^C, K\geq k+2.$$
\item \label{typeEight} For $1\leq i\leq n, i\neq j$, $x=(z_{1,in},z_{2,j1},...,z_{2,jn})\in \mathcal{T}_K^C, K\geq 2.$
\end{enumerate}
\end{prop}
\begin{proof}
For the proof of (\ref{typeOne}), we consider the mapping $P:\C{n+1}\to \C{n}$ given by
 $$x:=(z_{k-1,mm},z_{k,11},...,z_{k,nn})\mapsto \begin{pmatrix}
A&B
\end{pmatrix}\Eu{Z_{k}}\El{Z_{k-1}}\begin{pmatrix}
c\\d
\end{pmatrix},$$ where $A$ and $B$ are arbitrary $n\times n$-matrices, both constant in $x$; whereas $c$ and $d$ are arbitrary vectors in $\C{n}$ both constant in $x$. Observe that
$$ \begin{pmatrix}
A&B
\end{pmatrix}\Eu{Z_k}\El{Z_{k-1}}\begin{pmatrix}
c\\d
\end{pmatrix} = \begin{pmatrix}
B&A
\end{pmatrix}\El{Z_k}\Eu{Z_{k-1}}\begin{pmatrix}
d\\c
\end{pmatrix}.$$
Thanks to this symmetry condition, we don't need to make a case distinction between even and odd $K$. In fact, it is enough to prove that $D_x(P)$ is a $\mathbb{C}$-complete vector field. At first, note that 
$$ \frac{\partial^2}{\partial z_{k-1,mm}^2}P(x)\equiv 0,\quad \Deriv{z}{k,ii}\Deriv{z}{k,jj}P(x)\equiv 0, \quad 1\leq i,j\leq n.$$
Hence most of the $\lambda$'s in Lemma \ref{lemma:classifyCompleteFields} can be chosen to be zero. It remains to consider $\Deriv{z}{k,ii}\Deriv{z}{k-1,mm}P(x)$, $1\leq i\leq n$.
We get
\begin{align*}
\Deriv{z}{k,ii}\Deriv{z}{k-1,mm}P(x) &= \begin{pmatrix}
A&B
\end{pmatrix}\begin{pmatrix}
0&\tilde{E}_{ii}\\0&0
\end{pmatrix}\begin{pmatrix}
0&0\\ \tilde{E}_{mm} & 0
\end{pmatrix}\begin{pmatrix}
c\\d
\end{pmatrix}\\&= \begin{pmatrix}
0&A\tilde{E}_{ii}
\end{pmatrix}\begin{pmatrix}
0\\ \tilde{E}_{mm}c
\end{pmatrix}\\ &= A\tilde{E}_{ii}\tilde{E}_{mm}c = c_m\delta_{mi}Ae_m,
\end{align*}
for all $1\leq i\leq n$. Since $v:=Ae_m$ is independent of $i$, the conditions of Lemma \ref{lemma:classifyCompleteFields} are satisfied. Therefore $D_x(P)$ is a complete field and we conclude that vector fields $\partial_x^K$, $K\geq k$, of (\ref{typeOne}) are complete.

For the proof of (\ref{typeTwo}) we choose $P$ as before. Again we have $\frac{\partial^2}{\partial z_{k-1,mm}^2}P\equiv 0$ and $\Deriv{z}{k,li}\Deriv{z}{k,lj}P\equiv 0$, $1\leq i,j\leq n$. Further, we compute
\begin{align*}
\Deriv{z}{k,li}\Deriv{z}{k-1,mm} P(x) =\begin{pmatrix}
A&B
\end{pmatrix}\begin{pmatrix}
0&\tilde{E}_{li}\\0&0
\end{pmatrix}\begin{pmatrix}
0&0\\ \tilde{E}_{mm} & 0
\end{pmatrix}\begin{pmatrix}
c\\d
\end{pmatrix}= A\tilde{E}_{li}\tilde{E}_{mm}c = c_mA\tilde{E}_{li}e_m.
\end{align*}
Observe that $\tilde{E}_{li}e_m = \delta_{im}e_l$, since we assume $l\neq m$. Therefore we obtain $$\Deriv{z}{k,li}\Deriv{z}{k-1,mm} P(x)= \delta_{im}c_mAe_l$$ for all $1\leq i\leq n$. This proves that the conditions of Lemma \ref{lemma:classifyCompleteFields} are satisfied and hence the field $D_x(P)$ is complete. We conclude that the fields $\partial_x^K$, $K\geq k$, of (\ref{typeTwo}) are complete.

For the proof of (\ref{typeThree}), we set $P(x)=\begin{pmatrix}
A&B
\end{pmatrix}\Eu{Z_k}\begin{pmatrix}
c\\d
\end{pmatrix}$. Observe that ${\Deriv{z}{k,i_rj_r}P(x)\rel 0}$ for all $r=0,...,n$. Hence the conditions of Lemma \ref{lemma:classifyCompleteFields} are trivially satisfied and we conclude that the vector fields $\partial_x^K$ of (\ref{typeThree}) are complete.

For the proof of (\ref{typeFour}) we set $P(x)=\begin{pmatrix}
A&B
\end{pmatrix}\El{Z_2}\Eu{Z_1}e_{2n}$. Note that $\frac{\partial^2}{\partial z_{1,ni^*}^2}P\equiv 0$ and $\Deriv{z}{2,ii}\Deriv{z}{2,jj}P\equiv 0$, $1\leq i,j\leq n$. Furthermore, we compute
$$ \Deriv{z}{2,jj}\Deriv{z}{1,ni^*}P(x) = \begin{pmatrix}
A & B
\end{pmatrix} \begin{pmatrix}
0&0\\ \tilde{E}_{jj} & 0
\end{pmatrix} \begin{pmatrix}
0& \tilde{E}_{ni^*} \\ 0&0
\end{pmatrix} \begin{pmatrix}
0\\ e_n
\end{pmatrix} = B\tilde{E}_{jj}\underbrace{\tilde{E}_{ni^*}e_n}_{=e_{i^*}} = \delta_{ji^*}Be_{i^*}.$$
Hence we can apply Lemma \ref{lemma:classifyCompleteFields} and conclude that vector fields $\partial_x^K$, $K\geq 2$, of (\ref{typeFour}) are complete.

For the proof of (\ref{typeFive}) we first consider
$$ P(x)= \begin{pmatrix}
A&B
\end{pmatrix}\El{Z_{l}}\begin{pmatrix}
U & V \\ W & X
\end{pmatrix}\Eu{Z_k}\begin{pmatrix}
c\\d
\end{pmatrix},$$
where $U,V,W$ and $X$ are arbitrary $n\times n$-matrices constant in $x$. As in the previous cases, we have $\frac{\partial^2}{\partial z_{l,j^*j^*}^2}P\equiv 0$ and $\Deriv{z}{k,i_rj_r}\Deriv{z}{k,i_sj_s}P\equiv 0$, $1\leq r,s\leq n$. Further, let's compute
\begin{align*}
\Deriv{z}{l,j^*j^*}\Deriv{z}{k,ij}P(x) &= \begin{pmatrix}
A&B
\end{pmatrix} \begin{pmatrix}
0&0\\ \tilde{E}_{j^*j^*}&0
\end{pmatrix} \begin{pmatrix}
U&V\\ W& X
\end{pmatrix} \begin{pmatrix}
0& \tilde{E}_{ij} \\ 0& 0
\end{pmatrix}\begin{pmatrix}
c\\d
\end{pmatrix}\\ &= B\tilde{E}_{j^*j^*} U \tilde{E}_{ij} d = \underbrace{(e_{j^*}^TU\tilde{E}_{ij}d)}_{\rel 0}Be_{j^*}.
\end{align*}
If we replace $\El{Z_l}$ by $\Eu{Z_l}$ in $P$, we obtain
$$ \Deriv{z}{l,j^*j^*}\Deriv{z}{k,ij}P(x) = A\tilde{E}_{j^*j^*}W\tilde{E}_{ij}d = (e_{j^*}^TW\tilde{E}_{ij}d)Ae_{j^*}.$$
In both cases, Lemma \ref{lemma:classifyCompleteFields} implies that $D_x(P)$ is complete and in conclusion, the vector fields $\partial_x^K$, $K\geq l$, of (\ref{typeFive}) are complete.

For the proof of (\ref{typeSix}), we set $P$ as for (\ref{typeOne}). Observe that $\tilde{E}_{i^*i}\tilde{E}_{j^*j}=0$ for $i^*,i\in\{i_1,...,i_r\}$ and $j^*,j\in \{j_1,...,j_{n-r}\}$. Then we get 
\begin{align*}
\Deriv{z}{k-1,j^*j}\Deriv{z}{k,i^*i} P(x) &= \begin{pmatrix}
A&B
\end{pmatrix} \begin{pmatrix} 0&\tilde{E}_{i^*i}\\0&0
\end{pmatrix} \begin{pmatrix}
0&0\\\tilde{E}_{j^*j}&0
\end{pmatrix}\begin{pmatrix}
c\\d
\end{pmatrix}\\
&=A\tilde{E}_{i^*i}\tilde{E}_{j^*j}c =0,
\end{align*}
for all $i\in \{i_1,...,i_r\}$ and $j\in \{j_1,...,j_{n-r}\}$. Furthermore, there holds
\begin{align*}
\Deriv{z}{k-1,j^*j}\Deriv{z}{k,i^*j'}P(x)&=A\tilde{E}_{i^*j'}\tilde{E}_{j^*j}c \rel 0,
\end{align*}
and hence the vector fields $\partial_x^K$, $K\geq k$, of (\ref{typeSix}) are complete, by Lemma \ref{lemma:classifyCompleteFields}.

For the proof of (\ref{typeSeven}) we set $P(x)=\begin{pmatrix}
A&B
\end{pmatrix}\Eu{Z_{k+2}}\El{Z_{k+1}}\Eu{Z_k}\begin{pmatrix}
c\\d
\end{pmatrix}$. By the same argument as in (\ref{typeSix}) we obtain $ \Deriv{z}{k+2,j'j'}\Deriv{z}{k+1,i^*i}P(x)=0$ for all $i\in \{i_1,...,i_r\}$ and $\Deriv{z}{k+1,i^*i}\Deriv{z}{k,j^*j}P(x)=0$ for all $i\in \{i_1,...,i_r\}$, $j\in \{j_1,...,j_{n-r}\}$. It remains to compute
\begin{align*}
\Deriv{z}{k+2,j'j'}\Deriv{z}{k,j^*j}P(x)&=\begin{pmatrix}
A&B
\end{pmatrix} \begin{pmatrix}
0&\tilde{E}_{j'j'} \\ 0&0
\end{pmatrix} \El{Z_{k+1}} \begin{pmatrix}
0&\tilde{E}_{j^*j}\\0&0
\end{pmatrix}\begin{pmatrix}
c\\d
\end{pmatrix}\\
&= A\tilde{E}_{j'j'}\begin{pmatrix}
Z_{k+1}&I_n
\end{pmatrix}\begin{pmatrix}
\tilde{E}_{j^*j}d\\0
\end{pmatrix}\\
&= A\tilde{E}_{j'j'}(Z_{k+1}\tilde{E}_{j^*j}d) = (e_{j'}^TZ_{k+1}\tilde{E}_{j^*j}d)Ae_{j'}.
\end{align*}
Observe that $e_{j'}^TZ_{k+1}$ is constant in $x$, since we assume $j'\in\{j_1,...,j_{n-r}\}$. Hence the conditions of Lemma \ref{lemma:classifyCompleteFields} are satisfied and we conclude that the fields $\partial_x^K$, $K\geq k+2$, of (\ref{typeSeven}) are complete.

For the proof of (\ref{typeEight}), let $P(x)=\begin{pmatrix}
A&B
\end{pmatrix}\El{Z_2}\begin{pmatrix}
Z_1e_n\\e_n
\end{pmatrix}$. For $1\leq r\leq n$, we compute
$$ \Deriv{z}{1,in}\Deriv{z}{2,jr} P(x) = \begin{pmatrix}
A&B
\end{pmatrix}\begin{pmatrix}
0&0\\ \tilde{E}_{jr}&0 
\end{pmatrix}\begin{pmatrix}
e_i\\0
\end{pmatrix} = \delta_{ir}Be_j.$$
Hence the conditions of Lemma \ref{lemma:classifyCompleteFields} are met and the fields $\partial_x^K$, $K\geq 2$, of (\ref{typeEight}) are complete.
\end{proof}

In the following, we find an example of an incomplete vector field $\partial_x^2$. As a direct consequence of this, we don't know how large the collection $\mathcal{V}_K$ of principal vector fields is. We shall see later, however, that it is powerful enough to construct stratified sprays. 

\begin{Exa}\label{example:incomplete}
Consider the tupel $x=(z_{1,n1},...,z_{1,nn},z_{2,n1})$ and the mapping $ P(x) = e_n + Z_2Z_1e_n.$ The Jacobian $JP$ is given by $\begin{pmatrix}
Z_2 & z_{1,nn}e_1+z_{1,n1}e_n
\end{pmatrix}$ and we get the vector field
\begin{align*}
\partial_x^2 = \det \begin{pmatrix}
\partial / \partial z_{1,n1} & \cdots & \partial/\partial z_{1,nn} & \partial/\partial z_{2,n1} \\
z_{2,11} & \cdots & z_{2,1n} & z_{1,nn}\\
z_{2,21} & \cdots & z_{2,2n} & 0 \\
\vdots & &\vdots & \vdots \\
z_{2,n1}&\cdots & z_{2,nn} & z_{1,n1}
\end{pmatrix}.
\end{align*}
Therefore $ \partial_x^2(z_{2,n1}) = \pm \det(Z_2) = \pm(\alpha_1z_{2,n1}^2+\alpha_2z_{2,n1}+\alpha_3)$ for some $\alpha_1,\alpha_2,\alpha_3\in \mathbb{C}\cap \ker(\partial_x^2)$. In fact, $\alpha_1$ is the principal minor of order $n-2$ obtained by removing the first and last rows and columns from $Z_2$. Hence $\alpha_1\not\equiv 0$ on $(\Cm)^2$, which means that the variable $z_{2,n1}$ occurs quadratically. We conclude that $\partial_x^2$ is incomplete.
\end{Exa}

\subsection{Topological analysis of the fibers}\label{subsection:topologicalAnalysis}
So far we haven't learned anything about the fibers $\mathcal{F}^K_y$  from a topological perspective. In this subsection we will show that all fibers are connected for $K\geq 3$. In fact, all fibers are irreducible, except the singular fibers $\mathcal{F}^3_{a,e_n^T}$ and $\mathcal{F}^4_{0,e_n^T}$ which consist of two irreducible components, with the smooth part breaking down in two connected components. 
\begin{lemma}\label{lemma:3fibersAreConnected}
The fibers $\mathcal{F}^3_{a,b}$ are connected. The singular fibers, i.e. $\mathcal{F}^3_{a,e_n^T}$ consists of two irreducible components.
\end{lemma}

\begin{proof}
We start with the non-generic fibers, i.e. we assume $b=0$. In this case, we have $\mathcal{F}^3_{a,0}=\mathcal{F}^2_{a,0}\times \Cm$ with $a\neq 0$. The defining equations of $\mathcal{F}^2_{a,0}$ are given by
$$ \begin{pmatrix}
a\\ 0 
\end{pmatrix} = \Eu{Z_2}\begin{pmatrix}
Z_1e_n\\ e_n
\end{pmatrix} = \begin{pmatrix}
Z_1e_n \\ (I_n+Z_2Z_1)e_n
\end{pmatrix},$$
and therefore $\mathcal{F}^2_{a,0} \cong \C{\frac{n(n-1)}{2}}\times \{ Z\in \Cm: e_n + Za=0\} \cong \C{n(n-1)}$. In conclusion, the non-generic fiber $\mathcal{F}^3_{a,0}$ is biholomorphic to some $\C{N}$ and hence connected.

We continue with the generic fibers $\mathcal{F}^3_{a,b}$, i.e. $b\neq 0$. In the first step, we consider the smooth fibers, that is, we assume $b\neq e_n$ in addition. Due to Lemma \ref{lemma:genericFibers}, the fiber
 $\mathcal{F}^3_{a,b}$ is biholomorphic to $\mathcal{G}_{\pi_3(a,b)}\times \C{\frac{n(n-1)}{2}}$, where
 $$ \mathcal{G}_{\pi_3(a,b)} = \{\vec{Z}_2\in (\Cm)^2: b=P_s(\vec{Z}_2)=e_n+Z_2Z_1e_n\}.$$
Define $ C_j :=\{\vec{Z}_2\in \mathcal{G}_{\pi_3(a,b)}: z_{1,nj}\neq 0\},\  1\leq j\leq n.$ Similar as in Lemma \ref{lemma:genericFibers}, we can express the variables $z_{2,j1},...,z_{2,jn}$, which proves that $C_j$ is biholomorphic to $\mathbb{C}^*\times \C{N}$ for some natural number $N$ and therefore connected. Since we assume $b\neq e_n$, $Z_1e_n=0$ is not possible in $\mathcal{G}_{\pi_3(a,b)}$, which means that $\mathcal{G}_{\pi_3(a,b)}$ is covered by $\cup_{j=1}^n C_j$. It remains to show, that the intersection $\cap_{j=1}^n C_j \neq \emptyset$ is non-empty. Choose a symmetric matrix $Z_1^*$ with $Z_1^*e_n=(1,...,1)^T$. By Lemma \ref{lemma:symmetricSurjective}, there is a symmetric matrix $Z_2$ with $b-e_n = Z_2(Z_1^*e_n)$. This shows that the intersection is indeed non-empty. In conclusion, $\mathcal{G}_{\pi_3(a,b)}$ and $\mathcal{F}^3_{a,b}$ are connected.

Finally, let's consider the singular fibers $\mathcal{F}^3_{a,e_n^T}$. By Lemma \ref{lemma:genericFibers}, such fibers are biholomorphic to $\mathcal{G}_{\pi_3(a,b)}\times \C{\frac{n(n-1)}{2}}$ where
$$ \mathcal{G}_{\pi_3(a,b)} = \{\vec{Z}_2\in (\Cm)^2: Z_2Z_1e_n=0\}.$$
Observe that this variety has two irreducible components $ A_1 = \{\vec{Z}_2\in \mathcal{G}_{\pi_3(a,e_n^T)}: Z_1e_n=0\}$ and 
$ A_2 = \{\vec{Z}_2\in \mathcal{G}_{\pi_3(a,e_n^T)}: \det(Z_2)=0\}.$ This proves that singular fibers $\mathcal{F}^3_{a,e_n^T}$ have two irreducible components. Since the intersection of these components equals the singularity set $S_3$, $\mathcal{F}^3_{a,e_n^T}$ is connected.
\end{proof}

\begin{theorem}\label{theorem:connectedFibers}
The fibers $\mathcal{F}^K_y$ are connected for $K\geq 3$. Moreover, the smooth part of the singular fibers is connected for $K\geq 5$.
\end{theorem}

\begin{proof}
We prove this theorem by induction on $K$. Note that we've shown the base case $K=3$ in the previous lemma. Let $K\geq 4$ and assume that the $(K-1)$ fibers $\mathcal{F}^{K-1}_y$ are connected. Recall that
\begin{equation}\label{equation:Fibration}
\mathcal{F}^K_y = \bigcup_{Z\in \Cm} \mathcal{F}^{K-1}_{M_K(Z)y}.
\end{equation}
We will now introduce the following auxiliary function. Let $\rho:\mathcal{F}^K_y \to \Cm$ denote the restriction of the projection $(\vec{Z}_{K-1},Z_K)\mapsto Z_K$ to the fiber $\mathcal{F}^K_y$. We'll show some useful facts.
\begin{enumerate}[label=(\roman*)]
\item $\rho$ is surjective: Observe that $\vec{Z}_K\in \mathcal{F}^K_y$ if and only if $\vec{Z}_{K-1}\in \mathcal{F}^{K-1}_{M_K(-Z_K)y}$, by (\ref{equation:Fibration}).
Hence the $\rho$-fibers are given by $\rho^{-1}(Z_K^*) = \mathcal{F}^{K-1}_{M_K(-Z_K^*)y}$. The $(K-1)$-fibers $\mathcal{F}^{K-1}_{M_K(Z)y}$ are non-empty for all $Z\in \Cm$, by Theorem \ref{theorem:PhiSurjective}.
\item $\rho$ is submersive in $\vec{Z}_K$ if $\vec{Z}_{K-1}\not\in S_{K-1}$: Observe that $T_{\vec{Z}_K}(\mathcal{F}^K_y\cap S_K^c) = \ker J\Phi_K(\vec{Z}_K)$. By assumption, the Jacobian $J\Phi_{K-1}(\vec{Z}_{K-1})$ is surjective. Given $W_K\in \Cm$, we find $\vec{W}_{K-1}\in (\Cm)^{K-1}$ such that $J\Phi_{K-1}(\vec{Z}_{K-1})\vec{W}_{K-1} = -M_K(-Z_K)A_K(\Phi_{K-1}(\vec{Z}_{K-1}))W_K$. By the recursive formula of the Jacobian (see Corollary \ref{cor:JacobianRecursive}), we get
$$ J\Phi_K(\vec{Z}_K)\begin{pmatrix}
\vec{W}_{K-1}\\ W_K
\end{pmatrix} = M_K(Z_K)J\Phi_{K-1}(\vec{Z}_{K-1})\vec{W}_{K-1} + A_K(\Phi_{K-1}(\vec{Z}_{K-1}))W_K = 0$$
and $d\rho_{\vec{Z}_K}(\vec{W}_{K-1},W_K)=W_K$. 
\item Each connected component $A\subset \mathcal{F}^K_y$ is $\rho$-saturated, that is, $A=\rho^{-1}(\rho(A))$: It is enough to show ``$\supset$'', by definition of the preimage. Each $\rho$-fiber is connected and therefore we have either $\rho^{-1}(b)\subset A $ or $\rho^{-1}(b)\cap A=\emptyset$. 
\item $\rho(A)$ is open for each connected component $A\subset \mathcal{F}^K_y$: Given a point $b\in \rho(A)$ we find a regular point $a\in \rho^{-1}(b)$. To see this, notice that in every $(K-1)$ fiber we find points with $Z_1e_n\neq 0$, by the previous lemma and by (\ref{equation:Fibration}). Submersivity is a local property, hence there exists an open neighborhood $U\subset A$ of $a$ in which $\rho$ is submersive. Furthermore, $U$ is mapped openly, that is, $\rho(U)$ is an open neighborhood of $b$ in $\rho(A)$.
\end{enumerate}

\noindent We can write $\mathcal{F}^K_y$ as a disjoint union of connected components $\dot{\bigcup}_{i\in I} A_i$. Then
\begin{align*}
\Cm \underset{(i)}{=} \rho(\mathcal{F}^K_y) \underset{(iii)}{=} \dot{\bigcup}_{i\in I} \rho(A_i)
\end{align*}
can be written as the disjoint union of open sets. Since $\Cm$ is connected, we conclude that $\mathcal{F}^K_y$ has to be connected too.

It remains to show, that the smooth part of the fibers $\kfiber$ is connected for $K\geq 5$. Let's start with the case $K=2k+1$. Then the singular fibers are $\mathcal{F}^K_{a,e_n^T}$, by Lemma \ref{lemma:singularFibers}. If we now also note (\ref{equation:Fibration}), then the singularities are all in the subfiber $\mathcal{F}^{2k}_{0,e_n^T}$. Since $\mathcal{F}^{2k}_{0,e_n^T}\subset \mathcal{F}^{2k+1}_{a,e_n^T}$ has codimension $n$, the complement $U:= \mathcal{F}^{2k+1}_{a,e_n^T}\setminus \mathcal{F}^{2k}_{0,e_n^T}$ is connected. If we consider a smooth point $p\in \mathcal{F}^{2k}_{0,e_n^T}$, then each open neighborhood of $p$ intersects with $U$. Hence $p$ has to be in the same connected component as $U$. This proves that the smooth part $\mathcal{F}^K_{a,e_n^T}\setminus \mathrm{Sing}(\mathcal{F}^K_{a,e_n^T})$ is connected.

In the case $K=2k$, observe that $\mathcal{F}^{2k}_{0,e_n^T}$ is the only singular fiber. Furthermore, there holds
$$ \mathcal{F}^{2k}_{0,e_n^T}\setminus \mathrm{Sing}(\mathcal{F}^{2k}_{0,e_n^T}) = \left(\mathcal{F}^{2k-1}_{0,e_n^T}\setminus \mathrm{Sing}(\mathcal{F}^{2k-1}_{0,e_n^T})\right)\times \Cm,$$
and this is connected by the induction hypothesis. This completes the proof of the theorem.
\end{proof}

\begin{theorem}\label{theorem:connectedFibersSecond}
The fibers of the submersion $\Phi_K:\mathcal{W}_K\to \CwO$ are connected for $K\geq 3$.
\end{theorem}
\begin{proof}
We want to show that the intersection $\mathcal{F}^K_y\cap \mathcal{W}_K$ is connected. Smooth fibers $\mathcal{F}^K_y$ are contained in $\mathcal{W}_K$, hence we only need to consider the case when $\mathcal{F}^K_y$ is a singular fiber. For $K=3$ we can apply the strategy from Lemma \ref{lemma:3fibersAreConnected} for smooth fibers and cover $\mathcal{F}^3_{a,e_n^T}\cap \mathcal{W}_3$ by $n$ intersection connected charts.

Next, assume the claim to be true for $K-1=2k-1$. For
$$ \mathcal{R}:=\{\vec{Z}_{2k-1}\in (\Cm)^{2k-1}: Z_1e_n=...=Z_{2k-3}e_n=0, Z_{2k-1}e_n\neq 0\},$$
we can rewrite
$$ \mathcal{W}_{2k} = \mathcal{W}_{2k-1}\times \Cm \dot{\cup} \mathcal{R}\times \Cm.$$
Observe that $\mathcal{F}^{2k-1}_{0,e_n^T}\cap \mathcal{R}=\emptyset$, by definition. Therefore, the non-generic singular fiber $\mathcal{F}^{2k}_{0,e_n^T}$ satisfies
$$ \mathcal{F}^{2k}\cap \mathcal{W}_{2k} = \left( \mathcal{F}^{2k-1}_{0,e_n^T}\cap \mathcal{W}_{2k-1}\right) \times \Cm,$$
which is connected by the induction hypothesis.

Finally, we assume the claim to be true for $K-1=2k$. Let $\rho:\mathcal{F}^K_{a,e_n^T}\cap \mathcal{W}_K\to \Cm$ be the restriction of the projection $(\vec{Z}_{K-1},Z_K)\mapsto Z_K$ to $\mathcal{F}^K_{a,e_n^T}\cap \mathcal{W}_K$. The fiber $\rho^{-1}(Z_K)$ is given by $\mathcal{F}^{K-1}_{M_K(-Z_K)y}\cap \mathcal{W}_{K-1}$. From here we can argue as in the proof of Theorem \ref{theorem:connectedFibers}.
\end{proof}

\begin{lemma}
Each generic fiber $\mathcal{F}^3_y$ containes points $\vec{Z}_3$ with $Q_f^1(\vec{Z}_3)=z_{1,n1}\cdots z_{1,nn}\neq 0$. Moreover, each generic fiber $\mathcal{F}^K_y$ containes points $\vec{Z}_K$ with
$${Q_f^{K-2}(\vec{Z}_K)Q_s^{K-2}(\vec{Z}_K)= P^{K-2}_1(\vec{Z}_{K-2})\cdots P^{K-2}_n(\vec{Z}_{K-2})P^{K-2}_{n+1}(\vec{Z}_{K-2})\cdots P^{K-2}_{2n}(\vec{Z}_{K-2})\neq 0}.$$
\end{lemma}
\begin{proof}
Let's start with the case $K=3$. We consider a generic fiber $\mathcal{F}^3_{a,b}$, $b\neq 0$. Recall that $\vec{Z}_3\in \mathcal{F}^3_{a,b}$ if and only if $\vec{Z}_2\in \mathcal{F}^2_{a-bZ_3,b}$ by the recursive formula (\ref{RecursiveFormula}). Since we assume $b\neq 0$, there is a symmetric matrix $Z_3\in \Cm$ with $a-Z_3b=(1,...,1)^T$, by Lemma \ref{lemma:symmetricSurjective}. The fiber $\mathcal{F}^2_{(1,...,1),b}\neq \emptyset$ is non-empty by Theorem \ref{theorem:PhiSurjective}. Furthermore, there holds $P_f^2\equiv P_f^1$, by the recursive formula (\ref{RecursiveFormula}). This implies $Q_f^1\equiv 1$ on the fiber $\mathcal{F}^2_{(1,...,1),b}\subset \mathcal{F}^3_{a,b}$ and we're done with the case $K=3$. 

We now consider the case $K\geq 4$. The fiber $\mathcal{F}^{K-2}_{(1,...,1),(1,...,1)}\neq \emptyset$ is non-empty, by Theorem \ref{theorem:PhiSurjective}, and there clearly holds $Q_f^{K-2}Q_s^{K-2}\equiv 1$ on the whole fiber. Therefore it is enough to show, that this fiber sits inside each generic $K$-fiber, that is, $\mathcal{F}^{K-2}_{(1,...,1),(1,...,1)}\subset\mathcal{F}^K_y$.
We prove this claim only for $K=2k+1$, since the proof is symmetric for $K=2k$. In this case, a point $y=(a,b)\in Y_g^K$ in the generic stratum satisfies $b\neq 0$. Let's write $v:=(1,...,1)^T$. Then the fiber $\mathcal{F}^{K-2}_{v,v}$ sits inside $\kfiber$ if and only if we find symmetric matrices $Z_{K-1},Z_K\in \Cm$ with
$$ \begin{pmatrix}
a\\ b 
\end{pmatrix} = \Eu{Z_K}\El{Z_{K-1}}\begin{pmatrix}
v\\ v
\end{pmatrix},$$
or equivalently,
$$ \begin{pmatrix}
a - Z_Kb\\ b
\end{pmatrix} = \begin{pmatrix}
v\\ Z_{K-1}v+v
\end{pmatrix}.$$
We can split up this system of equations into two independent systems $a-Z_Kb=v$ and $b=Z_{K-1}v+v$. An application of Lemma \ref{lemma:symmetricSurjective} to both systems yields the existence of such matrices $Z_{K-1}$ and $Z_K$, since $b\neq 0$ and $v\neq 0$. This completes the proof.
\end{proof}

\subsection{Construction of stratified sprays}
The most convenient way to construct a dominating spray associated to a submersion is to define some finite composition of flow maps of complete fiber-preserving vector fields. The collection of complete vector fields in Proposition \ref{lemma:completeFields} does not span the tangent space of the fibers $\fiber$ in every point. This is the reason why these are a priori not sufficient to define a dominating spray. However, we can enlarge this collection until it spans the tangent space in a \textit{sufficiently large} set of points. In the following section we discuss the meaning of 'sufficiently large' and the mathematical details for the enlargement. 
\subsubsection{Strategy for the construction}
Let $M$ be a Stein manifold and let $\complete$ denote the set of $\mathbb{C}$-complete holomorphic vector fields on $M$. For a vector field $V\in \complete$ its corresponding flow $\alpha^V_t$, $t\in \mathbb{C}$, is a one-parameter subgroup of $\mathrm{Aut}_{hol}(M)$, the holomorphic automorphism group on $M$. For a set $A\subset \complete$ of complete holomorphic vector fields on $M$  we define
$$ S_A:= \bigcup_{V\in A} \{\alpha^V_t:t\in \mathbb{C}\}\subset \mathrm{Aut}_{hol}(M).$$
Let $G_A:=\langle S_A \rangle$ denote the subgroup generated by $S_A$. Furthermore, we denote the pull-back of an automorphism $\alpha$ by $\alpha^*$. 
\begin{defi}
For a set $A\subset \complete$ of complete holomorphic vector fields on $M$, define
$$ \Gamma(A):=\{\alpha^*X: \alpha \in G_A, X\in A\}$$
the \textit{collection of complete holomorphic vector fields generated by $A$,}
and for an open set $Y\subset M$ define
$$ C_A(Y):=\{\alpha(y): \alpha \in G_A, y\in Y\}$$
the $G_A$-closure of $Y$.
\end{defi}
Some basic properties follow directly from the definition. Let $Y\subset M$ be an open set. Then $C_A(Y)$ is the smallest set containing $Y$, which is invariant under $G_A$. Moreover, $C_A(Y)$ is open in $M$, hence, for a fixed collection $A\subset \complete$, $C_A$ can be interpreted as a map $C_A:\mathcal{T}_M\to \mathcal{T}_M$, where $\mathcal{T}_M$ denotes the natural topology on $M$. In particular, $C_A$ satisfies the conditions of a topological closure operator.

The following lemma describes another basic property. 
\begin{lemma}\label{lemma:aboutTheClosure}
Let $A,B\subset \complete$ be finite collections of complete holomorphic vector fields on $M$ with $A\subset B\subset \Gamma(A)$. Then $C_A(X) = C_B(X)$ for all open subsets $X\subset M$.
\end{lemma}
\begin{proof}
We're going to prove that $G_A=G_B$. To do this, it suffices to show that $G_B\subset G_A$, since the reverse inclusion trivially holds by assumption $A\subset B$.

At first, we consider $V\in B$. There is $\beta\in G_A$ and $W\in A$ with $V=\beta^*W$, since $B\subset \Gamma(A)$. The flow of $V$ satisfies
$$ \alpha_t^V = \alpha_t^{\beta^*W} = \beta \circ \alpha_t^W\circ \beta^{-1} \in G_A,$$
since $\beta, \alpha_t^W\in G_A$. 

In a next step, let $\beta\in G_B$ be any automorphism. By definition of $G_B$, there are vector fields  $V_{i_1},...,V_{i_m}\in B$ and times $t_1,...,t_m\in \mathbb{C}$ with $\beta = \alpha_{t_1}^{V_{i_1}}\circ \cdots \circ \alpha_{t_m}^{V_{i_m}}$. From the previous step we know that each $\alpha_{t_j}^{V_{i_j}}\in G_A$ is an automorphism in $G_A$ and hence so is $\beta$, i.e. $\beta\in G_A$ and this proves the claim. 
\end{proof}

Here comes the main step for the enlargement of the collection of complete fields. 
\begin{lemma}\label{lemma:noInvariantSetInClosure}
Let $M$ be a Stein manifold, $X\subset M$ an open subset and $A\subset \mathcal{VC}_{hol}(M)$ a finite set of $\mathbb{C}$-complete holomorphic vector fields on $M$ which spans the tangent bundle $TX$. Then there is a finite subset $A\subset B\subset \Gamma(A)$ which spans the tangent bundle $TC_A(X)$.
\end{lemma}
\begin{proof}
For each field $V\in A$ let $\alpha^V_t,t\in \mathbb{C},$ be the corresponding vector flow. Let $N_0$ be the set of points $x\in C_A(X)$ where the fields of $A$ don't span the tangent space $T_xM$. This is an analytic subset $N_0\subset C_A(X)\setminus X$. Further, we define
$$ N_k:=\{x\in N_{k-1}:\alpha^V_t(x)\in N_{k-1},\forall V\in A, \forall t\in \mathbb{C}\}, \quad k\geq 1.$$
Let $k\geq 0$ be arbitrary but fixed. Then $N_k$ has at most countably many connected components. Denote by $A_i^k, i\in I_k$, those connected components of $N_k$ which aren't entirely contained in $N_{k+1}$ and let $a_k:=\max_{i\in I_k} \dim A_i^k$ be the maximal dimension of them.  Choose a point $x_i^k\in A_i^k\cap N_k\setminus N_{k+1},i\in I_k,$ of each such component. By definition of the sets $N_k$ and $N_{k+1}$, there is a field $V\in A$ for each point $x$ in the sequence $\{x_i^k\}_{i\in I}$, such that $\alpha^V_t(x)\not\in N_k$ for some $t\in \mathbb{C}$. For $V\in A$ define
$$ u_V^k:=\{x\in \{x_i^k\}_{i\in I_k}: \alpha^V_t(x)\not\in N_k,\text{ for some } t\in \mathbb{C}\}.$$
Then this yields 
$$ \{x_i^k\}_{i\in I_k}= \bigcup_{V\in A} u_V^k.$$
For each point $x\in u_V^k$ the set $\{t\in \mathbb{C}: \alpha^V_t(x)\in N_k\}$ is discrete and hence
$$\bigcup_{x\in u_V^k} \{t\in \mathbb{C}: \alpha^V_t(x)\in N_k\}$$
is a meagre set in $\mathbb{C}$. This implies the existence of a time $t_V^k\in \mathbb{C}$ such that $\alpha^V_{t_V^k}(x)\not\in N_k$ for all $x\in u_V^k$.
Define
$$ \tilde{N}_{k+1}:=\{x\in N_k: \alpha^V_{t_V^k}(x)\in N_k,\forall V\in A\}.$$
Clearly, there holds
$$ N_{k+1}\subset \tilde{N}_{k+1}\subset N_k.$$
The set $\tilde{N}_{k+1}$ has at most countably many connected components. Let $\tilde{a}_{k+1}$ denote the maximal dimension of them. There holds $a_{k+1}\leq \tilde{a}_{k+1}<a_k$ by construction. Since $M$ is finite dimensional, this implies that there is $L\in \mathbb{N}$ such that $N_k=\emptyset$ for all $k>L$.

Let $B_k$ be the set of pullbacks $(\alpha_{t_V^k}^V)^*(W)$ for  $V,W\in A$ and set
$$ B:=\bigcup_{k\geq 0}^L A\cup B_k.$$
Again by construction, the collection $B\subset \Gamma(A)$ is a finite set of $\mathbb{C}$-complete holomorphic vector fields on $M$ that spans the tangent bundle $TC_A(X)$. Moreover, $C_A(X)=C_B(X)$ by the previous lemma, which implies that $C_A(X)$ is invariant under the flows of $B$. This finishes the proof.
\end{proof}

In a next step, we want to adapt this argument to our setting so that we can apply it to every fiber simultaneously, so to speak. Moreover, we are ready to define a dominating spray.
\begin{lemma}\label{cor:spanningFields}
Let $M$ be a Stein manifold, $\pi:M\to Y$ a holomorphic mapping and $X\subset M$ a (connected) open subset such that the restriction $\pi|_{X}:X\to Y$ is a surjective submersion with connected fibers. 

Suppose that there is a finite set $A\subset \mathcal{VC}_{hol}(M)$ of $\mathbb{C}$-complete fiber-preserving holomorphic vector fields on $M$ which spans the tangent bundle $T(M_y\cap X)$ of each fiber $M_y:=\pi^{-1}(y)$. Then there is a finite set $B\subset \Gamma(A)$ of $\mathbb{C}$-complete fiber-preserving holomorphic vector fields which spans the tangent bundle $T(C_A(M_y\cap X))$ of each fiber $M_y$. In particular, the surjective submersion $\pi|_{C_A(X)}$ admits a spray. 
\end{lemma}

\begin{proof}
We can proceed similarly as in the previous lemma to obtain a finite collection $B\subset \Gamma(A)$ which spans the tangent bundle $T(C_A(M_y\cap X))$ of each fiber $M_y$. This follows from the assumption that all fields in $A$ are fiber-preserving.

Moreover, the map $\pi|_X:X\to Y$ is a surjective submersion, hence $\pi|_{C_A(X)}:C_A(X)\to Y$ is also a surjective submersion. Write $B=\{W_1,...,W_L\}$. Then the map $s:C_A(X)\times \C{L}\to C_A(X)$ given by
$$ s(z,t_1,...,t_L) = \alpha_{t_1}^{W_1}\circ \cdots \circ \alpha_{t_L}^{W_L}(z)$$
is a dominating spray associated to $\pi|_{C_A(X)}$, since $C_A(X)$ is invariant with respect to the flows of $W_1,...,W_L$ by Lemma \ref{lemma:aboutTheClosure}. 
\end{proof}

This corollary shows that we can relax the assumptions of the previous lemma and we will apply it in this form later.
\begin{cor}\label{cor:shortCutSpanningTheorem} Let $M$ be a Stein manifold, $\pi:M\to Y$ a holomorphic map and $X\subset M$ a connected open set such that the restriction $\pi|_X$ is a surjective submersion with connected fibers. Furthermore, we are given an open subset $W\subset X$ and a finite collection $A\subset \complete$ of complete fiber-preserving holomorphic vector fields on $M$ which spans the tangent bundle $T(M_y\cap W)$ of each fiber $M_y$.

Suppose that there is a finite collection $B\subset \complete$ of complete fiber-preserving holomorphic vector fields on $M$ such that $X\subset C_B(W)$ and $A\subset B$. Then $\pi|_{C_{B}(W)}$ admits a spray.
\end{cor}

\begin{proof}
Since $B$ contains $A$ by assumption, $B$ spans the tangent bundle $T(M_y\cap W)$ of each fiber, hence there is a finite collection $\tilde{B}\subset \Gamma(B)$ which spans the tangent bundle $T(M_y\cap C_B(W))$ of each fiber. Since $X\subset C_B(W)$ by assumption, $\tilde{B}$ spans the tangent bundle $T(M_y\cap X)$ for each fiber. Now, we apply Lemma \ref{cor:spanningFields} to finish the proof. 
\end{proof}

\begin{Rem}
Given a finite collection $A\subset \complete$ and two open sets $X,W\subset M$ with $W\subset X$, there holds $X\subset C_A(W)$ if and only if $X\setminus W \subset C_A(W)$. This follows from the fact that $C_A$ is extensive.
\end{Rem}

This lemma will help us decide if we have $X\subset C_A(X)$ for a suitable finite collection $A$. 
\begin{lemma}\label{lemma:leavingASubset}
Let $M$ be a Stein manifold and $N\subset M$ an analytic subvariety given by
$$ N:=\{x\in M: f(x)=0\}$$
for some holomorphic mapping $f:M\to \mathbb{C}$. Assume that there are complete holomorphic vector fields $V_1,...,V_k$ on $M$ (and we denote the respective flows by $\alpha^1_t,...,\alpha^k_t$) such that $$ V_{i_n}\circ \cdots \circ V_{i_1}(f(x))\neq 0, \forall x\in N,$$
for some finite sequence $\{i_1,...,i_n\}\subset \{1,...,k\}$. Then there is a composition of the flows $\alpha^1_t,...,\alpha^k_t$ which leaves the subvariety $N$. More precisely, there holds
$$\{\alpha^{i_n}_{t_n}\circ \cdots \circ \alpha^{i_1}_{t_1}(x):t_1,...,t_n\in \mathbb{C}\}\not\subset N,\quad \forall x\in N.$$
\end{lemma}
\begin{proof}
Define the subvariety $ N_1:=\{x\in N: V_{i_1}(f(x))=0\}$. The orbit of $\alpha_t^{i_1}$ through points of $N\setminus N_1$ is leaving $N$. Next, define the subvariety $N_2:=\{x\in N_1: V_{i_2}(V_{i_1}(f(x)))=0\}$. Then the orbit of $\alpha_t^{i_2}$ through points of $N_1\setminus N_2$ is leaving $N_1$. We proceed inductively and set $N_l:=\{x\in N_{l-1}: V_{i_l}\circ \cdots \circ V_{i_1}(f(x))=0\}$. Then the orbit of $\alpha_t^{i_l}$ through points of $N_{l-1}\setminus N_l$ is leaving $N_{l-1}$. This is true for all $1\leq l \leq n-1$, which implies that an invariant set with respect to the fields $V_1,...,V_k$ has to be in the set $N_n$. There holds $V_{i_n}\circ\cdots \circ V_{i_1}(f(x))\neq 0$, by assumption. Hence the set $N_n$ is empty and there is no invariant set in $N$ with respect to the fields $V_1,...,V_k$. This proves the claim.
\end{proof}

\subsubsection{Construction of a spray over the generic stratum}
We begin this subsection with a definition. Recall the collection $\mathcal{V}_K$ of $\mathbb{C}$-complete fiber-preserving holomorphic vector fields from Definition \ref{def:liftedVectorFields}. These collections are defined for $K\geq 3$. According to Lemma \ref{lemma:genericFibers}, there are $n$ meromorphic mappings $\psi_1,...,\psi_n$ such that each generic fiber $\mathcal{F}^2_y$ is biholomorphic to $\mathcal{G}_{\pi_2(y)}\cong \C{n(n-1)}$ via $\psi_j$ for some $1\leq j\leq n$. Let $\deriv{x_i}$, $1\leq i\leq n(n-1)$ denote the standard vector fields on $\C{n(n-1)}$. Then
$$\mathcal{V}_2=\bigcup_{j=1}^n \left\lbrace (\psi_j)_*\left(\deriv{x_i}\right): 1\leq i\leq n(n-1) \right\rbrace$$
is the collection obtained by pushing forward the standard vector fields via $\psi_1,...,\psi_n$. 
\begin{defi}
For $K\geq 2$, we define $\mathcal{Q}_K:=\Gamma(\mathcal{V}_K)$ the collection of $\mathbb{C}$-complete $\fiber$-fiber-preserving holomorphic vector fields generated by $\mathcal{V}_K$.
\end{defi}

Recall the set $\mathcal{W}_K$ which is given by
$$ \mathcal{W}_K := \left\lbrace\vec{Z}_K\in (\Cm)^K: Z_{2i-1}e_n\neq 0,\text{ for some } 1\leq i \leq \big\lceil \tfrac{K-1}{2} \big\rceil\right\rbrace,$$
where $\lceil x \rceil$ is the \textit{ceiling function}, which maps $x$ to the least integer greater than or equal to $x$. The set $\mathcal{W}_K$ is open and connected. 

The following result, from now on we will call it \textit{Spanning theorem}, is cruical for the construction of a spray over the generic stratum. 
\begin{theorem}[Spanning theorem]\label{theorem:spanTangentSpace} Let $K\geq 2$. Then there is a finite set $A_K\subset\mathcal{Q}_K$ of $\mathbb{C}$-complete fiber-preserving holomorphic vector fields on $(\Cm)^K$ spanning the tangent bundle $T(\fiber\cap \mathcal{W}_K)$ of every generic fiber $\fiber$. In particular, there holds $A_{K}\subset A_{K+1}$, when considering $A_{K}$ as a subset of $\mathcal{Q}_{K+1}$.
\end{theorem}

Let $K\geq 3$. Then the mapping $\Phi_K|_{\Wk{K}}:\Wk{K}\to \CwO$ is a surjective submersion with connected fibers $\fiber$ (see Theorem \ref{theorem:PhiSurjective} and Theorem \ref{theorem:connectedFibersSecond}). By the \spanning, the conditions of Lemma \ref{cor:spanningFields} are satisfied. Hence the submersion
$$ \Phi_K:C_{A_K}(\mathcal{W}_K)|_{Y_{g}^K}\to Y_{g}^K$$
over the generic stratum $Y_{g}^K\subset \CwO$ admits a spray.

\begin{Rem}[Application of the \spanning]
Recall that, by Lemma \ref{lemma:singularFibers}, a fiber $\fiber$ contains singularities if and only if $y=(y_1,...,y_n,0,...,0,1)$ for $y_1,...,y_n\in \mathbb{C}$. In particular, the intersection of singular fibers and $\Wk{K}^c$, the complement of $\Wk{K}$, is non-empty. We do not know, whether the collection $A_K$ from the \spanning\ can be supplemented by finitely many fields in $\Qk{K}$ so that $A_K$ spans the tangent space $T_{\vec{Z}_K}\fiber$ for a smooth point $\vec{Z}_K$ in $\Wk{K}^c$. However, since smooth fibers $\fiber$ are completely contained in $\Wk{K}$, the collection $A_K$ spans the tangent bundle $T\fiber$ of each smooth generic fiber $\fiber$. 
\end{Rem}

\subsubsection{Construction of a spray over the non-generic stratum}
In this subsection, we show that the submersion
$$ \Phi_K:C_{A_K}(\mathcal{W}_K)|_{Y_{ng}^K}\to Y_{ng}^K$$
over the non-generic stratum $Y_{ng}^K\subset \CwO$ admits a spray.

We need the following result. 
\begin{lemma}\label{lemma:closureNonGeneric}
Let $K\geq 3$ and let $A\subset \mathcal{Q}_K$ be a finite collection of complete holomorphic fiber preserving vector fields. Then there holds
$$ C_A(\mathcal{W}_K\cap \fiber) = C_A(\mathcal{W}_{K-1}\cap \mathcal{F}^{K-1}_y)\times \Cm$$
for each non-generic fiber $\fiber$.
\end{lemma}
\begin{proof}
In a first step, we prove that
\begin{align}\label{equation:nongeneric}
\mathcal{W}_K\cap \fiber = (\mathcal{W}_{K-1}\cap \mathcal{F}^{K-1}_y)\times \Cm
\end{align}
for each non-generic fiber $\fiber$ and $K\geq 3$. From the definition of the set $\mathcal{W}_K$ we directly conclude $\mathcal{W}_{2k+1}=\mathcal{W}_{2k}\times \Cm$. Since each non-generic fiber satisfies $\fiber = \mathcal{F}^{K-1}_y \times \Cm$ by Lemma \ref{lemma:nonGenericFibers}, equation (\ref{equation:nongeneric}) follows for $K=2k+1$. 

Consider now $K=2k$ even and a non-generic fiber $\fiber$, that is, $y=(0,b)^T$ for some non-zero $b \in \C{n}$. 
Further, observe that $\mathcal{W}_{2k} = \mathcal{W}_{2k-1}\times \Cm \cup \mathcal{R},$
where
$$ \mathcal{R}:= \{\vec{Z}_{2k}\in (\Cm)^{2k}: \pi_n(Z_{2i-1})=0, 1\leq i\leq k-1, \pi_n(Z_{2k-1})\neq 0\}.$$
It suffices to show that $\fiber\cap \mathcal{R}=\emptyset$, in order to prove (\ref{equation:nongeneric}). Assume for contradiction there is $\vec{Z}_{2k}\in \fiber\cap \mathcal{R}$. Lemma \ref{lemma:propertyOne} implies $\Phi_{2k-2}(Z_{2k-2})=e_{2n}$ and $\Phi_{2k-1}(\vec{Z}_{2k-1}) = \begin{pmatrix}
0&b
\end{pmatrix}^T$ follows by Lemma \ref{lemma:nonGenericFibers}. According to the recursive formula (\ref{RecursiveFormula}), we obtain
$$ \begin{pmatrix}
0\\ b
\end{pmatrix} = \Eu{Z_{2k-1}}e_{2n} = \begin{pmatrix}
\pi_n(Z_{2k-1})\\ e_n
\end{pmatrix}$$
contradicting assumption $\pi_n(Z_{2k-1})\neq 0$. This proves equation (\ref{equation:nongeneric}).

In a second step, we show that, over the non-generic stratum, none of the vector fields $V\in \mathcal{Q}_K$ flows in a new direction. It suffices to prove the claim for the generating set $\mathcal{V}_K$. Note that there is nothing to show for vector fields in $\mathcal{V}_{K-1}$. And fields in $\mathcal{V}_K\setminus \mathcal{V}_{K-1}$ vanish over the non-generic stratum, by Lemma \ref{lemma:liftingVectorFields}. This proves the claim. In particular, the vector flow $\alpha_t^V$ of $V\in \mathcal{Q}_K$ fixes the new directions, i.e. $\alpha_t^V(\vec{Z}_{K-1},Z_K) = (f(\vec{Z}_{K-1},Z_K), Z_K)$ for some suitable holomorphic function $f$.

In the last step, we apply equation (\ref{equation:nongeneric}) and step two. We get
\begin{align*}
C_A(\mathcal{W}_K\cap \fiber) = C_A(\mathcal{W}_{K-1}\cap \mathcal{F}^{K-1}_y\times \Cm) = C_A(\mathcal{W}_{K-1}\cap \mathcal{F}^{K-1}_y)\times \Cm.
\end{align*}
This finishes the proof.
\end{proof}

\begin{lemma}
Let $K\geq 3$ and $A_K\subset \mathcal{Q}_K$ be the finite collection provided by the \spanning such that the tangent bundle $T(\fiber \cap \mathcal{W}_K)$ of every generic fiber $\fiber$ is spanned by $A_K$. Then the restricted submersion
$$ \Phi_K:C_{A_K}(\mathcal{W}_K)|_{Y_{ng}^K}\to Y_{ng}^K$$
over the non-generic stratum $Y_{ng}^K\subset \CwO$ admits a spray. 
\end{lemma}

\begin{proof}
Let $A_{K-1}$ and $A_K$ be the collections from the \spanning\ and let $\fiber$ be a non-generic fiber. 
There holds $\fiber = \mathcal{F}^{K-1}_y\times \Cm$ by Lemma \ref{lemma:nonGenericFibers}, where $\mathcal{F}^{K-1}_y$ is a generic $(K-1)$-fiber. According to Lemma \ref{lemma:nonGenericFibers}, the vector fields from $A_K\setminus A_{K-1}$ vanish over $\fiber$ and we get
$$ C_{A_K}(\mathcal{W}_K\cap \fiber) = C_{A_{K-1}}(\mathcal{W}_{K-1}\cap \mathcal{F}^{K-1}_y)\times \Cm.$$
Collection $A_{K-1}$ spans the tangent bundle $T(\mathcal{W}_{K-1}\cap \mathcal{F}^{K-1}_y)$ for every generic fiber $\mathcal{F}^{K-1}_y$, by the \spanning. By Theorem \ref{cor:spanningFields}, there is a finite collection $B\subset \Gamma(A_{K-1})$ spanning the tangent bundle $T(C_{A_{K-1}}(\mathcal{W}_{K-1}\cap \mathcal{F}^{K-1}_y))$ for every generic fiber $\mathcal{F}^{K-1}_y$. We add the vector fields
$$ \left( \frac{\partial}{\partial z_{K,ij}}\right)_{1\leq i\leq j\leq n}$$
which span $\Cm$ to the collection $B$. This new collection, let's call it $\tilde{B}$, spans the tangent bundle ${T(C_{A_K}(\mathcal{W}_K\cap \fiber))}$ for every non-generic fiber $\fiber$. Similar as in Theorem \ref{cor:spanningFields}, we can use the vector flows $\alpha_t^V$, $V\in \tilde{B}$, to construct a dominating spray associated to the submersion $\Phi_K:C_{A_K}(\mathcal{W}_K)|_{Y_{ng}^K}\to Y_{ng}^K$.
\end{proof}
\subsection{The Spanning theorem}
In this subsection we prove the \spanning, which we will do by induction on the number of factors $K$. As it turns out, for various reasons, it requires several base steps before we get to the actual induction step. It makes sense to explain the proof strategy continuously. So let's first introduce or recall some notations and then we start with the first base step, $K=2$. 

We write $\Phi_K=(P_1^K,...,P_{2n}^K)^T$ as well as ${P_f^K=(P_1^K,...,P_n^K)^T}$ and ${P_s^K=(P_{n+1}^K,...,P_{2n}^K)^T}$. Similarly, we write ${Q_f^K=P_1^K\cdots P_n^K}$ and ${Q_s^K=P_{n+1}^K\cdots P_{2n}^K}$. For $y\in \CwO$ we sometimes write $y=(a,b)$ with $a=(a_1,...,a_n)$ and $b=(b_1,...,b_n)$.

\begin{lemma}[Spanning theorem for $K=2$]\label{lemma:spanningTheoremForCaseTwo}
Let $\mathcal{F}^2_y$ be a generic fiber, i.e. $\pi_2(y)_{j^*}=P^1_{j^*}\neq 0$ for some $1\leq j^*\leq n$. Then the collection
$$ A_2 = \left\lbrace (P^1_{j^*})^2\Deriv{z}{1,ij}: 1\leq i\leq j <n\right\rbrace \cup \left\lbrace \gamma^2_{ij,j^*}:{1\leq i\leq j\leq n}, { i\neq j^*}, {j\neq j^*} \right\rbrace  $$
consist of complete holomorphic vector fields which are fiber-preserving for $\Phi_2$. Moreover, $A_2$ spans the tangent bundle $T\mathcal{F}^2_y$.
\end{lemma}
\begin{proof}
According to Lemma \ref{lemma:genericFibers}, we are able to express the variables $z_{2,1j^*},...,z_{2,nj^*}$. This gives us a meromorphic mapping $\psi_{j^*}$ which maps $\mathcal{G}_{\pi_2(y)}\cong \C{n(n-1)}$ biholomorphic to $\mathcal{F}^2_y$. In particular, the vector fields $\Deriv{z}{1,ij}, 1\leq i\leq j<n$ and $\Deriv{z}{2,ij}, {1\leq i\leq j\leq n}, {i\neq j^*},{j\neq j^*}$ are complete holomorphic and tangential to $\mathcal{G}_{\pi_2(y)}$. Moreover, they span the tangent bundle $T\mathcal{G}_{\pi_2(y)}$. The collection $A_2$ is obtained by computing the push-forwards with respect to the mapping $\psi_{j^*}$. 
\end{proof}
\subsubsection{Preparation and explanation of the induction step}
Recall that each $K$-fiber $\mathcal{F}^K_y$ can be written as a union of $(K-1)$-fibers, that is,
$$ \mathcal{F}^K_y = \bigcup_{Z\in \Cm}\mathcal{F}^{K-1}_{M_K(Z)y}.$$
Equivalently, ${\vec{Z}_K=(\vec{Z}_{K-1},Z_K)\in \mathcal{F}^K_y}$ if and only if $\vec{Z}_{K-1}\in \mathcal{F}^{K-1}_{M_K(-Z_K)y}$. The directions\newline $\left(\Deriv{z}{K,ij}\right)_{1\leq i\leq j\leq n}$ are somehow transversal to the 'fibration' and we use to speak of \textit{the new directions}. The following statement is very useful for the induction step.

\begin{lemma}\label{proposition:spanningNewDirecitons}
For $K=3$, there are finitely many vector fields from $\mathcal{V}_3$ spanning the new directions in a generic fiber $\threefiber$ in points with
$$ Q_f^1\neq 0.$$

For $K\geq 4$, there are finitely many vector fields from $\mathcal{V}_K$ spanning the new directions in a generic fiber $\fiber$ in points with
$$ Q_f^{K-2}\neq 0\quad \text{and} \quad Q_s^{K-2}\neq 0.$$

\end{lemma}

Before we prove this lemma, let's apply it. We therefore define the sets
$$\mathcal{U}_3:=\{\vec{Z}_3\in \Wk{3}: Q_f^1(\vec{Z}_3)\neq 0\}$$
and for $K\geq 4$,
$$\mathcal{U}_K:= \{\vec{Z}_K\in \Wk{K}: Q_f^{K-2}(\vec{Z}_K)Q_s^{K-2}(\vec{Z}_K)\neq 0\}.$$
What now follows is one argument of the induction step. 
\begin{lemma}
Let $K\geq 3$ and suppose that the \spanning\ is true for $K-1$. Then there is a finite collection $A\subset \Qk{K}$ spanning the tangent bundle $T(\fiber\cap \mathcal{U}_K)$ for every generic fiber $\fiber$. 
\end{lemma}

\begin{proof}
In a first step, assume $K=3$ and consider a point $\vec{Z}_3\in \mathcal{U}_3$ such that $y:=\Phi_3(\vec{Z}_3)\in Y^3_g$ is in the generic stratum, i.e. $b=P_s^3(\vec{Z}_3)\neq 0$. By definition of the set $\mathcal{U}_3$, there holds $Q_f^1(\vec{Z}_3)\neq 0$. According to Lemma \ref{proposition:spanningNewDirecitons}, there is a finite collection $A\subset \Qk{3}$ which spans the new directions in points with $Q_f^1\neq 0$. Hence $\vec{Z}_3$ is a point, where the new directions are spanned.

By the recursive formula (\ref{RecursiveFormula}), there holds $\Phi_2(\vec{Z}_2)=M_3(-Z_3)y$. Moreover, 
$$Z_1e_n = \pi_2\circ \Phi_2(\vec{Z}_2) = \pi_2\circ M_3(-Z_3)y = a - Z_3b$$
implies that $\vec{Z}_2 \in \mathcal{F}^2_{M_3(-Z_3)y}$ is contained a generic $2$-fiber, since $\vec{Z}_3=(\vec{Z}_2,Z_3)\in \mathcal{U}_3$. 
Therefore the tangent space ${T_{\vec{Z}_3}\mathcal{F}^2_{M_3(-Z_3)y}\subset T_{\vec{Z}_3}\threefiber}$ is spanned by $A_2\subset \Qk{2}$ according to Lemma \ref{lemma:spanningTheoremForCaseTwo}. Since the new directions are complementary, the collection $A\cup A_2$ spans the tangent space $T_{\vec{Z}_3}\threefiber$. This holds for every generic fiber $\threefiber$, hence there is a finite collection in $\Qk{3}$ which spans the tangent bundle $T(\threefiber\cap \mathcal{U}_3)$ for every generic fiber.

In the case $K\geq 4$ we'll argue similarly. Let $\vec{Z}_K\in \mathcal{U}_K$ be a point such that $y:=\Phi_K(\vec{Z}_K)\in Y_g^K$ is in the generic stratum, i.e. $\pi_K(y)\neq 0$. Write $\vec{Z}_K=(\vec{Z}_{K-1},Z_K)$ and $\vec{Z}_{K-1}=(\vec{Z}_{K-2},Z_{K-1})$. Note that
$$\Phi_{K-1}(\vec{Z}_{K-1})=M_K(-Z_K)y \quad\text{ and } \quad\Phi_{K-1}(\vec{Z}_{K-1}) = M_{K-1}(Z_{K-1})\Phi_{K-2}(\vec{Z}_{K-2})$$
by the recursive formula (\ref{RecursiveFormula}). By defintion of $\pi_K$ and $M_K$, there holds $\pi_K\circ M_K=\pi_K$, hence $\pi_{K-1}\circ \Phi_{K-1} = \pi_{K-1}\circ \Phi_{K-2}$. Moreover, by definition of $\mathcal{U}_K$, we have $Q_f^{K-2}Q_s^{K-2}(\vec{Z}_K)\neq 0$, which implies $\pi_{K-1}\circ \Phi_{K-1}(\vec{Z}_{K-1})\neq 0$, that is, $\vec{Z}_{K-1}$ is contained in the generic fiber $\mathcal{F}^{K-1}_{M_K(-Z_K)y}$. In addition, we have $\vec{Z}_{K-1}\in \Wk{K-1}$, since otherwise this would contradict $Q_f^{K-2}Q_s^{K-2}\neq 0$ by Lemma \ref{lemma:propertyOne}. Therefore the tangent space $T_{\vec{Z}_{K-1}}\mathcal{F}^{K-1}_{M_K(-Z_K)y}$ is spanned by the finite collection in $\Qk{K-1}$ provided from the \spanning\ for $K-1$. Similarly as in the previous case, $\vec{Z}_K$ is a point where the new directions are spanned, according to Lemma \ref{proposition:spanningNewDirecitons}. Again, the new directions are complementary and we've shown that there is a finite collection in $\Qk{K}$ spanning the tangent space $T_{\vec{Z}_K}\fiber$ for each generic fiber. And this proves the lemma.
\end{proof}

\begin{Rem}\label{remark:strategySpanningTheorem}
An application of Corollary \ref{cor:shortCutSpanningTheorem} and the former lemma leads us to the following observation. 
To complete the induction step, it suffices to find a finite collection $A\subset \Qk{K}$ satisfying
\begin{align}\label{equation:inductionStep}
\fiber \cap \Wk{K}\subset C_A(\fiber \cap \mathcal{U}_K), \quad \text{for each generic fiber }\fiber.
\end{align}
The idea of the proof is to stratify $\mathcal{W}_K$ suitably, i.e. find a finite descending chain of subspaces
$$\mathcal{W}_K =: X_N\supset ... \supset X_0 = \emptyset,$$ where the spaces $X_0,...,X_{N-1}$ are closed. In a first step, we'll find a finite collection $A_N\subset \Qk{K}$ such that the stratum $S_N:=X_N\setminus X_{N-1}$ satisfies
$$ \fiber \cap S_N \subset C_{A_N}(\fiber \cap \mathcal{U}_K)$$
for each generic fiber $\fiber$. Then, we find a finite collection $A_k\subset \Qk{K}$, for each stratum $S_k:=X_k\setminus X_{k-1}$, $1\leq k\leq N-1$, such that
$$ \fiber \cap S_k\subset C_{A_k}(\fiber \cap S_{k+1})$$
for each generic fiber $\fiber$. Define the finite collection $ A:= \bigcup_{k=1}^N A_k \subset \Qk{K}.$ 
Then we get
$$ \fiber \cap S_1 \subset C_A(\fiber \cap S_2) \subset \cdots \subset C_A(\fiber \cap S_N) \subset C_A(\fiber \cap \mathcal{U}_K)$$
for each generic fiber $\fiber$.
As it turns out, the stratification of $\mathcal{W}_K$ in the cases $K=3$, $K=4$ and $K\geq 5$ differ in some elementary ways. 
\end{Rem}

\begin{proof}[Proof of Lemma \ref{proposition:spanningNewDirecitons}]
Let $\fiber$ be a generic fiber, that is, there holds $\tilde{P}^K_{j^*}\neq 0$ for some $1\leq j^*\leq n$. Recall that this fiber is biholomorphic to $\mathcal{G}_{\pi_K(y)}\times \C{\frac{n(n-1)}{2}}$ with the biholomorphism obtained by expressing the $n$ variables $z_{K,j^*1},...,z_{K,j^*n}$ in terms of the remaining variables $\left(z_{K,ij}\right)_{\substack{1\leq i \leq j \leq n; i,j\neq j^*}}$. Clearly, the fields $\left(\Deriv{z}{K,ij} \right)_{\substack{1\leq i \leq j \leq n; i,j\neq j^*}}$ are tangential to $\mathcal{G}_{\pi_K(y)}\times \C{\frac{n(n-1)}{2}}$ and moreover, they span $\{0\}\times \C{\frac{n(n-1)}{2}}$. Hence the corresponding lifts span the new directions $\left(\Deriv{z}{K,ij} \right)_{\substack{1\leq i \leq j \leq n; i,j\neq j^*}}$ in $\mathcal{F}^K_y$. It remains to find $n$ vector fields $\partial_{x_1}^{K-1},...,\partial_{x_n}^{K-1}$ tangent to $\mathcal{G}_{\pi_K(y)}\times \{0\}$ such that the projection of the corresponding lifts $\varphi_{x_1,j^*}^K,...,\varphi_{x_n,j^*}^K$ to the new directions $\left( \Deriv{z}{K,j^*j}\right)_{1\leq j\leq n}$ are linearly independent. Let $\alpha_{j^*,1},...,\alpha_{j^*,n}$ denote the component vectors of such projections in the frame $\left( \Deriv{z}{K,j^*j}\right)_{1\leq j\leq n}$. We need to show that the matrix $A:=(\alpha_{j^*,1}|\cdots | \alpha_{j^*,n})$ is regular. Set $u_m:=(\partial_{x_m}^{K-1}(\tilde{P}^{K-2}_1),...,\partial_{x_m}^{K-1}(\tilde{P}^{K-2}_n))^T$ and recall from the proof of Lemma \ref{lemma:liftingVectorFields}, that there is a regular matrix $B:=F_{j^*}(\tilde{P}^K)^{-1}$ with $\alpha_{j^*,m}=Bu_m$. Therefore $A$ is regular if and only if $U:=(u_1|\cdots | u_n)$ is regular.

In a first step, we assume $K\geq 4$ and consider the tupels of (\ref{typeOne}) $$x_m:=(z_{K-2,mm},z_{K-1,11},...,z_{K-1,nn}), 1\leq m\leq n.$$ Without loss of generality, let $K=2k+1$. The entries of $U$ are given by $u_{ij}:=\partial_{x_i}^{2k}(P_{j}^{2k})$. Let's compute the following derivatives
$$ \Deriv{z}{K-2,ii}P_s^{2k} = \begin{pmatrix}
Z_{K-1} & I_n
\end{pmatrix} \begin{pmatrix}
0 & \tilde{E}_{ii}\\ 0&0
\end{pmatrix}\begin{pmatrix}
P_f^{2k-2}\\ P_s^{2k-2}
\end{pmatrix}=P_{n+i}^{2k-2} Z_{K-1}e_i, $$
and
$$ \Deriv{z}{K-1,mm}P_s^{2k} = \begin{pmatrix}
\tilde{E}_{mm}&0 
\end{pmatrix} \begin{pmatrix}
P_f^{2k-1}\\ P_s^{2k-1}
\end{pmatrix} = P_m^{2k-1}e_m. $$
Furthermore, we have $\Deriv{z}{K-1,mm}P_f^{2k}\equiv 0$ and
$$ \Deriv{z}{K-2,ii} P_f^{2k} = \tilde{E}_{ii}P_s^{2k-1} = P_{n+i}^{2k-1}e_i.$$
We obtain
\begin{align*}
u_{ij}= \partial_{x_i}^{2k}(P_j^{2k}) &= \det \begin{pmatrix}
\delta_{ij}P_{n+i}^{2k-1} & 0 & \cdots & 0\\
P_{n+i}^{2k-1}Z_{K-1}e_i & P_1^{2k-1}e_1 & \cdots & P_n^{2k-1}e_n
\end{pmatrix} = \delta_{ij}P_{n+i}^{2k-1}Q_f^{2k-1},
\end{align*}
and hence the matrix 
$$ U = \begin{pmatrix}
P_{n+1}^{2k-1}Q_f^{2k-1} & & \\ & \ddots & \\
& & P_{2n}^{2k-1}Q_f^{2k-1}
\end{pmatrix}$$
is regular, provided $Q_f^{2k-1}Q_s^{2k-1}\neq 0$.

In a second step, let $K=3$ and consider the tupels of (\ref{typeFour})
$$ x_j:=(z_{1,nj},z_{2,11},...,z_{2,nn}), 1\leq j \leq n.$$
We obtain derivatives $\Deriv{z}{1,nj}P_s^2 = Z_2e_j, \Deriv{z}{2,ii}P_s^2=z_{1,ni}e_i$ and $\Deriv{z}{1,nj}P_f^2 = e_j, \Deriv{z}{2,ii}P_f^2\equiv 0$, for $1\leq i,j\leq n$. Hence we get
$$ u_{ij}=\det \begin{pmatrix}
\delta_{ij} & 0 & \cdots & 0\\ Z_2e_j & z_{1,n1}e_1 & \cdots & z_{1,nn}e_n
\end{pmatrix} = \delta_{ij}Q_f^1$$
and $U=Q_f^1I_n$. This completes the proof.
\end{proof}

\subsubsection{Spanning theorem for \texorpdfstring{$K=3$}{}}
We want to find a finite collection $A\subset \Qk{3}$ such that (\ref{equation:inductionStep}) is fulfilled for $K=3$. 

Let $1\leq i, j\leq n$, $i\neq j$, and consider the tupels of (\ref{typeEight})
$$ {x_{i,j}:= (z_{1,ni}, z_{2,j1},...,z_{2,jn}).}$$
The corresponding vector fields $\partial_{x_{i,j}}^2$ are of the form
$$ \partial^2_{x_{i,j}} = \det \begin{bmatrix} \partial / \partial z_{1,ni} & \partial / \partial z_{2,j1} & \cdots & \partial / \partial z_{2,jj} & \cdots & \partial/\partial z_{2,jn} \\
z_{2,1i} & z_{1,nj} & &0 & & 0\\
\vdots &  & \ddots & & \\
z_{2,ji} & z_{1,n1} & \cdots & z_{1,nj} & \cdots & z_{1,nn} \\
\vdots &  & & & \ddots & & \\
z_{2,ni} & 0 & & 0& & z_{1,nj}
\end{bmatrix}  $$
and they satisfy $$ \partial_{x_{i,j}}^2(z_{1,kn}) = \delta_{ik}(z_{1,nj})^n,$$for $1\leq i,j,k\leq n$. 

By definition, points in the set $\mathcal{N}:=\Wk{3}\setminus \mathcal{U}_3$ satisfy $Q_f^1=z_{1,1n}\cdots z_{1,nn} = 0$ and $P_f^1=Z_1e_n\neq 0$. Hence there are indices $ 1\leq i_1<...<i_k\leq n$ and $1\leq j_1<...<j_{n-k}\leq n$ satisfying
\begin{enumerate}[label=(\roman*)]
\item $\{1,...,n\}=\{i_1,...,i_k\}\dot{\cup}\{j_1,...,j_{n-k}\}$,
\item $z_{1,in}=0$, for all $i\in \{i_1,...,i_k\}$,
\item $z_{1,jn}\neq 0$, for all $j\in \{j_1,...,j_{n-k}\}$.
\end{enumerate}
Fix an index $j\in \{j_1,...,j_{n-k}\}$. Then observe
$$ \partial_{x_{i_1,j}}^2\circ \cdots \circ \partial_{x_{i_k,j}}^2(z_{1,i_1n}\cdots z_{1,i_kn}) = (z_{1,jn})^n \partial_{x_{i_1,j}}^2\circ \cdots \circ \partial_{x_{i_{k-1},j}}^2(z_{1,i_1n}\cdots z_{1,i_{k-1}n}),$$
which inductively implies
$$ \partial_{x_{i_1,j}}^2\circ \cdots \circ \partial_{x_{i_k,j}}^2(z_{1,i_1n}\cdots z_{1,i_kn}) = (z_{1,jn})^{kn} \neq 0.$$
Let $\vec{Z}_3\in \mathcal{N}$ be a point over the generic stratum, that is, $\Phi_3(\vec{Z}_3)\in Y_g^3$. According to Lemma \ref{lemma:leavingASubset}, we find finitely many of the above fields $\partial_{x_{i,j}}^2$ such that a suitable finite composition of the respective flows moves $\vec{Z}_3$ away from $\mathcal{N}$. More precisely, we have a finite collection $A\subset \Qk{3}$ and an automorphism $\alpha\in G_A$ such that $\alpha(\vec{Z}_3)\in C_A(\mathcal{U}_3)$. That's exactly what we need to get (\ref{equation:inductionStep}). 

\subsubsection{The Spanning theorem for \texorpdfstring{$K=4$}{}}
As in the previous step, we want to find a suitable finite collection $A$ such that $(\ref{equation:inductionStep})$ is satisfied. In order to do this, we'll do a \textit{divide and conquer} with the set $ \mathcal{N}_4:=\Wk{4}\setminus \mathcal{U}_4.$
\begin{lemma}\label{lemma:case4firstStep}
Define the set
$$ \mathcal{N}_1:=\{\vec{Z}_4\in \Wk{4}: Q_f^2(\vec{Z}_4)Q_s^2(\vec{Z}_4)=0, P_s^2(\vec{Z}_4)\neq 0\}.$$
Then there is a finite collection $A\subset \Qk{4}$ such that
$$ \mathcal{F}^4_y\cap \mathcal{N}_1\subset C_A(\mathcal{F}^4_y\cap \mathcal{U}_4)$$
for each generic fiber $\mathcal{F}^4_y$.
\end{lemma}
\begin{proof}
In a first step, recall that $P_f^2=P_f^1$ by the recursive formula (\ref{RecursiveFormula}). Consider a point $\vec{Z}_4$ in the set
$$ X_1 = \{\vec{Z}_4\in \mathcal{N}_1: Q_f^1\neq 0\}.$$ By definition of this set, we find indices $1\leq i^*,i\leq n$ with $P_{n+i^*}^2(\vec{Z}_4)\neq 0$ and $P_{n+i}^2(\vec{Z}_4)=0$. Pick the vector field $\partial_{x_i}^3$ corresponding to the tupel $x_{i}=(z_{2,ii},z_{3,i^*1},...,z_{3,i^*n})$ of (\ref{typeTwo}). It is of the form
$$ \pm (P_{n+i^*}^2)^n\Deriv{z}{2,ii} + \sum_{l=1}^n \alpha_l \Deriv{z}{3,i^*l}$$
for some suitable holomorphic functions $\alpha_1,...,\alpha_n$. Then
$$ \partial_{x_i}^3(P_{n+i}^2) = \partial_{x_i}^3(\delta_{in} + e_i^TZ_2Z_1e_n) = z_{1,in}(P_{n+i^*}^2)^n \neq 0,$$
and this applies for every $1\leq i\leq n$ with $P_{n+i}^2(\vec{Z}_4)=0$. Hence there is a finite collection $A\subset \Qk{4}$ such that
$$ \mathcal{F}^4_y \cap X_1 \subset C_A(\mathcal{F}^4_y\cap \mathcal{U}_4)$$
for each generic fiber $\mathcal{F}^4_y$. 

In a second step, let's consider a point $\vec{Z}_4=(\vec{Z}_3, Z_4)$ in the set
$$ X_2 = \{\vec{Z}_4: Q_f^1=0, P_f^1\neq 0, P_s^2\neq 0\} \subset \mathcal{N}_1.$$
Observe that the projection of $\vec{Z}_4$ to the first component $\vec{Z}_3$ is contained in a generic $3$-fiber $\mathcal{F}^3_{\tilde{y}}$ and in $\Wk{3}$. From the case $K=3$, we know the existence of a finite collection $B\subset \Qk{3}$ such that $\vec{Z}_3\in C_B(\mathcal{F}^3_{\tilde{y}}\cap \mathcal{U}_3)$. Put the corresponding pullbacks into the collection $A\subset \Qk{4}$. Then we've found a finite collection $A$ such that
$$\mathcal{F}^4_y\cap X_2\subset C_A(\mathcal{F}_y^4\cap X_1) \subset C_A(\mathcal{F}^4_y \cap  \mathcal{U}_4)$$
for each generic fiber $\mathcal{F}^4_y$. 

In a third step, let's consider a point $\vec{Z}_4$ in the set
$$ X_3 = \{\vec{Z}_4: P_f^1=0\} \subset \mathcal{N}_1,$$
that is, $Z_1e_n=0$. The vector field $\partial_x^3$ corresponding to the tupel $x=(z_{1,1n},...,z_{1,nn},z_{3,nn})$ of (\ref{typeSeven}), satisfies
$$ \partial_x^3(z_{1,nn}) = \det \left( e_1+ Z_3Z_2e_1,...,e_{n-1}+ Z_3Z_2e_{n-1},e_n\right)$$
in $\vec{Z}_4$. Hence there is a finite collection $A\subset \Qk{4}$ such that
$$\mathcal{F}^4_y\cap \underbrace{\{\vec{Z}_4: P_f^1(\vec{Z}_4)=0,  \partial_x^3(z_{1,nn})\neq 0\}}_{=:\tilde{N}}\subset C_A(\mathcal{F}^4_y\cap \{\vec{Z}_4: P_f^1\neq 0, P_s^2\neq 0\})$$
for each generic fiber $\mathcal{F}^4_y$.

Next, we consider the vector fields $\partial_{x_i}^3$ and $\partial_{y_i}^3$, $i=1,...,n-1$, corresponding to the tupels ${x_i=(z_{2,ii},z_{3,1n},...,z_{3,nn})}$ and ${y_i=(z_{3,ii},z_{3,1n},...,z_{nn})}$ of  (\ref{typeTwo}) and (\ref{typeThree}), respectively. They are of the form $\partial_{x_i}^3= \Deriv{z}{2,ii}$ and $\partial_{y_i}^3=\Deriv{z}{3,ii}$ on $\Oset{4}{P_f^2}$. The formula
\begin{align}\label{equation:productFormula}
\frac{d}{dt} \det (A_1(t),...,A_n(t)) = \sum_{l=1}^n \det ( A_1(t),...,A_l'(t),...,A_n(t))
\end{align}
is sort of a product rule, where $A_1(t),...,A_n(t)$ denote columns of a $n\times n$-matrix $A(t)$ depending on $t$. In our case, we consider a matrix $A$, where the first $(n-1)$ columns are given by 
$$ A_i= e_i + \sum_{l=1}^n z_{2,li}Z_3e_l, \quad 1\leq i\leq n-1,$$
and the $n$-th column $A_n=e_n$. Observe that
$$\Deriv{z}{2,ii}A_j=\Deriv{z}{2,ii}(e_j+ \sum_{l=1}^n z_{2,lj}Z_3e_l) = \delta_{ij}Z_3e_i.$$
Hence
$$ \Deriv{z}{2,11}\cdots \Deriv{z}{2,(n-1)(n-1)}\det(A) = \det(Z_3e_1,...,Z_3e_{n-1},e_n).$$
Furthermore, we obtain
$$ \Deriv{z}{3,11}\cdots \Deriv{z}{3,(n-1)(n-1)}\det(Z_3e_1,...,Z_3e_{n-1},e_n) = \det(e_1,...,e_n)=1.$$
By Lemma \ref{lemma:leavingASubset} and the previous step, we conclude the existence of a finite collection $A\subset \Qk{4}$ such that
$$\mathcal{F}^4_y\cap \{\vec{Z}_4\in \mathcal{N}_1: P_f^1(\vec{Z}_4)=0, \partial_x^3(z_{1,nn})= 0\}\subset C_A(\mathcal{F}^4_y\cap \tilde{N})$$
and
$$ \mathcal{F}^4_y\cap \tilde{N}\subset C_A(\mathcal{F}^4_y\cap \{\vec{Z}_4: Q_f^1\neq 0, P_s^2\neq 0\})$$
for each generic fiber $\mathcal{F}^4_y$. In particular, this implies
$$\mathcal{F}^4_y\cap X_3\subset C_A(\mathcal{F}^4_y\cap X_2) \subset C_A(\mathcal{F}^4_y \cap  \mathcal{U}_4)$$
for each generic fiber $\mathcal{F}^4_y$.

For the final step, observe that $\mathcal{N}_1 = X_1\cup X_2 \cup X_3$. We put all of the involved vector fields above together and obtain a finite collection $A\subset \Qk{4}$ satisfying
$$ \mathcal{F}^4_y \cap\mathcal{N}_1 \subset C_A(\mathcal{F}^4_y \cap \mathcal{U}_4)$$
for each generic fiber $\mathcal{F}^4_y$. This finishes the proof. 
\end{proof}

\begin{lemma}\label{lemma:nonGeneric3fiberInClosure}
Define the set
$$ \mathcal{N}_2:=\{\vec{Z}_4: P_s^2(\vec{Z}_4)=0\}.$$
Then there is a finite collection $A\subset \Qk{4}$ such that
$$\mathcal{F}^4_y\cap \mathcal{N}_2\subset C_A(\mathcal{F}^4_y\cap \mathcal{U}_4)$$
for each generic fiber $\mathcal{F}^4_y$.
\end{lemma}

In order to prove this lemma, the symplectic nature of the elementary matrices $M_i(Z)$ comes into play. The following result, also called \textit{complementary bases theorem}, is proved by Dopico and Johnson \cite{Dopico:2006vy}.
\begin{theorem}[Complementary bases theorem]
Let $$ M=\begin{pmatrix}
A&B\\ C&D
\end{pmatrix}\in \SpC$$ be a symplectic $2n\times 2n$-matrix and let $k:=\mathrm{rank}(B)$, i.e. there are $k$ indices $j_1,...,j_k$ such that the vectors $Be_{j_1},...,Be_{j_k}$ form a basis of the image $\mathrm{Im}(B)$. Let $ i_1,...,i_{n-k}$ denote the complementary indices in $\{1,...,n\}$, that is, $ \{1,...,n\} = \{i_1,...,i_{n-k}\}\dot{\cup}\{j_1,...,j_{k}\}.$

Then the $n\times n$ matrix
$$ X=\begin{pmatrix}
Ae_{i_1}&\cdots & Ae_{i_{n-k}} & Be_{j_1} & \cdots & Be_{j_{k}}
\end{pmatrix}$$
is regular.
\end{theorem}

\begin{Exa}\label{lemma:principalMinor}
Consider the symplectic elementary matrix $$ \Eu{Z},$$ and assume the symmetric matrix $Z\in \Cm$ to be of rank $k$. We find complementary indices as in the theorem and obtain that the matrix
$$ X = \begin{pmatrix}
e_{i_1}&\cdots & e_{i_{n-k}} & Ze_{j_1} & \cdots & Ze_{j_k}
\end{pmatrix}$$
is regular. By applying some Gauss-elimination to $X$, we conclude that $Z$ has a non-vanishing principal minor of order $k$. More precisely, let $(Z)_{i_1,...,i_{n-k}}$ be the symmetric $k\times k$-matrix, obtained by removing columns and rows $i_1,...,i_{n-k}$ from $Z$. Then $\det((Z)_{i_1,...,i_{n-k}})=\pm \det(X)\neq 0$. 
\end{Exa}

\begin{proof}[Proof of Lemma \ref{lemma:nonGeneric3fiberInClosure}]
Let's stratify $\mathcal{N}_2$ in the following way. For $0\leq k \leq n$, define 
$$ \mathcal{A}_k:=\{\vec{Z}_4\in \mathcal{N}_2: \mathrm{rank}(Z_3)\leq k\}$$
and as a convention let $\mathcal{A}_{-1}:=\emptyset$. Then each stratum
$$\mathcal{B}_k := \mathcal{A}_k\setminus \mathcal{A}_{k-1}, \quad 0\leq k\leq n,$$
consists of those points $\vec{Z}_4\in \mathcal{N}_2$ with $\mathrm{rank}(Z_3)=k$. 

Now we want to proceed as in Remark \ref{remark:strategySpanningTheorem}. We start by assuming $A$ to be the finite collection of Lemma \ref{lemma:case4firstStep}, and then successively add matching fields to $A$.

For each point in $\mathcal{N}_2$, there is an index $1\leq i*\leq n$ with $z_{1,i^*n}\neq 0$, since $P_f^1=Z_1e_n$ and $P_s^2=0$ don't vanish simultaneously by the recursive formula (\ref{RecursiveFormula}). So this is especially true for points in the stratum $\mathcal{B}_n$. The tupel ${x=(z_{1,nj},z_{2,i^*1},...,z_{2,i^*n})}$, $j\neq i^*$, is of (\ref{typeEight}) and we claim that the corresponding vector field $\partial_x^3$ satisfies $ \partial_x^3(z_{1,nj}) = (z_{1,ni^*})^n \det(Z_3) \neq 0$ and $\partial_x^3(z_{3,ij})=0$ for all $1\leq i\leq j\leq n$. By construction, there holds $\partial_x^3(P_i^3)=0$ for $1\leq i\leq n$. The recursive formula (\ref{RecursiveFormula}) implies
$$ P_i^3 = z_{1,ni} + \sum_{l=1}^n z_{3,il}P_{n+l}^2.$$
Hence
\begin{align*}
0 = \partial_x^3(P_i^3) = \partial_x^3(z_{1,ni})+ \sum_{l=1}^n z_{3,il}\partial_x^3(P_{n+l}^2), \quad 1\leq i\leq n.
\end{align*}
Suppose that $\partial_x^3(P_{n+l}^2)=0$, for all $1\leq l\leq n$. Then we obtain a contradiction
$$0=\partial_x^3(P_j^3) = \partial_x^3(z_{1,nj}) \neq 0.$$
Therefore, the flow of $\partial_x^3$ through points of $\mathcal{B}_n$ leaves $\mathcal{N}_2$. We add $\partial_x^3$ (actually its $n$ push-forwards with respect to the biholomorphisms from Lemma \ref{lemma:genericFibers}) to the collection $A$ and by Lemma \ref{lemma:case4firstStep}, we get
$$ \mathcal{F}^4_y\cap \mathcal{B}_n\subset C_A(\mathcal{F}^4_y\cap \mathcal{N}_1) \subset C_A(\mathcal{F}^4_y\cap \mathcal{U}_4).$$

Now, let's prove the claim. The vector field $\partial_x^3$ is constant in the directions $\tfrac{\partial}{\partial z_{3,ij}}$, by definition. So we only need to check $\partial_x^3(z_{1,jn})\neq 0$. By definition, there holds $\partial_x^3(z_{1,jn})=\det(A)$ for $ A:=\begin{pmatrix}
\Deriv{z}{2,i^*1}P_f^3 & \cdots & \Deriv{z}{2,i^*n}P_f^3 
\end{pmatrix}$. Note that
$$ \Deriv{z}{2,ij} P_f^3 = \begin{pmatrix}
I_n & Z_3
\end{pmatrix} \begin{pmatrix}
0 & 0 \\ \tilde{E}_{ij} & 0
\end{pmatrix} \begin{pmatrix}
Z_1e_n\\ e_n
\end{pmatrix} = Z_3\tilde{E}_{ij}Z_1e_n.$$
Therefore, by definition of $F_{i^*}$ (see paragraph before Corollary \ref{cor:JacobianElementaryMatrix}),
$$ A=\begin{pmatrix}
Z_3\tilde{E}_{i^*1}Z_1e_n&\cdots & Z_3\tilde{E}_{i^*n}Z_1e_n
\end{pmatrix} = Z_3 F_{i^*}(Z_1e_n) = Z_3 \begin{pmatrix}
z_{1,i^*n} & && & \\ & \ddots & && \\z_{1,1n} & \cdots & z_{1,i^*n} & \cdots & z_{1,nn}\\
&&& \ddots & \\ &&&& z_{1,i^*n}
\end{pmatrix}$$
and hence
$$ \partial_x^3(z_{1,jn}) = \det(A) = (z_{1,i^*n})^n\det(Z_3)\neq 0.$$

In a next step, we assume $1\leq k\leq n-1$ and we divide the stratum $\mathcal{B}_k$ into two more strata. Define
$$ \mathcal{C}_k := \{\vec{Z}_4\in \mathcal{B}_k: Z_3e_i\neq 0\Rightarrow z_{1,ni}=0, \forall 1\leq i\leq n\}$$
and consider a point $\vec{Z}_4\in \mathcal{B}_k\setminus \mathcal{C}_k$. By definition, there is an index $1\leq j^*\leq n$ such that $Z_3e_{j^*}\neq 0$ and $z_{1,j^*n}\neq 0$. Furthermore, again by definition, we assume that the matrix $Z_3$ has rank $k$. By the complementary bases theorem, we can choose complementary indices $\{i_1,...,i_{n-k}\}\dot{\cup}\{j_1,...,j_k\} = \{1,...,n\}$ such that $j^*\in \{j_1,...,j_k\}$ and such that the vectors $Z_3e_{j_1},...,Z_3e_{j_k}$ form a basis of the image $\mathrm{Im}(Z_3)$. Choose $i\in \{i_1,...,i_{n-k}\}$, then the tupel $x=(z_{1,i_1n},...,z_{1,i_{n-k}n},z_{2,j^*j_1},...,z_{2,j^*j_k},z_{3,ii})$ is of {(\ref{typeSeven})}. Since $\Deriv{z}{3,ii}P_f^3 = \tilde{E}_{ii}P_s^2=0$ on $\mathcal{N}_2$, the vector field $\partial_x^3$ is given by $\pm\det(B)\Deriv{z}{3,ii}$ on $\mathcal{N}_2$, where
\begin{align*}
B &= \begin{pmatrix}
\Deriv{z}{1,i_1n}P_f^3 & \cdots & \Deriv{z}{1,i_{n-k}n}P_f^3 & \Deriv{z}{2,j^*j_1}P_f^3 & \cdots & \Deriv{z}{2,j^*j_k}P_f^3
\end{pmatrix}\\ &= \begin{pmatrix}
(I_n+Z_3Z_2)e_{i_1} & \cdots & (I_n+Z_3Z_2)e_{i_{n-k}} & Z_3\tilde{E}_{j^*j_1}Z_1e_n & \cdots & Z_3\tilde{E}_{j^*j_k}Z_1e_n
\end{pmatrix}.
\end{align*}
We're going to show that $B$ is regular. At first, observe that by some basic linear algebra arguments, we get
$$ \mathrm{span}\{\tilde{E}_{j^*j_1}Z_1e_n,...,\tilde{E}_{j^*j_k}Z_1e_n\} = \mathrm{span}\{z_{1,j^*n}e_{j_1},...,z_{1,j^*n}e_{j_k}\},$$
if $j^*\in \{j_1,...,j_k\}$ and $z_{1,j^*n}\neq 0$. By multi-linearity of the determinant operator, this implies
$$ \det (B) = \pm (z_{1,j^*n})^k \det \begin{pmatrix}
(I_n+Z_3Z_2)e_{i_1}& \cdots & (I_n+Z_3Z_2)e_{i_{n-k}} & Z_3e_{j_1} & \cdots & Z_3e_{j_k}
\end{pmatrix}$$
and by the complementary bases theorem, applied to the matrix
$$ \Eu{Z_3}\El{Z_2} = \begin{pmatrix}
I_n+Z_3Z_2 & Z_3\\ Z_2 & I_n
\end{pmatrix},$$
this is non-vanishing, i.e. $\det(B)\neq 0$ in the given point $\vec{Z}_4\in \mathcal{B}_k\setminus \mathcal{C}_k$. Next consider the $(k+1)\times (k+1)$ submatrix of $Z_3$ 
$$ Z:=\begin{pmatrix}
z_{3,ii} & z_{3,ij_1} & \cdots & z_{3,ij_k}\\ z_{3,j_1i} & z_{3,j_1j_1} & \cdots & z_{3,j_1j_k}\\ \vdots & \vdots & \ddots & \vdots \\ z_{3,j_ki} & z_{3,j_kj_1} & \cdots & z_{3,j_kj_k}
\end{pmatrix},$$
for $i\in \{i_1,...,i_{n-k}\}$. Its determinant $\det(Z)$, written as a function in $z_{3,ii}$, is given by
$$ \det(Z) = \det((Z_3)_{i_1,...,i_{n-k}}) z_{3,ii} + \alpha,$$
where $\alpha\in \mathbb{C}$ is constant in $z_{3,ii}$. There holds $\det(Z)=0$ on $\mathcal{B}_k$, since the rank of $Z_3$ is $k$ on this stratum. Apply the vector field $\partial_x^3$ to the equation $\det(Z)=0$. Using the same notation as in Example \ref{lemma:principalMinor}, this gives us
$$ \partial_x^3(\det(Z)) = \det(B) \Deriv{z}{3,ii} \det(Z) = \det(B)\det((Z_3)_{i_1,...,i_{n-k}}) \neq 0.$$
Hence the flow of $\partial_x^3$ through the given point $\vec{Z}_4\in\mathcal{B}_k\setminus \mathcal{C}_k$ leaves the set $\mathcal{A}_k$ and intersects the stratum $\mathcal{B}_{k+1}$.
Hence there exists a finite collection $A\subset \Qk{4}$ of complete fiber-preserving holomorphic vector fields such that
$$ \mathcal{F}^4_y\cap \mathcal{B}_k\setminus \mathcal{C}_k\subset C_A(\mathcal{F}^4_y\cap \mathcal{B}_{k+1})$$
for each generic fiber $\mathcal{F}^4_y$.

Now, consider a point $\vec{Z}_4\in \mathcal{C}_k$. Choose again a complementary set of indices \newline $\{i_1,...,i_{n-k}\}\dot{\cup}\{j_1,...,j_k\}=\{1,...,n\}$ such that the vectors $Z_3e_{j_1},...,Z_3e_{j_k}$ form a basis of the image $\mathrm{Im}(Z_3)$. By definition of $\mathcal{C}_k$, there holds $z_{1,j_1n} = ... = z_{1,j_kn} = 0$. Recall that by definition of $\mathcal{N}_2$, there is an index $1\leq i^*\leq n$ with $z_{1,i^*n}\neq 0$. In particular, there holds $Z_3e_{i^*}=0$, again by definition of $\mathcal{C}_k$. The vector fields $\gamma^2_{jj,i^*}$, $j\in \{j_1,...,j_k\}$, from Lemma \ref{lemma:fiberPreservingVectorFields} are given by
$$ \gamma^2_{jj,i^*} = (z_{1,i^*n})^2\Deriv{z}{2,jj}, \quad j\in \{j_1,...,j_k\},$$
in the given point $\vec{Z}_4\in \mathcal{C}_k$. Consider the tupel ${x=(z_{1,n1},...,z_{1,nn},z_{3,i^*i^*})}$ of (\ref{typeFive}). Its corresponding vector field $\partial_x^3$ is of the form $\pm\det ( I_n+Z_3Z_2)\Deriv{z}{3,i^*i^*}$ in $\vec{Z}_4$. Apply a suitable composition to the equation $z_{3,i^*i^*}=0$, namely
\begin{align*}
\gamma^2_{j_kj_k,i^*}\circ \cdots \circ \gamma^2_{j_1j_1,i^*} \circ \partial_x^3(z_{3,i^*i^*}) &=\pm \gamma^2_{j_kj_k,i^*}\circ \cdots \circ \gamma^2_{j_1j_1,i^*}(\det(I_n+Z_3Z_2)) \\ &=\pm (z_{1,i^*n})^{2k} \Deriv{z}{2,j_kj_k}\circ\cdots \circ \Deriv{z}{2,j_1j_1}(\det(I_n+Z_3Z_2)).
\end{align*}
Note that $\Deriv{z}{2,jj}(I_n+Z_3Z_2) = Z_3\tilde{E}_{jj}$ which means that the $j$-th column $(I_n+Z_3Z_2)e_j$ is the only column in $I_n+Z_3Z_2$ depending on the variable $z_{2,jj}$. In particular, the product formula for determinants (\ref{equation:productFormula}) and a suitable rearrangement of the columns yields
$$ \Deriv{z}{2,j_kj_k}\circ \cdots \circ \Deriv{z}{2,j_1j_1}(\det(I_n+Z_3Z_2)) = \pm \det((I_n+Z_3Z_2)e_{i_1},...,(I_n+Z_3Z_2)e_{i_{n-k}},Z_3e_{j_1},...,Z_3e_{j_k}).$$
Another application of the complementary bases theorem implies
$$ \gamma^2_{j_kj_k,i^*}\circ \cdots \circ \gamma^2_{j_1j_1,i^*} \circ \partial_x^3(z_{3,i^*i^*})\neq 0.$$
By Lemma \ref{lemma:leavingASubset}, we have found a suitable finite collection $A\subset \Qk{4}$ such that
$$ \mathcal{F}^4_y\cap \mathcal{C}_k \subset C_A(\mathcal{F}^4_y\cap \mathcal{B}_k\setminus \mathcal{C}_k)$$
for all generic fibers $\mathcal{F}^4_y$.

In summary, we have found a finite collection $A\subset \Qk{4}$ satisfying
$$ \mathcal{F}^4_y\cap \mathcal{B}_k \subset C_A(\mathcal{F}^4_y\cap \mathcal{B}_{k+1}),$$
for each $1\leq k\leq n-1$ and each generic fiber $\mathcal{F}^4_y$.

In a last step, consider a point $\vec{Z}_4\in \mathcal{B}_0$. Then the field $\partial_x^3$, corresponding to the tupel $x=(z_{1,1n},...,z_{1,nn},z_{3,nn})$ of (\ref{typeSeven}), is given by
$$ \partial_x^3 = \pm \det(I_n+Z_3Z_2)\Deriv{z}{3,nn} = \pm \Deriv{z}{3,nn}.$$
Therefore, there is a finite $A\subset \Qk{4}$ with $\mathcal{F}^4_y\cap \mathcal{B}_0\subset C_A(\mathcal{F}^4_y\cap \mathcal{B}_1)$ for each generic fiber.

It remains to put everything together and apply the properties of the closure operator $C_A$. More precisely, we have found a finite $A\subset \Qk{4}$ such that
$$ C_A(\mathcal{F}^4_y\cap \mathcal{B}_k)\subset C_A(\mathcal{F}^4_y\cap \mathcal{B}_{k+1})$$
for each $0\leq k\leq n-1$ and each generic fiber. And since we already know
$$ \mathcal{F}^4_y\cap \mathcal{B}_n\subset C_A(\mathcal{F}^4_y\cap \mathcal{U}_4)$$
for each generic fiber, the proof is complete.
\end{proof}

\subsubsection{The Spanning theorem for \texorpdfstring{$K\geq 5$}{}}
The strategy for $K\geq 5$ is very similar to that for $K=4$. Let's write the function $Q_f^{K-2}Q_s^{K-2}$ in terms of the notation $\tilde{P}^K$ (see Lemma \ref{lemma:genericFibers}). Recall
$$ \tilde{P}^K = \begin{cases}
P_f^K &\text{if } K=2k+1\\ P_s^K & \text{if } K=2k
\end{cases}$$
and the recursive formula (\ref{RecursiveFormulaGeneralized})
$$ \PK = \PKMM + Z_K\PKM.$$
Then we get,
$$ Q_f^{K-2}Q_s^{K-2} = \PKMM_1\cdots \PKMM_n \PKMMM_1\cdots \PKMMM_n$$
for all $K\geq 5$. 

\begin{lemma}
Let $K\geq 5$ and define
$$ \mathcal{N}_1:=\{\vec{Z}_K\in \Wk{K}: \PKMM_1\cdots \PKMM_n \PKMMM_1\cdots \PKMMM_n=0,\PKMM\neq 0\}.$$
Then there is a finite collection $A\subset \Qk{K}$ such that
$$ \fiber \cap \mathcal{N}_1\subset C_A(\fiber \cap \mathcal{U}_K)$$
for each generic fibers $\fiber$. 
\end{lemma}
\begin{proof}
In a first step, consider a point $\vec{Z}_K$ in the set
$$ X_1 = \{\vec{Z}_K\in \mathcal{N}_1: \PKMMM_1\cdots \PKMMM_n \neq 0\}.$$
By definition of this set, we find indices $1\leq i,i^*\leq n$ with $\PKMM_i(\vec{Z}_K)=0$ and $\PKMM_{i^*}(\vec{Z}_K)\neq 0$. Pick the vector field $\partial_{x_i}^{K-1}$ corresponding to the tupel
$x_i=(z_{K-2,ii},z_{K-1,i^*1},...,z_{K-1,i^*n})$ of (\ref{typeTwo}). It is of the form
$$ \pm (\PKMM_{i^*})^n \Deriv{z}{K-2,ii}+ \sum_{l=1}^n \alpha_l \Deriv{z}{K-1,i^*l}$$
for some suitable holomorphic functions $\alpha_1,...,\alpha_n$. By the recursive formula (\ref{RecursiveFormulaGeneralized}), we get
$$ \PKMM_i = \tilde{P}^{K-4}_i+e_i^TZ_{K-2}\PKMMM$$
and therefore
$$\partial_{x_i}^{K-1}(\PKMM_i) = \pm (\PKMM_{i^*})^n \PKMMM_i\neq 0.$$
This applies for every $1\leq i\leq n$ with $\PKMM_i=0$. Hence there is a finite collection $A\subset \Qk{K}$ such that
$$ \fiber \cap X_1 \subset C_A(\fiber \cap \mathcal{U}_K)$$
for each generic fiber $\fiber$.

In a second step, consider a point $\vec{Z}_K=(\vec{Z}_{K-1},Z_K)$ in the set
$$ X_2 = \{\vec{Z}_K\in \mathcal{N}_1: \PKMMM_1\cdots \PKMMM_n=0, \PKMMM\neq 0\}.$$
Observe that the projection of $\vec{Z}_K$ to the first component $\vec{Z}_{K-1}$ is contained in a generic $(K-1)$-fiber $\mathcal{F}^{K-1}_{\tilde{y}}$ and in $\Wk{K-1}$. By the induction hypothesis, there is a finite collection $B\subset \Qk{K-1}$ such that
$$ \vec{Z}_{K-1}\in C_B(\mathcal{F}^{K-1}_{\tilde{y}}\cap \{\vec{Z}_{K-1}: \PKMMM_1\cdots \PKMMM_n\neq 0\}).$$
Put the corresponding pullbacks into the collection $A$. This gives us a finite collection $A\subset \Qk{K}$ such that
$$ \fiber \cap X_2 \subset C_A(\fiber \cap X_1) \subset C_A(\fiber \cap \mathcal{U}_K)$$
for each generic fiber $\fiber$. 

In a third step, consider a point $\vec{Z}_K=(\vec{Z}_{K-1},Z_K)$ in the set
$$ X_3 = \{\vec{Z}_K\in \mathcal{N}_1: \PKMMM=0\}.$$
Define the subset
$$\tilde{X}_3 = \{(\vec{Z}_{K-1},Z_K)\in X_3: \vec{Z}_{K-1}\in \Wk{K-1}\}.$$
By Lemma \ref{lemma:nonGeneric3fiberInClosure}, we can apply the induction hypothesis, to obtain a finite collection $A\subset \Qk{K}$ such that
$$\fiber\cap \tilde{X}_3\subset C_A(\fiber \cap X_2)$$
for each generic fiber $\fiber$. Observe that $\Wk{2k+1}=\Wk{2k}\times \Cm$, hence the two sets $X_3$ and $\tilde{X}_3$ coincide for $K=2k+1$. Now, let us assume $K=2k$ and consider a point $\vec{Z}_K\in X_3\setminus \tilde{X}_3$. This implies $Z_1e_n=Z_3e_n=...=Z_{2k-3}e_n=0$. Pick the vector field $\partial_x^{K-1}$ corresponding to the tupel $x=(z_{K-3,1n},...,z_{K-3,nn},z_{K-1,nn})$ of (\ref{typeFive}). It satisfies
$$ \partial_x^{K-1}(z_{K-3,nn}) = \det(e_1+Z_{K-1}Z_{K-2}e_1,...,e_{n-1}+Z_{K-1}Z_{K-2}e_{n-1}, e_n)$$ in $\vec{Z}_K$. Also consider the fields $\partial_{x_i}^{K-1}$ and $\partial_{y_i}^{K-1}$, $i=1,...,n-1$, corresponding to the tupels $x_i=(z_{K-2,ii},z_{K-1,1n},...,z_{K-1,nn})$ and $y_i=(z_{K-1,ii},z_{K-1,1n},...,z_{K-1,nn})$ of (\ref{typeTwo}) and {(\ref{typeThree})}, respectively. They are of the form $\partial_{x_i}^{K-1}=\Deriv{z}{K-2,ii}$ and $\partial_{y_i}^{K-1}=\Deriv{z}{K-1,ii}$ on $X_3\setminus \tilde{X}_3$. With the very same reasoning as in the third step of Lemma \ref{lemma:case4firstStep}, we obtain a finite collection $A\subset \Qk{K}$ such that
$$\fiber \cap X_3\setminus \tilde{X}_3 \subset C_A(\fiber \cap X_2)$$
for each generic fiber $\fiber$.
In summary, we obtain a finite collection $A\subset \Qk{K}$ such that
$$\fiber \cap X_3 \subset C_A(\fiber \cap X_2 ) \subset C_A(\fiber\cap \mathcal{U}_K)$$
for each generic fiber $\fiber$ and for each $K\geq 5$.

In a last step, observe that $\mathcal{N}_1=X_1\cup X_2 \cup X_3$. We put all the involved vector fields above together and obtain a finite collection $A\subset \Qk{K}$ satisfying
$$ \fiber \cap \mathcal{N}_1\subset C_A(\fiber\cap \mathcal{U}_K)$$
for each generic fiber $\fiber$. This finishes the proof.
\end{proof}

\begin{lemma}
Let $K\geq 5$ and define
$$\mathcal{N}_2:=\{\vec{Z}_K\in \Wk{K}: \PKMM=0\}.$$
Then there is a finite collection $A\subset \Qk{K}$ such that
$$ \fiber \cap \mathcal{N}_2\subset C_A(\fiber \cap \mathcal{U}_K)$$
for each generic fibers $\fiber$. 
\end{lemma}

\begin{proof}
The proof of this lemma follows more or less the strategy of Lemma \ref{lemma:nonGeneric3fiberInClosure}. We stratify $\mathcal{N}_2$ in the following way. For $0\leq k\leq n$, define
$$ \mathcal{A}_k = \{\vec{Z}_K\in \mathcal{N}_2: \mathrm{rank}(Z_{K-1})\leq k\}$$
and as a convention let $\mathcal{A}_{-1}:=\emptyset$. Then each stratum
$$ \mathcal{B}_k:=\mathcal{A}_k\setminus \mathcal{A}_{k-1},\quad 0\leq k\leq n,$$
consists of those points $\vec{Z}_K\in \mathcal{N}_2$ with $\mathrm{rank}(Z_{K-1})=k$.

In a first step, we prove the following claim: for $\vec{Z}_K=(\vec{Z}_{K-2},Z_{K-1},Z_K)\in \mathcal{N}_2$, the projection to the first component $\vec{Z}_{K-2}$ is contained in a generic $(K-2)$-fiber and in $\Wk{K-2}$. The first part of the claim follows directly from the fact that $\PKMMM$ and $\PKMM$ don't vanish simultaneously, by the recursive formula (\ref{RecursiveFormula}). For the second part of the claim, recall that $\vec{Z}_K\in \Wk{K}$ by definition of $\mathcal{N}_2$. Assume that $\vec{Z}_{K-2}\not\in \Wk{K-2}$ and observe $$\lceil \tfrac{K-3}{2}\rceil = \begin{cases} k-1 & \text{if } K=2k+1\\ k-1 & \text{if } K=2k.
\end{cases}$$
Then there holds $Z_1e_n=Z_3e_n = ... =Z_{2k-3}e_n=0$ and by Lemma \ref{lemma:propertyOne}, we conclude $\Phi_{2k-2}=e_{2n}$. In the case $K=2k$, this contradicts $0=\PKMM = P_s^{2k-2}=e_n$. And in the case $K=2k+1$, this leads to $$0=\PKMM=P_f^{2k-1} = Z_{2k-1}e_n,$$ which contradicts $\vec{Z}_K\in \Wk{K}$. This proves the claim. We are now able to apply the induction hypothesis (the \spanning\ for $K-2$) to $\vec{Z}_{K-2}$ and assume without loss of generality that $\tilde{P}^{K-4}_1\cdots \tilde{P}^{K-4}_n\neq 0$. 

For the second step, observe that for each point in $\mathcal{N}_2$, there is an index $1\leq j\leq n$ with $\PKMMM_{j}\neq 0$, since $\PKMMM$ and $\PKMM$ don't vanish simultaneously. So this is especially true for points $\vec{Z}_K$ in the stratum $\mathcal{B}_n$. Consider the tupel $x=(z_{K-3,ii},z_{K-2,j1},...,z_{K-2,jn})$, $i\neq j$, of (\ref{typeTwo}). The corresponding vector field $\partial_x^{K-1}$ is of the form
$$ \alpha\Deriv{z}{K-3,ii}+\sum_{r=1}^n \beta_r \Deriv{z}{K-2,jr}$$
for some suitable holomorphic functions $\alpha,\beta_1,...,\beta_n$. By construction there holds ${\partial_x^{K-1}(\tilde{P}^{K-1})=0}$. Furthermore, there holds $\tilde{P}^{K-1}=\tilde{P}^{K-3}+Z_{K-1}\tilde{P}^{K-2}$ by the recursive formula (\ref{RecursiveFormulaGeneralized}). Hence
$$ 0 = \partial_x^{K-1}(\tilde{P}^{K-1}) = \partial_x^{K-1}(\tilde{P}^{K-3}) + Z_{K-1}\partial_x^{K-1}(\tilde{P}^{K-2}),$$ 
which implies that if $\PKMM$ is in the kernel of $\partial_x^{K-1}$, then $\tilde{P}^{K-3}$ is in that kernel too. 

Observe that $\tilde{P}^{K-4}$ and $\tilde{P}^{K-5}$ are in the kernel of $\partial_x^{K-1}$, since they don't depend on the matrices $Z_{K-3}$ and $Z_{K-2}$. Therefore
$$ \partial_x^{K-1}(\tilde{P}^{K-3}) = \partial_x^{K-1}(Z_{K-3})\tilde{P}^{K-4} = \alpha \tilde{E}_{ii}\tilde{P}^{K-4} = \alpha \tilde{P}^{K-4}_i,$$
and by the previous paragraph, we may assume $\tilde{P}^{K-4}_i\neq 0$. It remains to compute $\alpha$. First, let's compute the derivatives
$$ \Deriv{z}{K-2,jr} \tilde{P}^{K-1} =Z_{K-1}\Deriv{z}{K-2,jr}\tilde{P}^{K-2}= Z_{K-1}\tilde{E}_{jr}\tilde{P}^{K-3}, \quad 1\leq r\leq n.$$
This gives us
\begin{align*}
\alpha &= \det(Z_{K-1}\tilde{E}_{j1}\tilde{P}^{K-3},...,Z_{K-1}\tilde{E}_{jn}\tilde{P}^{K-3}) \\
&= \det(Z_{K-1})\det(\tilde{E}_{j1}\tilde{P}^{K-3},...,\tilde{E}_{jn}\tilde{P}^{K-3})\\
&=\det(Z_{K-1})\det\begin{pmatrix} \PKMMM_j&&&&\\ & \ddots &&& \\ \PKMMM_1&\cdots & \PKMMM_j & \cdots & \PKMMM_n\\ &&&\ddots & \\ &&&&\PKMMM_j
\end{pmatrix}\\
&= (\tilde{P}^{K-3}_j)^n\det(Z_{K-1})\neq 0.
\end{align*}
Hence $\partial_x^{K-1}(\tilde{P}^{K-3})\neq 0$ and therefore also $\partial_x^{K-1}(\tilde{P}^{K-2})\neq 0$. We conclude that there is a finite $A\subset \Qk{K}$ such that
$$\fiber \cap\mathcal{B}_n \subset C_A(\fiber\cap \mathcal{N}_1)\subset C_A(\fiber\cap \mathcal{U}_K)$$
for each generic fiber $\fiber$.

In a third step, let's assume $1\leq k\leq n-1$ and we divide the stratum $\mathcal{B}_k$ into two more strata. Define
$$ \mathcal{C}_k:=\{\vec{Z}_K\in \mathcal{B}_k: Z_{K-1}e_i\neq 0\Rightarrow \PKMMM_i=0, \forall 1\leq i\leq n\}$$
and consider a point $\vec{Z}_K\in \mathcal{B}_k\setminus \mathcal{C}_k$. By definition, there is an index $1\leq j^*\leq n$ such that $Z_{K-1}e_{j^*}\neq 0$ and $\PKMMM_{j^*}\neq 0$. Moreover, also by definition of the stratum $\mathcal{B}_k$, we assume that the matrix $Z_{K-1}$ has rank $k$. By the complementary bases theorem, we can choose complementary indices $\{i_1,...,i_{n-k}\}\dot{\cup} \{j_1,...,j_k\}=\{1,...,n\}$ such that $j^*\in \{j_1,...,j_k\}$ and such that the vectors $Z_{K-1}e_{j_1},...,Z_{K-1}e_{j_k}$ form a basis of the image $\mathrm{Im}(Z_{K-1})$. Choose $i,i^*\in \{i_1,...,i_{n-k}\}$, then the tupel $x=(z_{K-3,i^*i_1},...,z_{K-3,i^*i_{n-k}},z_{K-2,j^*j_1},...,z_{K-2,j^*j_k},z_{K-1,ii})$ is of (\ref{typeSeven}). Since $\Deriv{z}{K-1,ii}\PKM = \tilde{E}_{ii}\PKMM =0$ on $\mathcal{N}_2$, the vector field $\partial_x^{K-1}$ is given $\pm\det(B)\Deriv{z}{K-1,ii}$ on $\mathcal{N}_2$, where 
\begin{align*}
B&= \begin{pmatrix}
\Deriv{z}{K-3,i^*i_1}\PKM &\cdots & \Deriv{z}{K-3,i^*i_{n-k}}\PKM &\Deriv{z}{K-2,j^*j_1}\PKM & \cdots & \Deriv{z}{K-2,j^*j_k}\PKM \end{pmatrix} \\
&= \begin{pmatrix}
(I_n+Z_{K-1}Z_{K-2})e_{i_1} & \cdots & (I_n+Z_{K-1}Z_{K-2})e_{i_{n-k}} & Z_{K-1}\tilde{E}_{j^*j_1}\PKMMM & \cdots & Z_{K-1}\tilde{E}_{j^*j_k}\PKMMM 
\end{pmatrix}.
\end{align*}
A similar argument as in Lemma \ref{lemma:nonGeneric3fiberInClosure} implies that $B$ is regular in $\vec{Z}_K$. Moreover, we apply the same strategy from the mentioned lemma, to conclude that the flow of $\partial_x^{K-1}$ through $\vec{Z}_K$ leaves the set $\mathcal{A}_k$ and intersects the stratum $\mathcal{B}_{k+1}$. Hence there exists a finite collection $A\subset \Qk{K}$ such that
$$ \fiber \cap \mathcal{B}_k\setminus \mathcal{C}_k \subset C_A(\fiber\cap \mathcal{B}_{k+1})$$
for each generic fiber $\fiber$. 

Next, consider a point $\vec{Z}_K\in \mathcal{C}_k$. Choose again a complementary set of indices\newline $\{i_1,...,i_{n-k}\}\dot{\cup}\{j_1,...,j_k\}=\{1,...,n\}$ such that the vectors $Z_{K-1}e_{j_1},...,Z_{K-1}e_{j_k}$ form a basis of the image $\mathrm{Im}(Z_{K-1})$. By definition of $\mathcal{C}_k$, there holds $\PKMMM_{j_1}=...=\PKMMM_{j_k}=0$. Recall that there is an index $1\leq i^*\leq n$ with $\PKMMM_{i^*}\neq 0$, since $\PKMMM$ and $\PKMM$ don't vanish simultaneously. The vector fields $\gamma^{K-2}_{jj,i^*}, j\in \{j_1,...,j_k\}$, from Lemma \ref{lemma:fiberPreservingVectorFields}, are given by
$$ \gamma^{K-2}_{jj,i^*} = (\PKMMM_{i^*})^2 \Deriv{z}{K-2,jj}, \quad j\in \{j_1,...,j_k\},$$
in $\vec{Z}_K$. Now consider the tupel $x=(z_{K-3,n1},...,z_{K-3,nn},z_{K-1,i^*i^*})$ of (\ref{typeFive}). Its corresponding vector field $\partial_x^{K-1}$ is of the form $\pm(\tilde{P}^{K-4}_n)^n\det(I_n+Z_{K-1}Z_{K-2})\Deriv{z}{K-1,i^*i^*}$ in $\vec{Z}_K$. As in Lemma \ref{lemma:nonGeneric3fiberInClosure}, we have
$$ \gamma^{K-2}_{j_kj_k,i^*}\circ \cdots \circ \gamma^{K-2}_{j_1j_1,i^*}\circ \partial_x^{K-1}(z_{K-1,i^*i^*})\neq 0.$$
This implies the existence of a finite collection $A\subset \Qk{K}$ such that
$$ \fiber\cap \mathcal{C}_k\subset C_A(\fiber\cap \mathcal{B}_k\setminus \mathcal{C}_k)$$
for each generic fiber $\fiber$. In summary, this leads to
$$ \fiber\cap \mathcal{B}_k\subset C_A(\fiber\cap \mathcal{B}_{k+1})$$
for each generic fiber $\fiber$ and $1\leq k\leq n-1$.

It remains to consider a point $\vec{Z}_K$ in the stratum $\mathcal{B}_0$, that is, we assume $Z_{K-1}=0$. Choose again the vector field $\partial_x^{K-1}$ corresponding to the tupel $x=(z_{K-3,n1},...,z_{K-3,nn},z_{K-1,i^*i^*})$ of (\ref{typeFive}). Then
$$ \partial_x^{K-1}(z_{K-1,i^*i^*}) = \pm (\tilde{P}^{K-4}_n)^n\underbrace{\det(I_n+Z_{K-1}Z_{K-2})}_{=1} \neq 0,$$
and therefore there is a finite collection $A\subset \Qk{K}$ such that
$$ \fiber \cap \mathcal{B}_0\subset C_A(\fiber\cap \mathcal{B}_1)$$
for each generic fiber $\fiber$.

For the final step, we only need to put everything together. We have
$$ C_A(\fiber \cap \mathcal{B}_0)\subset C_A(\fiber\cap \mathcal{B}_1) \subset \cdots \subset C_A(\fiber \cap \mathcal{B}_n)\subset C_A(\fiber \cap \mathcal{U}_K)$$
for each generic fiber $\fiber$, and since $\mathcal{N}_2=\bigcup_{k=0}^n \mathcal{B}_k$, we obtain 
$$\fiber \cap \mathcal{N}_2\subset C_A(\fiber \cap \mathcal{U}_K)$$
for each generic fiber $\fiber$. This finishes the proof.
\end{proof}
%
%


\begin{thebibliography}{99}
  
 \bibitem[Co66]{Cohn:1966tz} Cohn, P. M.,
\emph{On the structure of the $GL_2$ of a ring.} 
Inst. Hautes \'Etudes Sci. Publ. Math. No. 30 (1966), 5–-53.

\bibitem[DJ06]{Dopico:2006vy} Dopico, F. M.; Johnson, C. R., \emph{Complementary bases in symplectic matrices and a proof that their determinant is one.} Linear Algebra and Its Applications {\bf 419} (2006), 772-778.

\bibitem[For10]{Forstneric:2010aa}
Forstneri{\v{c}}, F., \emph{The {O}ka principle for sections of stratified fiber bundles.}, Pure Appl. Math. Q. \textbf{6} (2010), 843--874.

\bibitem[For17]{Forstneric:2011wk}
Forstneri{\v{c}}, F., \emph{Stein manifolds and holomorphic mappings.}, Second edition, Springer-Verlag, 2017

\bibitem[Gro89]{Gromov:1989aa}
Gromov, M., \emph{Oka's principle for holomorphic sections of elliptic bundles.}, J. Amer. Math. Soc. \textbf{2} (1989), 851--897.

\bibitem[GMV91]{Grunewald:1991ww}
Grunewald, F.; Mennicke, J.; Vaserstein, L., \emph{On symplectic groups over polynomial rings.}, Math. Z. \textbf{206} (1991), 35--56.

\bibitem[IK12]{Ivarsson:2012aa}
Ivarsson, B.; Kutzschebauch, F., \emph{Holomorphic factorization of
  mappings into $\mbox{SL}_n(\mathbb{C})$.}, Ann. of Math. (2)
\textbf{175}
  (2012), 45--69.
  
\bibitem[IKL20a]{Bjorn-Ivarson:2020aa}
Ivarsson, B.; Kutzschebauch, F.; L{\o}w, E., \emph{Factorization of symplectic matrices into elementary factors.}  
Proc. Amer. Math. Soc. {\bf 148} (2020), no. 5, 1963–-1970.

\bibitem[IKL20b]{Ivarson:2020aa} Ivarsson, B.; Kutzschebauch, F.; L{\o}w, E., \emph{Holomorphic Factorization of Mappings into $\mbox{Sp}_4(\mathbb{C})$.} (2020) arXiv:2005.07454.

\bibitem[Kop78]{Kopeiko:1978ub} Kope\u{\i}ko, V. I., \emph{Stabilization of symplectic groups over a ring of polynomials.} (Russian) Mat. Sb. (N.S.) 106(148) (1978), no. 1, 94–-107, 144.

\bibitem[Kus21]{Kusakabe:2021tv}
Kusakabe, Y., \emph{Elliptic characterization and unification of Oka maps}, Math. Z. \textbf{298} (2021), 1735--1750. 

\bibitem[Qu76]{Quillen:1976vz} Quillen, D., \emph{Projective modules over polynomial rings.} Inventiones math. {\bf 36} (1976), 167--171.
  
\bibitem[Sus77]{Suslin:1977us}
Suslin, A.\ A.\,\emph{The structure of the special linear group over rings of polynomials.} {\it Izv.\ Akad.\ Nauk SSSR Ser.\ Mat.\ }{\bf 41} (1977), no.\ 2, 235--252, 477, (English translation, {\it Math.\ USSR Izv.\ }{\bf 11} (1977), 221--238.) 
 
 \bibitem[Vas88]{Vaserstein:1988td}
Vaserstein, L., \emph{ Reduction of a matrix depending on parameters to a diagonal form by addition operations.}  Proc.\ Amer.\ Math.\ Soc.\ {\bf 103} (1988), no.\ 3, 741--746.

\bibitem[VS13]{Vavilov:2013tq}
Vavilov, N. A.; Stepanov, A. V., \emph{Linear groups over general rings. I. Generalities.} (Russian) Zap. Nauchn. Sem. S.-Peterburg. Otdel. Mat. Inst. Steklov. (POMI) 394 (2011), Voprosy Teorii Predstavleniĭ Algebr i Grupp. 22, 33–-139, 295; (English translation in J. Math. Sci. (N.Y.) {\bf 188} (2013), no. 5, 490–-550).  
\end{thebibliography}
\end{document}